\documentclass[a4paper,reqno]{amsart}
\usepackage{amsfonts}

\usepackage{caption}

\usepackage{hyperref}

\usepackage[textwidth=408pt,textheight=658pt,heightrounded, headsep=12pt,vmarginratio=1:1]{geometry}
\usepackage[T1]{fontenc} 
\usepackage[utf8]{inputenc}
\usepackage[french, italian, english]{babel}
\usepackage{amsmath,amssymb,amsthm,mathrsfs,esint}
\usepackage{mathtools}

\newcommand{\R}{{\mathbb R}}
\newcommand{\N}{{\mathbb N}}
\newcommand{\Z}{{\mathbb Z}}
\newcommand{\km}{{k^{-1}}}

\newcommand{\tq}{{\tilde{q}}}
\newcommand{\qz}{{q_z^k}}
\newcommand{\tqz}{{\tilde{q}_z^k}}
\newcommand{\Qz}{{Q_z^k}}

\newcommand{\Rn}{{\R}^n}
\newcommand{\dx}{\, \mathrm{d} x}
\renewcommand{\dh}{\, \mathrm{d} \mathcal{H}^{n-1}}
\newcommand{\hn}{\mathcal{H}^{n-1}}
\newcommand{\Sn}{{\mathbb{S}^{n-1}}}
\newcommand{\ol}{\overline}
\newcommand{\sm}{\setminus}

\newcommand{\tr}{\mathrm{tr}\,}
\newcommand{\xy}{{^\xi_y}}
\newcommand{\Mnn}{{\mathbb{M}^{n\times n}_{sym}}}

\newcommand{\dod}{{\partial_D \Omega}}
\newcommand{\don}{{\partial_N \Omega}}
\newcommand{\dom}{{\partial \Omega}}
\newcommand{\weak}{\rightharpoonup}

\DeclareMathOperator*{\aplim}{ap\,lim}

\theoremstyle{plain}
\begingroup
\theoremstyle{plain}
\newtheorem{theorem}{Theorem}[section]

\newtheorem{proposition}[theorem]{Proposition}
\newtheorem{lemma}[theorem]{Lemma}
\theoremstyle{definition}
\newtheorem{definition}[theorem]{Definition}
\theoremstyle{remark}
\newtheorem{remark}[theorem]{Remark}

\endgroup

\numberwithin{equation}{section}

\title[A density result in $GSBD^p$ and applications to brittle fractures]{A density result in $GSBD^p$ with applications to the approximation of brittle fracture energies}
\author{Antonin Chambolle \and Vito Crismale}
\address{CMAP, \'Ecole Polytechnique, 91128 Palaiseau Cedex, France}
\email[Antonin Chambolle]{antonin.chambolle@cmap.polytechnique.fr}
\email[Vito Crismale]{vito.crismale@polytechnique.edu}

\begin{document}
\begin{abstract}
We prove that any function in $GSBD^p(\Omega)$, with $\Omega$ a $n$-dimensional open bounded set with finite perimeter, is approximated by functions $u_k\in SBV(\Omega;\Rn)\cap L^\infty(\Omega;\Rn)$ whose jump is a finite union of $C^1$ hypersurfaces.
The approximation takes place in the sense of Griffith-type energies $\int_\Omega W(e(u)) \dx +\hn(J_u)$, $e(u)$ and $J_u$ being the approximate symmetric gradient and the jump set of $u$, and $W$ a nonnegative function with $p$-growth, $p>1$.
The difference between $u_k$ and $u$ is small in $L^p$ outside a sequence of sets $E_k\subset \Omega$ whose measure tends to 0 and if $|u|^r \in L^1(\Omega)$ with $r\in (0,p]$, then $|u_k-u|^r \to 0$ in $L^1(\Omega)$.
Moreover,
an approximation property for the (truncation of the) amplitude of the jump holds.
We apply the density result to deduce $\Gamma$-convergence approximation \emph{à la} Ambrosio-Tortorelli for Griffith-type energies with either Dirichlet boundary condition or a mild fidelity term, such that minimisers are \emph{a priori} not even in $L^1(\Omega;\Rn)$.
\end{abstract}
\maketitle

{\small
\keywords{\textbf{Keywords:}  generalised special functions of bounded deformation, strong approximation, brittle fracture, $\Gamma$-convergence, free discontinuity problems 

\bigskip
\subjclass{\textbf{MSC 2010:} 
49Q20,     
74R10,     
26A45,  	
49J45,  	   
74G65.  
}}

\tableofcontents

\section{Introduction}

A fundamental idea in the variational approach to fracture mechanics is that the formation of fracture is the result of the competition between the surface energy spent to produce the crack and the energy stored in the uncracked region. 
This idea dates back to the pioneering work of Griffith \cite{Griffith} and is the core of the model for quasistatic crack evolution proposed by Francfort and Marigo \cite{FraMar98}, which, in turn, is 
the starting point for a large number of variational models 
(see e.g.\ \cite{DMToa02, FraLar03, Cha03, BabGia14, FriSol16} and \cite{DMFraToa07, DMLaz10, Laz11} for brittle fracture in the small and finite strain framework, respectively, and e.g.\ \cite{CagToa11, DMZan07, CriLazOrl17} for cohesive fracture). 
For brittle fracture models, in small strain assumptions,    
the sum of the 
bulk energy and of the surface energy (that in brittle fracture is nothing but the measure of the crack) 
has usually the form
\begin{equation}\tag{G}\label{1207171129}
\int \limits_\Omega W(e(u)) \dx + \hn(J_u)\,
\end{equation}
in a \emph{reference configuration} $\Omega\subset \Rn$.
This depends on
the \emph{displacement} $u\colon \Omega \to \Rn$
through $e(u)$, the \emph{symmetric approximate gradient} of $u$, and $J_u$, the \emph{jump set} of $u$, that represents the crack set. 
 In order to give sense to \eqref{1207171129}, one assumes that $u$ admits a measurable  (with respect to the Lebesgue measure $\mathcal{L}^n$) \emph{symmetric approximate gradient} $e(u)(x) \in \Mnn$ for $\mathcal{L}^n$-a.e.\ $x\in \Omega$, characterised by
\begin{equation*}
\aplim\limits_{y\to x} \frac{\big(u(y)-u(x)-e(u)(x)(y-x)\big)\cdot (y-x)}{|y-x|^2}=0\,,
\end{equation*}
(see \eqref{def:aplim} for definition of approximate limit)
and that 
$J_u$  is countably $(\hn, n{-}1)$ rectifiable, where $J_u$ is defined as the set of discontinuity points $x$ where $u$ has one-sided approximate limits $u^+(x)\neq u^-(x)$ with respect to a suitable direction $\nu_u(x)$ normal to $J_u$.
The function $W$ is required to be convex with $p$-growth, with $p>1$ (cf.\ e.g.\ \cite[Section~2]{FraMar98} in the framework of elastic bulk energies, and \cite[Sections~10 and 11]{Hut} and references therein for a connection with elasto-plastic materials).

The space $BD(\Omega)$ of \emph{functions of bounded deformation} is an important example of function space in which \eqref{1207171129} is well defined. 
Employed in the mathematical modelling of small strain elasto-plasticity (see e.g.\ \cite{TemStr, Suquet1981, KohnTemam, Tem}) it consists of the functions $u\in L^1(\Omega;\Rn)$ whose symmetric distributional derivative $(\mathrm{E}u)_{ij}:=\frac{1}{2}( \mathrm{D}_i u_j + \mathrm{D}_j u_i)$ is a (matrix-valued) measure with finite total variation in $\Omega$. In particular (see for instance \cite{AmbCosDM97}), $J_u$ is countably $(\hn, n{-}1)$ rectifiable and
$\mathrm{E}u=\mathrm{E}^a u+ \mathrm{E}^c u+\mathrm{E}^j u$,
where $\mathrm{E}^a u = e(u) \mathcal{L}^n$, the \emph{Cantor part} $\mathrm{E}^c$ is singular with respect to $\mathcal{L}^n$ and vanishes on Borel sets of finite $\hn$ measure, and $\mathrm{E}^j u$ is concentrated on $J_u$.

In view of the assumptions on \eqref{1207171129}, and since in particular it gives no control on the Cantor part of $\mathrm{E}u$, in the present context it is useful to focus on the space $SBD(\Omega)$ of $BD$ functions with null Cantor part, introduced in \cite{AmbCosDM97}, and on its subspace
\begin{equation*}
SBD^p(\Omega):=\{u\in SBD(\Omega)\colon e(u)\in L^p(\Omega;\Mnn),\, \hn(J_u)<\infty\}\,.
\end{equation*}
Indeed, the existence of minimisers for \eqref{1207171129} is guaranteed in $SBD^p(\Omega)$ by the compactness result \cite[Theorem~1.1]{BelCosDM98}, provided one has an \emph{a priori} bound for $u$ in $L^\infty(\Omega;\Rn)$. Unfortunately, it is hard to obtain such a bound, even if the total energy includes additional lower order terms.

To overcome this drawback, Dal Maso introduced in \cite{DM13} the spaces $GBD$ and $GSBD$ of the \emph{generalised} $BD$ and $SBD$ functions, respectively (see Definition~\ref{def:GBD} for its definition, based on properties of one-dimensional slices). Every $GBD$ function admits a measurable symmetric approximate gradient and has a countably $(\hn,n{-}1)$ rectifiable jump set, so that \eqref{1207171129} makes sense. Moreover, the compactness result \cite[Theorem~11.3]{DM13} requires a very mild control for sequences in $GSBD^p$ (the space of $GSBD$ functions with $e(u)$ $p$-integrable and $\hn(J_u)$ finite), namely that $\psi_0(|u_k|)$ is bounded in $L^1$ for some $\psi_0$ nonnegative, continuous, increasing and unbounded. 
This gives compactness with respect to the convergence in measure of 
minimising sequences for total energies with main term \eqref{1207171129} plus a lower order \emph{fidelity term} of type $\int_\Omega \psi_0(|u-g|)\dx$, for a suitable datum $g$, so that the displacements are not even forced to be  in $L^1$. 

Notice that, differently from the case of image reconstruction, 
a fidelity term in the total energy is not in general meaningful in fracture mechanics. In particular, the original formulation in \cite[Section~2]{FraMar98} considers the energy \eqref{1207171129} only supplemented with a Dirichlet boundary condition. We remark that a Mumford-Shah-type energy, obtained from the \emph{Mumford-Shah image segmentation functional} \cite{MumSha, DeGCarLea} by replacing the $L^2$ fidelity term with a Dirichlet boundary condition, describes brittle fractures in the \emph{generalised antiplane} setting of e.g.\ \cite{FraLar03}.

An interesting 
issue is to provide $\Gamma$-convergence approximations, in the spirit of Ambrosio and Tortorelli \cite{AmbTorCPAM, AmbTorUMI}, for energies of the form \eqref{1207171129} plus some compliance conditions on the displacement. In \cite{AmbTorCPAM, AmbTorUMI} the Mumford-Shah functional is approximated by means of elliptic functionals, depending on the displacement and on a so-called \emph{phase field} variable, whose minimisers are easier to compute.
This result has been largely employed to numerically handle problems 
both in 
image reconstruction and in fracture mechanics (see for instance \cite{BouFraMar00, Bou07, BurOrtSul13}). In the vector-valued case, approximations \emph{à la} Ambrosio-Tortorelli have been proven by Chambolle \cite{Cha04, Cha05Add} and Iurlano \cite{Iur14} for the restriction of \eqref{1207171129} (assuming $W$ quadratic) to $SBD^2(\Omega)\cap L^2(\Omega;\Rn)$ and $GSBD^2(\Omega)\cap L^2(\Omega;\Rn)$, respectively. A crucial point in the proof of the $\Gamma$-limsup inequality is to approximate, in the sense of (limit) energy, any displacement 
by a sequence of functions in $SBV(\Omega;\Rn) \cap L^\infty(\Omega;\Rn)$ whose jump is a finite union of $C^1$ hypersurfaces.  For these displacements  it is not difficult to find a recovery sequence, 
and one concludes by a diagonal argument. 

In the scalar setting, the first results in
this direction are found in~\cite{BraChP96,DibSer97}.
Furthermore, by \cite[Theorem~3.1]{CorToa99} (see also \cite{AmarDeCicco})  one may consider approximating functions whose jump is \emph{essentially closed} and \emph{polyhedral}, which are of class $W^{m,\infty}$, for every $m\in \N$, in the complement of the (closure of the) jump. 

We should recall here that the (polyhedral) approximation of rectifiable
 sets of codimension one (such as jumps sets), or of more general rectifiable
currents of any dimension, is an old and 
important issue in Geometric Measure Theory~\cite{Fed, White} (see also the recent \cite{BCG17} for partitions). However the constructions in these
works are not particularly
adapted to the approximation of a whole function, nor its trace on the jump.

The techniques we use here for approximating the jump set are relatively
standard and mostly derived from~\cite{Cha04,BraChP96,Iur14}
(see also~\cite{CFI17Density} for a recent and simpler variant of~\cite{Cha04,Iur14}
based on a finite-elements discretization). 
As in~\cite{BraChP96,Iur14}, we also approximate strongly the trace;
in our case it needs a carefully built extension of the displacements
on both sides of the jump.
 A refinement for the approximation of trace is obtained in \cite{Cri19}, giving a density theorem for $SBD^p$ functions in the $BD$ norm. An analogous approximation for $SBV^p$
functions in the $BV$ norm  
is established in \cite{DPFusPra17}: this 
is based on convolutions with variable kernels, depending on the distance from a compact set containing (most of) $J_u$.

The most difficult part in  the present  setting
is the approximation of the bulk energy, which we provide without
any assumption on the integrability of the function itself.
It is based on a new approximation strategy
which was made possible thanks to the estimates in~\cite{CCF16}.

The only other work which also removes any integrability assumption
on the displacement is the Friedrich's work
~\cite{FriPWKorn}: there it is shown
that
for any $u\in GSBD^2(\Omega)$, with $\Omega\subset \R^2$, there exists a sequence $u_k$ in $SBV(\Omega;\R^2) \cap L^\infty(\Omega;\R^2)$ with regular jump, converging in measure to $u$ and such that $e(u_k)\to e(u)$ in $L^2(\Omega;\mathbb{M}^{2\times 2}_{sym})$ and $\hn(J_{u_k}\triangle J_u)\to 0$, $\triangle$ denoting the symmetric difference of sets (cf.\ \cite[Theorem~2.5]{FriPWKorn}). This follows from a piecewise Korn inequality, obtained with a careful analysis of the jump set of $GSBD$ functions in a $2$-dimensional setting.

 Our main result is the following density theorem:
\begin{theorem}\label{teo:main}
Let $\Omega\subset \Rn$ be a bounded open set with finite perimeter  and $\partial \Omega$ 
$(\hn,n{-}1)$-countably rectifiable, 
$ p > 1$,   
and $u\in GSBD^p(\Omega)$.   
Then there exist $u_k\in SBV^p(\Omega;\Rn)\cap L^\infty(\Omega; \Rn)$ and $E_k\subset \Omega$ such that each
$J_{u_k}$ is closed in $\Omega$ and included in a finite union of closed connected pieces of $C^1$ hypersurfaces, $u_k\in W^{1,\infty}(\Omega\setminus J_{u_k}; \Rn)$, and:
\begin{subequations}\label{eqs:main}
\begin{align}
\lim_{k\to \infty}\mathcal{L}^n(E_k&)= \lim_{k\to \infty}\int\limits_{\Omega\setminus E_k} |u_k -u|^p \dx = 0\,,\label{1main}\\
 e(u_k) &\to e(u)  \quad\text{in } L^p(\Omega;\Mnn) \,, \label{2main}\\
 &\hn(J_{u_k}\triangle J_u)\to 0 \,,\label{3main}\\
\int \limits_{J_{u_k}\cup J_u} \hspace{-1em}\tau(|u_k^\pm-u^\pm|)& \dh + \int \limits_{\partial \Omega} \tau(|\tr (u_k- u)|)\dh \to 0\,,  \label{4main}
\end{align}
for $\tau\in C^1(\R)$ with $-\tfrac12 \leq \tau \leq \tfrac12$, $0\leq \tau' \leq 1$. 
In particular, $u_k$ converge to $u$ in measure in~$\Omega$.
Moreover, if $
\int\limits_\Omega \psi(|u|)\dx
$ is
finite 
for $\psi\colon [0,\infty)\to [0,\infty)$ 
increasing, 
continuous, with (for $C_\psi >0$)
\begin{equation}\tag{HP$\psi$}\label{hppsi}
\psi(0)=0\,,\qquad \psi(s+t)\leq C_\psi \big(\psi(s)+\psi(t)\big)\,,\qquad \psi(s)\leq C_\psi (1+|s|^p)\,,\qquad \lim_{s\to \infty} \psi(s)=\infty\,,
\end{equation}  
then
\begin{equation}\label{5main}
\lim_{k\to \infty} \int\limits_{\Omega}\psi(|u_k-u|)\dx =0\,.
\end{equation}
\end{subequations}
\end{theorem}

Notice that here $n$ and $p$ are general, 
 with no integrability assumptions on the displacement.

As in \cite{Cha04, Iur14, CFI17Density}, we first prove an intermediate approximation (Theorem~\ref{teo:rough}) which controls the measure of the jump set up to a multiplicative parameter. Then we cover a large part of the jump set $J_u$ by suitable cubes, that are almost split into two parts by $J_u$. This gives a partition of $\Omega$ in subsets where the jump set has small $\hn$ measure, so that Theorem~\ref{teo:rough} provides here (in suitable neighbourhoods) approximating functions close in energy to the original one. 
A fundamental difference with respect to \cite{Cha04, Iur14, CFI17Density}
is that we do not use partitions of the unity neither to extend the original function in suitable neighbourhoods of the subsets of the partition nor to glue the approximating functions constructed in any subset. 
 This is done by employing 
a reflection technique for vector-valued functions due to Nitsche \cite{Nie81} (cf.\ Lemma~\ref{le:Nitsche}) and 
allows us to avoid any assumption on the integrability of $u$.

The proof of Theorem~\ref{teo:rough} is based 
 on \cite[Proposition~3]{CCF16} (cf.\ Proposition~\ref{prop:3CCF16}), and is close to what done in \cite{CCF17} to approximate a brittle fracture energy with a non-interpenetration constraint.
The idea is to partition the domain into cubes of side $k^{-1}$ and to distinguish, at any scale, the cubes where the ratio between the perimeter and the jump of $u$ is greater than a fixed small parameter $\theta$.
In such cubes, one may replace the original function $u$ with a constant function, since on the one hand the new jump is 
less than the original jump times $\theta^{-1}$, and on the other hand the total volume of these cubes is small as the length scale goes to 0.  
In the remaining cubes, where the relative jump is small, one applies Proposition~\ref{prop:3CCF16}: a Korn-Poincaré-type inequality holds up to a set of small volume, and in this small exceptional set the original function may be replaced by a suitable affine function without perturbing much its energy,  up to a convolution with a kernel of support with size comparable to that of the cubes. 

We prove also the approximation property \eqref{4main} for the amplitude of the jump $[u](x):=u^+(x)-u^-(x)$ for $x\in J_u$ (and for the traces at the reduced boundary of $\Omega$),  which might be useful in cohesive fracture models. Notice that $[u]$ is not integrable in $J_u$ with respect to $\hn$ for a general $u\in GSBD^p(\Omega)$,
thus we consider truncations for the jumps. 
We employ a fine estimate from \cite{Bab15} for the truncation of the trace components in any direction, whose symmetric gradient is a bounded measure.  Approximations  of this type  are  also proven in \cite{Iur14}, where $u$ is in $L^2$,  and in \cite{Cri19}: these are  used in the $\Gamma$-convergence results \cite{FocIur14}  and \cite{CC19b}  for cohesive fracture energies. 
 We may as well consider in Theorem~\ref{teo:main} smooth approximating functions in the sense of the aforementioned \cite[Theorem~3.1]{CorToa99} (which applies directly if $\Omega$ is Lipschitz), with minor modifications in our proof. 

In the last part of the work we present $\Gamma$-convergence results \emph{à la} Ambrosio-Tortorelli for brittle fracture energies. First, we approximate 
\eqref{1207171129} for every $u$ which is $GSBD^p$ in an open bounded set with finite perimeter (Theorem~\ref{teo:gammaconvG}). Then we focus on the sum of \eqref{1207171129} with suitable compliance terms, which prevent that the set of minimisers coincides with the constant displacements. In particular, we consider the cases of a mild fidelity term $|u-g|^r$, with $r \in (0,p]$ (see Theorem~\ref{teo:gammaconvF}), and of a Dirichlet boundary condition on a subset $\dod$ of $\dom$, under some geometric conditions (see Theorem~\ref{teo:gammaconvD}). 

 For the standard quadratic elastic energy ($\mathbb{C}$ is the fourth-order Cauchy stress tensor), Theorem~\ref{teo:gammaconvD} gives:
\begin{theorem}\label{teo:gammaconvDintro}
Let $u_0\in H^1(\Rn;\Rn)$, and $\Omega\subset \Rn$ be an open, bounded,  
Lipschitz domain for which $\dom=\dod\cup \don\cup N$, with $\dod$ and $\don$ relatively open, $\dod \cap \don =\emptyset$, $\mathcal{H}^{n-1}(N)=0$, 
$\dod \neq \emptyset$, and $\partial(\dod)=\partial(\don)$. 
Assume that there exist $\ol \delta$ and $x_0\in \Rn$ such that
\[O_{\delta,x_0}(\dod) \subset \Omega\]
 for $\delta \in (0,\ol \delta)$, 
where $O_{\delta,x_0}(x):=x_0+(1-\delta)(x-x_0)$. Moreover let $\varepsilon_k$, $\eta_k>0$ with
$
\varepsilon_k\to 0$, $\frac{\eta_k}{\varepsilon_k}\to 0$ as $k\to \infty$.
Then, for $H^{1}_{u_0}(\Omega;\Rn):=\{u\in H^{1}(\Omega;\Rn)\colon \mathrm{tr}_\Omega\, u =\mathrm{tr}_\Omega\, u_0 \text{ on }\dod\}$ and $V_k^1:=\{v\in H^{1}(\Omega)\colon \eta_k \leq v \leq 1\,,\ \mathrm{tr}_\Omega\, v=1 \text{ on }\dod\}$, the functionals 
\begin{equation*}
 D^2_k(u,v):=\begin{dcases}
\int \limits_\Omega \Big(v\, \mathbb{C}e(u) \colon e(u)+\frac{(1-v)^2}{4\varepsilon_k}+\varepsilon_k |\nabla v|^2 \Big)\dx \quad &\text{in }H^{1}_{u_0}(\Omega;\Rn){\times} V_k^1\,,\\
+\infty &\text{otherwise,}
\end{dcases}
\end{equation*}
$\Gamma$-converge as $k\to \infty$ to
\begin{equation*}
D^2(u,v):=
\begin{dcases}
\int\limits_\Omega \mathbb{C}e(u)\colon e(u) \dx + \hn\Big(J_u \cup \big(\dod \cap \{\mathrm{tr}_\Omega \, u \neq \mathrm{tr}_\Omega\, u_0\} \big)  \Big) \, &\text{in } GSBD^p(\Omega){\times} \{v=1\},\\
+\infty &\text{otherwise,}
\end{dcases}
\end{equation*} 
with respect to the topology of the convergence in measure for $u$ and $v$.
\end{theorem}

The existence of minimisers for the (limit) Dirichlet problem has been recently shown by Friedrich and Solombrino \cite[Theorem~6.2]{FriSol16} in dimension 2, and in \cite{CC18} in general dimension  (see also \cite{Fri19}).  
 We then prove Theorem~\ref{thm:compmin}, ensuring compactness for minimisers of the approximating functionals for the Dirichlet problem.

As for the case with a mild fidelity term, in Proposition~\ref{prop:compattezza} we prove compactness for minimisers of the approximating energies, which are shown to exist for Lipschitz domain, see Remark~\ref{rem:exappmin}.

We conclude this introduction by mentioning some other problems for which density results as 
Theorem~\ref{teo:main} are useful. For instance, \cite[Theorem~2.5]{FriPWKorn} is applied in \cite{Fri17ARMA} for the derivation of linearised Griffith energies from nonlinear models, while \cite{CFI17Density} is employed in \cite{CFI17DCL} and \cite{ChaConIur17} to prove existence of minimisers for the set function that is the strong counterpart of 
\eqref{1207171129},  provided a weak solution exists.  More precisely, in \cite{CFI17DCL} the setting is 2-dimensional and $W$ may have $p$-growth for any $p>1$, while \cite{ChaConIur17} considers the 
case $\Omega \subset \Rn$ with $W$ quadratic
(in these works internal regularity is proven,  we extend this up to the boundary in \cite{CC19}). Moreover, \cite{Cha04, Cha05Add} are useful in other $\Gamma$-convergence approximations of brittle fracture energies, such as \cite{Neg03, Neg06}.

The paper is organised as follows.  In Section~\ref{Sec1} we introduce  notation, functional spaces, and some technical tools useful in the following, as the reflection property Lemma~\ref{le:Nitsche}. In Section~\ref{Sec2} and Section~\ref{Sec3} we prove the rough and the main density results, respectively. Section~\ref{Sec4} is devoted to the applications.

\section{Notation and preliminaries}\label{Sec1}
For every $x\in \Rn$ and $\varrho>0$ let $B_\varrho(x)$ be the open ball with center $x$ and radius $\varrho$. For $x$, $y\in \Rn$, we use the notation $x\cdot y$ for the scalar product and $|x|$ for the norm.
We denote by $\mathcal{L}^n$ and $\mathcal{H}^k$ the $n$-dimensional Lebesgue measure and the $k$-dimensional Hausdorff measure. For any locally compact subset $B$ of $\Rn$, the space of bounded $\R^m$-valued Radon measures on $B$ is denoted by $\mathcal{M}_b(B;\R^m)$. For $m=1$ we write $\mathcal{M}_b(B)$ for $\mathcal{M}_b(B;\R)$ and $\mathcal{M}^+_b(B)$ for the subspace of positive measures of $\mathcal{M}_b(B)$. For every $\mu \in \mathcal{M}_b(B;\R^m)$, its total variation is denoted by $|\mu|(B)$.
We denote by $\chi_E$ the indicator function of any $E\subset \R^n$, which is 1 on $E$ and 0 otherwise.
\begin{definition}\label{def:aplim}
Let $A\subset \Rn$, $v\colon A \to \R^m$ an $\mathcal{L}^n$-measurable function, $x\in \Rn$ such that
\begin{equation*}
\limsup_{\varrho\to 0^+}\frac{\mathcal{L}^n(A\cap B_\varrho(x))}{\varrho^n}>0\,.
\end{equation*}
A vector $a\in \Rn$ is the \emph{approximate limit} of $v$ as $y$ tends to $x$ if for every $\varepsilon>0$
\begin{equation*}
\lim_{\varrho \to 0^+}\frac{\mathcal{L}^n(A \cap B_\varrho(x)\cap \{|v-a|>\varepsilon\})}{\varrho^n}=0\,,
\end{equation*}
and then we write
\begin{equation}\label{3105171542}
\aplim \limits_{y\to x} v(y)=a\,.
\end{equation}
\end{definition}
\begin{remark}\label{rem:3105171601}
Let $A$, $v$, $x$, and $a$ be as in Definition~\ref{def:aplim} and let $\psi$ be a homeomorphism between $\R^m$ and a bounded open subset of $\R^m$. Then \eqref{3105171542} holds if and only if \begin{equation*}
\lim_{\varrho\to 0^+}\frac{1}{\varrho^n}\hspace{-1em}\int \limits_{A\cap B_\varrho(x)}\hspace{-1em}|\psi(v(y))-\psi(a)|\,\mathrm{d}y=0\,.
\end{equation*}
\end{remark}

\begin{definition}
Let $U\subset \Rn$ open, and $v\colon U\to \R^m$ be $\mathcal{L}^n$-measurable. The \emph{approximate jump set} $J_v$ is the set of points $x\in U$ for which there exist $a$, $b\in \R^m$, with $a \neq b$, and $\nu\in \Sn$ such that
\begin{equation*}
\aplim\limits_{(y-x)\cdot \nu>0,\, y \to x} v(y)=a\quad\text{and}\quad \aplim\limits_{(y-x)\cdot \nu<0, \, y \to x} v(y)=b\,.
\end{equation*}
The triplet $(a,b,\nu)$ is uniquely determined up to a permutation of $(a,b)$ and a change of sign of $\nu$, and is denoted by $(v^+(x), v^-(x), \nu_v(x))$. The jump of $v$ is the function defined by $[v](x):=v^+(x)-v^-(x)$ for every $x\in J_v$. Moreover, we define
\begin{equation}
J_v^1:=\{x\in J_v \colon |[v](x)|\geq 1\}\,.
\end{equation}
\end{definition}
\begin{remark}
By Remark~\ref{rem:3105171601}, $J_v$ and $J^1_v$ are Borel sets and $[v]$ is a Borel function. Moreover, by Lebesgue's differentiation theorem, it follows that $\mathcal{L}^n(J_v)=0$.
\end{remark}

\par
\medskip
\paragraph{\bf $BV$ and $BD$ functions.}
If $U\subset \Rn$ open, a function $v\in L^1(U)$ is a \emph{function of bounded variation} on $U$, and we write $v\in BV(U)$, if $\mathrm{D}_i v\in \mathcal{M}_b(U)$ for $i=1,\dots,n$, where $\mathrm{D}v=(\mathrm{D}_1 v,\dots, \mathrm{D}_n v)$ is its distributional gradient. A vector-valued function $v\colon U\to \R^m$ is $BV(U;\R^m)$ if $v_j\in BV(U)$ for every $j=1,\dots, m$.
The space $BV_{\mathrm{loc}}(U)$ is the space of $v\in L^1_{\mathrm{loc}}(U)$ such that $\mathrm{D}_i v\in \mathcal{M}_b(U)$ for $i=1,\dots,n$. 

A $\mathcal{L}^n$-measurable bounded set $E\subset \R^n$ is a set of \emph{finite perimeter} if $\chi_E$ is a function of bounded variation. The \emph{reduced boundary} of $E$, denoted by $\partial^*E$, is the set of points $x\in \mathrm{supp}\, |\mathrm{D}\chi_E|$ such that the limit $\nu_E(x):=\lim_{\varrho \to 0^+}\frac{\mathrm{D}\chi_E(B_\varrho(x))}{|\mathrm{D}\chi_E|(B_\varrho(x))}$ exists and satisfies $|\nu_E(x)|=1$. The reduced boundary is countably $(\hn, n{-}1)$ rectifiable, and the function $\nu_E$ is called \emph{generalised inner normal} to $E$.

A function $v\in L^1(U;\Rn)$ belongs to the space of \emph{functions of bounded deformation} if its distributional symmetric gradient $\mathrm{E}v$ belongs to $\mathcal{M}_b(U;\Mnn)$.
It is well known (see \cite{AmbCosDM97, Tem}) that for $v\in BD(U)$, $J_v$ is countably $(\hn, n{-}1)$ rectifiable, and that
\begin{equation*}
\mathrm{E}v=\mathrm{E}^a v+ \mathrm{E}^c v + \mathrm{E}^j v\,,
\end{equation*}
where $\mathrm{E}^a v$ is absolutely continuous with respect to $\mathcal{L}^n$, $\mathrm{E}^c v$ is singular with respect to $\mathcal{L}^n$ and such that $|\mathrm{E}^c v|(B)=0$ if $\hn(B)<\infty$, while $\mathrm{E}^j v$ is concentrated on $J_v$. The density of $\mathrm{E}^a v$ with respect to $\mathcal{L}^n$ is denoted by $e(v)$, and we have that (see \cite[Theorem~4.3]{AmbCosDM97} and recall \eqref{def:aplim}) for $\mathcal{L}^n$-a.e.\ $x\in U$
\begin{equation}\label{3105171931}
\aplim\limits_{y\to x} \frac{\big(v(y)-v(x)-e(v)(x)(y-x)\big)\cdot (y-x)}{|y-x|^2}=0\,.
\end{equation}
The space $SBD(U)$ is the subspace of all functions $v\in BD(U)$ such that $\mathrm{E}^c v=0$, while for $p\in (1,\infty)$
\begin{equation*}
SBD^p(U):=\{v\in SBD(U)\colon e(v)\in L^p(U;\Mnn),\, \hn(J_v)<\infty\}\,.
\end{equation*}
Analogous properties hold for $BV$, as the countable rectifiability of the jump set and the decomposition of $\mathrm{D}v$, and the spaces $SBV(U;\R^m)$ and $SBV^p(U;\R^m)$ are defined similarly, with $\nabla v$, the density of $\mathrm{D}^a v$, in place of $e(v)$.
For more details on $BV$, $SBV$ and $BD$, $SBD$ functions, we refer to \cite{AFP} and to \cite{AmbCosDM97, BelCosDM98, Bab15, Tem}, respectively.
\par
\medskip
\paragraph{\bf $GBD$ functions.}
We now recall the definition and the main properties of the space $GBD$ of \emph{generalised functions of bounded deformation}, introduced in \cite{DM13}, referring to that paper for a general treatment and more details. Since the definition of $GBD$ is given by slicing (differently from the definition of $GBV$, cf.\ \cite{DeGioAmb88GBV, Amb90GSBV}), we introduce before some notation for slicing.

Fixed $\xi \in \Sn:=\{\xi \in \Rn\colon |\xi|=1\}$, for any $y\in \Rn$ and $B\subset \Rn$ let
\begin{equation*}
\Pi^\xi:=\{y\in \Rn\colon y\cdot \xi=0\},\qquad B^\xi_y:=\{t\in \R\colon y+t\xi \in B\}\,,
\end{equation*}
and for every function $v\colon B\to \R^n$ and $t\in B^\xi_y$ let
\begin{equation*}
v^\xi_y(t):=v(y+t\xi),\qquad \widehat{v}^\xi_y(t):=v^\xi_y(t)\cdot \xi\,.
\end{equation*}

\begin{definition}\label{def:GBD}
Let $\Omega\subset \Rn$ be bounded and open, and $u\colon \Omega\to \Rn$ be $\mathcal{L}^n$-measurable. Then $u\in GBD(\Omega)$ if there exists $\lambda_u\in \mathcal{M}^+_b(\Omega)$ such that one of the following equivalent conditions holds true for every $\xi \in \Sn$:
\begin{itemize}
\item[(a)] for every $\tau \in C^1(\R)$ with $-\tfrac{1}{2}\leq \tau \leq \tfrac{1}{2}$ and $0\leq \tau'\leq 1$, the partial derivative $\mathrm{D}_\xi\big(\tau(u\cdot \xi)\big)=\mathrm{D}\big(\tau(u\cdot \xi)\big)\cdot \xi$ belongs to $\mathcal{M}_b(\Omega)$, and for every Borel set $B\subset \Omega$ 
\begin{equation*}
\big|\mathrm{D}_\xi\big(\tau(u\cdot \xi)\big)\big|(B)\leq \lambda_u(B);
\end{equation*}
\item[(b)] $\widehat{u}^\xi_y \in BV_{\mathrm{loc}}(\Omega^\xi_y)$ for $\hn$-a.e.\ $y\in \Pi^\xi$, and for every Borel set $B\subset \Omega$ 
\begin{equation}\label{3105171445}
\int \limits_{\Pi_\xi} \Big(\big|\mathrm{D} {\widehat{u}}_y^\xi\big|\big(B^\xi_y\setminus J^1_{{\widehat{u}}^\xi_y}\big)+ \mathcal{H}^0\big(B^\xi_y\cap J^1_{{\widehat{u}}^\xi_y}\big)\Big)\dh(y)\leq \lambda_u(B)\,.
\end{equation}
\end{itemize} 
The function $u$ belongs to $GSBD(\Omega)$ if moreover $\widehat{u}^\xi_y \in SBV_{\mathrm{loc}}(\Omega^\xi_y)$ for every $\xi \in \Sn$ and for $\hn$-a.e.\ $y\in \Pi^\xi$.
\end{definition}
$GBD(\Omega)$ and $GSBD(\Omega)$ are vector spaces, as stated in \cite[Remark~4.6]{DM13}, and one has the inclusions $BD(\Omega)\subset GBD(\Omega)$, $SBD(\Omega)\subset GSBD(\Omega)$, which are in general strict (see \cite[Remark~4.5 and Example~12.3]{DM13}).
For every $u\in GBD(\Omega)$
the \emph{approximate jump set} $J_u$ is still countably $(\hn,n{-}1)$-rectifiable (cf.\ \cite[Theorem~6.2]{DM13}) and can be reconstructed from the jump of the slices $\widehat{u}^\xi_y$ (\cite[Theorem~8.1]{DM13}).
Indeed, for every $C^1$ manifold $M\subset \Omega$ with unit normal $\nu$, it holds that for $\hn$-a.e.\ $x\in M$ there exist the \emph{traces} $u_M^+(x)$, $u_M^-(x)\in \Rn$ such that
\begin{equation}\label{0106172148}
\aplim \limits_{\pm(y-x)\cdot \nu(x)>0, \, y\to x} \hspace{-1em} u(y)=u_M^{\pm}(x)
\end{equation}
and they can be reconstructed from the traces of the one-dimensional slices (see \cite[Theorem~5.2]{DM13}).
\begin{remark}\label{3105171143}
The trace of $GSBD$ functions on a given $C^1$ manifold $M\subset \Omega$ is linear.
Indeed, let us fix $\varepsilon>0$, $\eta >0$, $x\in M$.
Then there exists $\ol \varrho$ such that for $0<\varrho <\ol \varrho$
\begin{equation*}
\mathcal{L}^n\big(\Omega \cap B_\varrho^+(x)\cap \{|u-u_M^+(x)|> \varepsilon/2\}\big),\,\mathcal{L}^n\big(\Omega \cap B_\varrho^+(x)\cap \{|v-v_M^+(x)|> \varepsilon/2\}\big)<\eta/2\,,
\end{equation*}
where $B_\varrho^+(x)$ is the half ball with radius $\varrho$ positively oriented with respect to $\nu(x)$ .
Therefore, for $0<\varrho <\ol \varrho$ it holds that
\begin{equation*}
\mathcal{L}^n\big(\Omega \cap B_\varrho^+(x)\cap \{|(u+v)-(u_M^++v_M^+)(x)|> \varepsilon\}\big)<\eta\,,
\end{equation*}
so that $(u+v)_M^+(x)=u_M^++v_M^+(x)$.
\end{remark}

Every $u\in GBD(\Omega)$ has an \emph{approximate symmetric gradient} $e(u)\in L^1(\Omega;\Mnn)$, characterised by \eqref{3105171931} and such that for every $\xi \in \Sn$ and $\hn$-a.e.\ $y\in\Pi^\xi$
\begin{equation}\label{3105171927}
e(u)^\xi_y \xi\cdot \xi=\nabla \widehat{u}^\xi_y \quad\mathcal{L}^1\text{-a.e.\ on }\Omega^\xi_y\,.
\end{equation}
Using this property, we observe the following.
\begin{lemma}\label{le:ossGSBD}
For any $u \in GSBD(\Omega)$ and $A\in M^{n\times n}$, with $\mathrm{det}\, A \neq 0$, the function
\begin{equation}\label{3005172356}
u_A(x):=A^T u(A x)
\end{equation}
belongs to $GSBD\big(A^{-1}(\Omega)\big)$, with 
\begin{equation}\label{3005172334}
\lambda_{u_A}(B)=\lambda_u(A(B))\,,
\end{equation}
for any $B\subset A^{-1}(\Omega)$ Borel, with $\lambda$ and $\lambda_A$ the measures in \eqref{3105171445} corresponding to $u$ and $u_A$, and
\begin{equation}\label{3005172337}
\begin{split}
\hn(J_{u_A})&= \hn(A^{-1}(J_u))  \,,\\
e(u_A(x))&=A^T e(u)(Ax) \,A\,. 
\end{split}
\end{equation}
\end{lemma}

\begin{proof}
Let us fix $\xi \in \Sn$. A straightforward computation shows that for $\hn$-a.e.\ $y\in \Pi_\xi$ and $\mathcal{L}^1$-a.e.\ $t\in (A^{-1}(\Omega))^\xi_y$ we have 
\begin{equation}\label{2705171129}
(u_A)^\xi_y(t)\cdot \xi= u_{Ay}^{A\xi}\,(t)\cdot A\xi \,.
\end{equation}
Moreover, for any $B\subset A^{-1}(\Omega)$, we have that
\begin{equation*}
B^\xi_y=(A(B))^{A\xi}_{Ay}\,.
\end{equation*}
This implies that, for any Borel set $B\subset A^{-1}(\Omega)^\xi_y$
\begin{equation}\label{3005172335}
\Big|\mathrm{D} ({\widehat{u}_A})_y^\xi\Big|\Big(B^\xi_y\setminus J^1_{({\widehat{u}_A})^\xi_y}\Big)+ \mathcal{H}^0\Big(B^\xi_y\cap J^1_{({\widehat{u}_A})^\xi_y}\Big)=(\widehat{\mu}_u)_{Ay}^{A\xi}(A(B))\,,
\end{equation}
where $(\widehat{\mu}_u)^\xi_y$ is the measure in \cite[Definition~4.8]{DM13} for $u$.
By Definition~\ref{def:GBD}, \cite[Definition~4.10, Remark~4.12]{DM13}, and \eqref{3005172335}, it follows that $u_A\in GSBD\big(A^{-1}(\Omega)\big)$ and that \eqref{3005172334} holds. 

By definition of $u_A$ and of jump set, one has that $x\in J_{u_A}$ if and only if $Ax\in J_u$, thus
\begin{equation*}
\hn(J_{u_A})=  \hn(J_u) \,.
\end{equation*}
In order to show the second condition in \eqref{3005172337}, we can use \eqref{3105171927} which allows us to reconstruct the approximate symmetric gradient from the derivatives of the slices.
Thus, by taking the derivative of \eqref{2705171129} with respect to $t$, we deduce that for any $\xi \in \Sn$
\begin{equation*}
e(u_A)(x)\,\xi\cdot \xi= e(u)(Ax)\,A\xi \cdot A\xi \,.
\end{equation*}
Being $e(u)$ and $e(u_A)$ symmetric matrices, by the Polarisation Identity we obtain that for any $\xi$, $\eta$ in $\Sn$
\begin{equation*}
e(u_A)(x)\,\xi\cdot \eta= e(u)(Ax)\,A\xi \cdot A\eta\,.
\end{equation*} 
This gives \eqref{3005172337} and completes the proof.
\end{proof}
Similarly to the $SBD$ case, we recall the definition of the space $GSBD^p$, which is the energy space for the Griffith energy with $p$-growth in the bulk with respect to $e(u)$. We have
\begin{equation*}
GSBD^p(\Omega):=\{u\in GSBD(\Omega)\colon e(u)\in L^p(\Omega;\Mnn),\, \hn(J_u)<\infty\}\,.
\end{equation*}

We now show an extension result for $GSBD^p$ functions on rectangles, basing on \cite[Lemma~1]{Nie81}. A similar result is stated in \cite[Lemma~5.2]{FriSol16}, in dimension 2 and for $p=2$, and employed in \cite[Lemma~3.4]{CFI16ARMA}, in dimension 2 and for $SBD^p$. Notice that the proof of \cite[Lemma~5.2]{FriSol16} employs the density result, in dimension 2 and for $p=2$, that we prove in the current paper in the general framework. We follow Nitsche's argument  directly for $GSBD$ functions, without using density results.
\begin{lemma}\label{le:Nitsche}
Let $R\subset \Rn$ be an open rectangle, $R'$ be the reflection of $R$ with respect to one face $F$ of $R$, and $\widehat{R}$ be the union of $R$, $R'$, and $F$. Let $v \in GSBD^p(R)$. Then $v$ may be extended by a function $\widehat{v}\in GSBD^p(\widehat{R})$ such that 
\begin{subequations}
\begin{align}
\hn(J_{\widehat{v}}&\cap F)=0\,,\label{2705170936}\\
\hn(J_{\widehat{v}})&\leq c\, \hn(J_v)\,,\label{2705170937}\\
\int\limits_{\widehat{R}} |e(\widehat{v})|^p\dx&\leq c\, \int\limits_R |e(v)|^p \dx\,,\label{2705170938}
\end{align}
\end{subequations}
for a suitable $c>0$ independent of $R$ and $v$.
\end{lemma}
\begin{proof}
It is not restrictive to assume that 
$F\subset \{(x',x_n)\in \R^{n-1}\times \R\colon x_n=0\}$  and $R \subset \{(x',x_n)\in \R^{n-1}\times \R\colon x_n<0\}$. 
Fix any $\mu$, $\nu$ such that $0<\mu <\nu <1$, and let $q:=\frac{1+\nu}{\nu-\mu}$. 
We define $v'$ on $R'$ by
\begin{equation*}
v':=q\, v_{A_\mu} + (1-q) v_{A_\nu}\,,
\end{equation*}
where $v_A$ is defined in \eqref{3005172356} and $A_{\mu}=\mathrm{diag}\,(1,\dots,1,-\mu)$, $A_{\nu}=\mathrm{diag}\,(1,\dots,1,-\nu)$, so that
\begin{equation}\label{1701182113}
\begin{split}
v'_i(x)&:= q \, v_i(A_\mu \,x) + (1-q) v_i(A_\nu \,x)\,,\qquad\text{for }i=1,\dots, n-1\,,\\
v'_n(x)&:=-\mu\,q \, v_n(A_\mu \,x) - \nu (1-q) v_n(A_\nu \,x)\,.
\end{split}
\end{equation}
Notice that
\begin{equation}\label{1701182114}
-\mu\,q-\nu(1-q)=1\,,
\end{equation}
and that $v'$ is well defined since $F$ is a horizontal hyperplane and $0<\mu <\nu <1$, so that $A_\mu(R'), \, A_\nu(R')\subset R$. Thus the extension $\widehat{v}$ is
\begin{equation*}
\widehat{v}:=
\begin{cases}
v &\quad\text{in }R\,,\\
v' &\quad\text{in }R'\,.
\end{cases}
\end{equation*}
By Lemma~\ref{le:ossGSBD} we have that  $v'\in GSBD^p(R')$, and then $\widehat{v}\in GSBD(\widehat{R})$. Indeed for any $B\subset \widehat{R}$, any $\xi \in \Sn$, and for $\hn$-a.e.\ $y\in \Pi_\xi$, we have that the slice $\widehat{(\widehat{v})}\xy$ is equal to $(\widehat{v})\xy$ in $R\xy=\widehat{R}\xy$ and to $(\widehat{v'})\xy$ in $(R')\xy= \widehat{R}\xy \setminus R\xy$, so $\widehat{(\widehat{v})}\xy \in BV(\widehat{R}\xy)$ and the condition \eqref{3105171445} holds with $\lambda_{\widehat{v}}\in \mathcal{M}_b^+(\widehat{R})$ given by
 $\lambda_{\widehat{v}}(B)=\lambda_{v}(B\cap R) + \lambda_{v'}(B\cap R') + \hn(B\cap F)$.
 In order to prove \eqref{2705170936}, notice that for $\hn$-a.e.\ $x\in F$ there exist the traces $(\widehat{v})^+_F(x)$, $(\widehat{v})^-_F(x) \in \Rn$ (cf.\ \eqref{0106172148} and recall that $F$ is an hyperplane, so in particular of class $C^1$).  By \eqref{0106172148}, and since $\nu_F=e_n$, we get
 \begin{equation*}
 (\widehat{v})^-_F(x)= \aplim \limits_{y\in R, \, y\to x}  v(y)\,,\quad\qquad (\widehat{v})^+_F(x)= \aplim \limits_{y\in R', \, y\to x}  v'(y)
 \end{equation*}
 Notice now that for any $x\in F$ and $\varrho>0$ small enough $B^-_\varrho(x)=B_\varrho(x)\cap R$, $B^+_\varrho(x)=B_\varrho(x)\cap R'$, and, since $A_\mu(x)=A_\nu(x)=x$ for every $x\in F$, we have that for any $\varepsilon >0$
 \begin{equation*}
 \begin{split}
  B^+_\varrho(x) \cap \{|v(A_\mu(y))-v(A_\mu(x))| > \varepsilon \} &=  B^+_\varrho(x) \cap \{|v(A_\mu(y))-v(x)| > \varepsilon \} \\&\subset  B^-_{C(\mu)\varrho}(x) \cap \{|v(y)-v(x)| > \varepsilon\} \,,
  \end{split}
 \end{equation*}
 so that 
\begin{equation*} 
 \aplim \limits_{y\in R', \, y\to x}  v(A_{\mu}y)= \aplim \limits_{y\in R, \, y\to x}  v(y)= (\widehat{v})^-_F(x)\,,
 \end{equation*} 
 and the same holds for $\nu$ in place of $\mu$. Therefore we can argue as in Remark~\ref{3105171143} to conclude that
 \begin{equation*}
 (\widehat{v})^+_F(x)= (\widehat{v})^-_F(x)
 \end{equation*}
 Indeed, we employ the fact that, by \eqref{1701182113} and \eqref{1701182114}, for every $i=1,\dots,n$ the function $v'_i$ is a combination of $v_i(A_\mu x)$ and $v_i(A_\nu x)$ with two coefficients whose sum is $1$. This gives \eqref{2705170936} and in particular we observe that almost every slice $\widehat{(\widehat{v})}\xy $ does not jump on the unique point of $F\xy$, so that we can take 
 \[\lambda_{\widehat{v}}(B)=\lambda_{v}(B\cap R) + \lambda_{v'}(B\cap R')\]
 in the characterisation of $\widehat{v}\in GSBD$.

The first condition in \eqref{3005172337} now gives that
\begin{equation*}
\hn(J_{v'})\leq  \hn(A_\mu^{-1}(J_v))+  \hn(A_\nu^{-1}(J_v)) \,,
\end{equation*}
and \eqref{2705170937} follows. 
By the second condition in \eqref{3005172337} we deduce \eqref{2705170938}.
 Notice that the constant  $c$ depends on $p$, $\mu$, and $\nu$, but is independent on $R$ and $v$.
\end{proof}

Let us recall the following important result, proven in \cite[Proposition~3]{CCF16}. Notice that the result is stated in $SBD$, but the proof, which is based on the Fondamental Theorem of Calculus along lines, still holds for $GSBD$, with small adaptations.
\begin{proposition}\label{prop:3CCF16}
Let $Q =(-r,r)^n$, $Q'=(-r/2, r/2)^n$, $u\in GSBD^p(Q)$, $p\in [1,\infty)$. Then there exist a Borel set $\omega\subset Q'$ and an affine function $a\colon \Rn\to\Rn$ with $e(a)=0$ such that $\mathcal{L}^n(\omega)\leq cr \hn(J_u)$ and
\begin{equation}\label{prop3iCCF16}
\int\limits_{Q'\setminus \omega}(|u-a|^{p}) ^{1^*} \dx\leq cr^{(p-1)1^*}\Bigg(\int\limits_Q|e(u)|^p\dx\Bigg)^{1^*}\,.
\end{equation}
If additionally $p>1$, then there is $q>0$ (depending on $p$ and $n$) such that, for a given mollifier $\varphi_r\in C_c^{\infty}(B_{r/4})\,, \varphi_r(x)=r^{-n}\varphi_1(x/r)$, the function $v=u \chi_{Q'\setminus \omega}+a\chi_\omega$ obeys
\begin{equation}\label{prop3iiCCF16}
\int\limits_{Q''}|e(v\ast \varphi_r)-e(u)\ast \varphi_r|^p\dx\leq c\left(\frac{\hn(J_u)}{r^{n-1}}\right)^q \int\limits_Q|e(u)|^p\dx\,,
\end{equation}
where $Q''=(-r/4,r/4)^n$.
The constant in \eqref{prop3iCCF16} depends only on $p$ and $n$, the one in \eqref{prop3iiCCF16} also on $\varphi_1$.  
\end{proposition}

\begin{remark}
Condition \eqref{prop3iCCF16} is a Korn-Poincaré-type inequality, which guarantees the existence of an affine function $a$ such that, up to a small exceptional set, $u-a$ is controlled in a space better than $L^p$. The control in the optimal space $L^{p^\ast}$ is obtained only if $p=1$.
Even on the exceptional set, the affine function $a$ is in some sense ``close in energy'' to $u$, as follows from \eqref{prop3iiCCF16}.
 \end{remark}
 
 \begin{remark}
 By H\"older inequality and \eqref{prop3iCCF16} it follows that
 \begin{equation}\label{prop3CCFHolder}
\int\limits_{Q'\setminus \omega} |u-a|^{p} \dx\leq \mathcal{L}^n(Q'\setminus \omega)^{1/n}\Bigg(\int\limits_{Q'\setminus \omega}(|u-a|^{p}) ^{1^*} \dx\Bigg)^{1/1^*}\hspace{-1em}\leq  c r^p \int\limits_Q|e(u)|^p\dx
 \end{equation}

 \end{remark}

The following lemma will be employed in zones where the jump of $u$ is small, compared to the side of the square.
It will be useful to estimate, for two cubes with nonempty intersection, the difference of the corresponding affine functions.
\begin{lemma}\label{le:diffaffini}
For every $\alpha_i \in \{-1,0,1\}^n$, with $\alpha_0 = 0$, let $z_i=\frac{r}{2} \alpha_i \in \Rn$ and $Q_i$, $Q'_i$, $Q''_i$ be the $n$-dimensional cubes of center $z_i$ and sidelength $2r$, $r$, $r/2$, respectively (assume $r<1$). Let $u\in GSBD(B(0,6r))$   
and, for $i=0,\dots,3^n$, let
 $a_i$ and $\omega_i$ be the affine function and the exceptional set given by Proposition~\ref{prop:3CCF16}, corresponding to $Q_i$.
Assume that for every $i=0,\dots,3^n$
\begin{equation}\label{3005172113}
\hn(J_u\cap Q_i) \leq \theta r^{n-1}\,,
\end{equation}
with $\theta$ sufficiently small (for instance $\theta\leq 1/(16 c)$, for $c$ as in \eqref{prop3iCCF16}.
Then there exists a constant $C$, depending only on $p$ and $n$, such that for each $i \neq 0$
\begin{equation}\label{eq:diffaffinifinale}
\|a_0-a_i\|_{L^\infty(Q_0\cap Q_i;\,\Rn)}^p\leq C r^{-(n-p)}\hspace{-1em} \int \limits_{Q_0\cup Q_i}\hspace{-0.5em}|e(u)|^p\dx\,,
\end{equation}
\end{lemma} 

\begin{proof}
By \eqref{3005172113} we have that
\begin{equation*}
\begin{split}
\mathcal{L}^n(\omega_0\cup\omega_i) \leq cr\left( \hn(J_u\cap Q_0)+\hn(J_u\cap Q_i)\right) \leq  2c \, \theta \, r^n \leq \frac{\mathcal{L}^n(Q_0'\cap Q_i')}{4} \,.
\end{split}
\end{equation*}

Therefore, following the argument of \cite[Lemma~4.3]{ConFocIur15} for the rectangles $Q'_0 \cap Q'_i$ in place of $B$ (notice that for a given $i$ the shape of these rectangles is the same independently of $r$, that is the ratios between the sidelengths are independent of $r$) one has that for any affine function $a\colon \Rn\to \Rn$
\begin{equation*}
\mathcal{L}^n(Q'_0\cap Q'_i) \|a\|_{L^\infty(Q'_0\cap Q'_i;\,\Rn)} \leq \ol c \hspace{-2em}\int\limits_{(Q'_0\setminus \omega_{0}) \cap (Q'_i\setminus \omega_{i})}\hspace{-3em}|a| \dx \,,
\end{equation*}
for $\ol c>0$ depending only on $n$ (and on $i$).
By H\"older's inequality we deduce that for any $q\in [1,\infty)$  
\begin{equation*}\label{3005172103}
\mathcal{L}^n(Q'_0\cap Q'_i) \|a\|_{L^\infty(Q'_0\cap Q'_i;\,\Rn)}^q \leq \ol c\,^q \hspace{-2em}\int\limits_{(Q'_0\setminus \omega_{0}) \cap (Q'_i\setminus \omega_{i})}\hspace{-3em}|a|^q \dx \,.
\end{equation*}
For $q= p \, 1^*$ and $a= a_0-a_i$ we get 
\begin{equation}\label{3005172127}
\mathcal{L}^n(Q'_0\cap Q'_i) \|a_0-a_i\|_{L^\infty(Q'_0\cap Q'_i;\,\Rn)}^{p \,1^*}\leq \ol c\,^{p\,1^*} \hspace{-2em}\int \limits_{(Q'_0\setminus \omega_{0}) \cap (Q'_i\setminus \omega_{i})} \hspace{-3em} |a_0-a_i|^{p\,1^*}\dx\,.
\end{equation}
By triangle inequality and by \eqref{prop3iCCF16} it follows that
\begin{equation}\label{3005172128}
\int \limits_{(Q'_0\setminus \omega_{0}) \cap (Q'_i\setminus \omega_{i})} \hspace{-3em} |a_0-a_i|^{p\,1^*}\dx\leq \hspace{-3em} \int \limits_{(Q'_0\setminus \omega_{0}) \cap (Q'_i\setminus \omega_{i})} \hspace{-3em} \left(|u-a_0|+|u-a_i|\right)^{p\,1^*}\dx \leq c r^{(p-1)\,1^*} \Bigg( \int \limits_{Q_0\cup Q_i}\hspace{-0.5em}|e(u)|^p\dx \Bigg)^{1^*}\,.
\end{equation}
Moreover, since $a_0-a_i$ is an affine function, we have that
\begin{equation}\label{0805172328}
\|a_0-a_i\|_{L^\infty(Q_0\cap Q_i;\,\Rn)}\leq \ol C \|a_0-a_i\|_{L^\infty(Q_0'\cap Q_i';\,\Rn)}
\end{equation} 
for a constant $\ol C$ depending only on the ratio between $\mathcal{L}^n(Q_0\cap Q_i)$ and $\mathcal{L}^n(Q_0'\cap Q_i')$, which is independent of $r$. 

We deduce \eqref{eq:diffaffinifinale} by collecting \eqref{3005172127}, \eqref{3005172128}, and \eqref{0805172328}.
\end{proof}

\section{A first approximation result with a bad constant}\label{Sec2}

As in \cite{Cha04, Iur14, CFI17Density}, a first step toward the main density result consists in a rough approximation in the sense of energy.
In particular, in this section we construct an approximating sequence of functions whose jumps are controlled in terms of the original jump by a multiplicative parameter.
We employ this result in the next section for subdomains where the jump of the original function is very small, so that the total increase of energy will be small too.

\begin{theorem}\label{teo:rough}
Let $\Omega$, $\widetilde{\Omega}$ be bounded open
subsets of $\Rn$, with $\ol \Omega\subset \widetilde{\Omega}$, 
$p > 1$, $\theta \in (0,1)$, and let $u\in GSBD^p(\widetilde{\Omega})$. Then there exist
$u_k\in SBV^p(\Omega;\Rn)\cap L^\infty(\Omega; \Rn)$ and $E_k\subset \Omega$ Borel sets such that
$J_{u_k}$ is included in a finite union of $(n{-}1)$--dimensional closed cubes, $u_k\in W^{1,\infty}(\Omega\setminus J_{u_k}; \Rn)$, and the following hold:
\begin{subequations}
\begin{align}
\lim_{k\to \infty}\mathcal{L}^n(E_k)= \lim_{k\to \infty} & \int\limits_{\Omega\setminus E_k} |u_k -u|^p \dx = 0 \,, \label{1rough}\\
\limsup_{k\to \infty} \int\limits_\Omega  |e(u_k)|^p \dx & \leq \int\limits_\Omega |e(u)|^p \dx \,, \label{2rough}\\
\hn(J_{u_k}\cap \Omega)& \leq C \,\theta^{-1} \hn(J_u\cap \widetilde{\Omega})\,,\label{3rough}
\end{align}
for suitable $C>0$ independent of $\theta$.
In particular, $u_k$ converge to $u$ in measure in $\Omega$.      
Moreover, if $
\int\limits_\Omega \psi(|u|)\dx
$ is
finite 
for $\psi\colon [0,\infty)\to [0,\infty)$ 
increasing, 
continuous, and satisfying \eqref{hppsi} (see Theorem~\ref{teo:main}), then
\begin{equation}\label{4rough}
\lim_{k\to \infty} \int\limits_{\Omega}\psi(|u_k-u|)\dx = 0 \,.
\end{equation}
\end{subequations}
\end{theorem}
 The proof of the result above employs a technique introduced in \cite{CCF17}, which is based on Proposition~\ref{prop:3CCF16}.
The idea is to partition the domain into cubes of side $\tfrac{1}{k}$ and to distinguish, at any scale, the cubes where the ratio between the perimeter and the jump of $u$ is greater than the parameter $\theta$.

In such cubes, one may replace the original function $u$ with a constant function, since on the one hand the new jump is controlled by the original jump, and on the other hand the total volume of these cubes is small as the length scale goes to 0.

In the remaining cubes, where the relative jump is small, one applies Proposition~\ref{prop:3CCF16}: a Korn-Poincaré-type inequality holds up to a set of small volume, and in this small exceptional set the original function may be replaced by a suitable affine function without perturbing much its energy. 
We need $u$ be defined in a larger set $\widetilde{\Omega}$ since we will take convolutions of the original function.

\begin{proof}[Proof of Theorem~\ref{teo:rough}] 
Let $\ol \Omega\subset \widetilde{\Omega}$, $p >1$, $\theta\in (0,1)$, and $u \in GSBD^p(\widetilde{\Omega})$. 
Let us fix an integer $k$ with $k> \frac{ 12  \sqrt{n}}{\mathrm{dist }(\partial \Omega, \partial \widetilde{\Omega})}$, let $\varphi$ be a smooth radial function with compact support in the unit ball $B(0,1)$, 
and let $\varphi_k(x)=k^n \varphi(kx)$. 
\newline
\\
{\bf Good and bad nodes.}
For any $z\in (2 \km) \Z^n \cap  (\Omega+[-k^{-1},k^{-1}]^n) $ consider the cubes of center $z$
\begin{equation*}
\begin{split}
q_z^k&:=z+(-\km,\km)^n\,,\quad \hspace{1.3em}\tq_z^k:= z+(-2\km,2\km)^n\,,\\ Q_z^k&:=z+(-4\km,4\km)^n\,, \quad \widetilde{Q}_z^k:=z+(-8\km,8\km)^n\,.
\end{split}
\end{equation*}
Let us define the sets of the ``good'' and of the ``bad'' nodes
\begin{equation}\label{eq:defGoodBad}
\begin{split}
 & G^k:=\{z\in (2 \km) \Z^n \cap  (\Omega+[-k^{-1},k^{-1}]^n)  : \hn(J_u\cap \Qz)\leq \theta k^{-(n-1)}\}\,, \\
 & B^k:=(2 \km) \Z^n \cap  (\Omega+[-k^{-1},k^{-1}]^n)  \setminus G^k\,,
\end{split}
\end{equation}
such that the amount of jump of $u$ is small in a big neighbourhood of any $z\in G^k$, and the corresponding subsets of $\widetilde{\Omega}$
\begin{equation*}
\Omega_g^k:=\bigcup_{z\in G^k} \qz\,,\quad \widetilde{\Omega}^k_b:=\bigcup_{z\in B^k}\Qz\,.
\end{equation*}
Notice that $\widetilde{\Omega}^k_b$ is the union of cubes of sidelength $8k^{-1}$, while $\Omega_g^k$ is the union of cubes of sidelength $2k^{-1}$, so that $ \ol{\Omega}\setminus \Omega_g^k  \subset \widetilde{\Omega}^k_b$. More precisely, 
\begin{equation}\label{1205171038}
 \ol{\Omega}\setminus \Omega_g^k  + B(0, k^{-1}) \subset \widetilde{\Omega}^k_b
\end{equation}
Indeed, by construction, a row of ``boundary'' cubes of $\Omega_g^k$ belongs to $\widetilde{\Omega}_b^k$.
Moreover, by  \eqref{eq:defGoodBad} the set $B^k$ has at most $\hn(J_u)\, k^{n-1} \theta^{-1}$ elements, so that
\begin{equation}\label{1805171025}
\mathcal{L}^n\left(\widetilde{\Omega}^k_b\right)\leq 16^n \frac{\hn(J_u)}{k\,\theta}\,. 
\end{equation} 
Let us apply Proposition~\ref{prop:3CCF16} for any $z\in G^k$.
Then there exist a set $\omega_z \subset \tqz$, with 
\begin{equation}\label{1005171230}
\mathcal{L}^n(\omega_z)\leq c k^{-1} \hn(J_u\cap \Qz) \leq c \theta k^{-n}\,,
\end{equation}
 and an affine function $a_z\colon \Rn \to \Rn$, with $e(a_z)=0$, such that 
\begin{equation}\label{prop3iCCF16applicata}
\int\limits_{\tqz\setminus \omega_z}(|u-a_z|^{p}) ^{1^*} \dx\leq ck^{-(p-1)1^*}\Bigg(\int\limits_{\Qz}|e(u)|^p\dx\Bigg)^{1^*}
\end{equation}
and, 
letting $v_z:= u\chi_{\tqz\setminus \omega_z}+a_z \chi_{\omega_z}$,
\begin{equation}\label{prop3iiCCF16applicata}
\begin{split}
\int\limits_{\qz}|e(v_z\ast \varphi_k)-e(u)\ast \varphi_k|^p\dx & \leq c\left(\hn(J_u\cap \Qz)\,k^{n-1}\right)^q \int\limits_{\Qz}|e(u)|^p\dx \\ & \leq c \, \theta^q \int\limits_{\Qz}|e(u)|^p\dx \,,
\end{split}
\end{equation}
for a suitable $q>0$ depending on $p$ and $n$.

Let 
\begin{equation}\label{1209181245}
\omega^k:= \bigcup_{z\in G^k} \omega_{z}\,,\qquad
E_k:=\widetilde{\Omega}_b^k \cup \omega^k\,.
\end{equation}
Then
\begin{equation}\label{1805171033}
\lim_{k\to \infty}\mathcal{L}^n(E_k)=0\,,
\end{equation}
by \eqref{1805171025} and \eqref{1005171230}, which gives $\mathcal{L}^n(\omega^k)\leq c k^{-1} \sum_{z\in G^k}  \hn(J_u\cap Q_{z}^k) \leq c \hn(J_u)\, k^{-1}$. 

We split the set of good nodes in the two subsets
\begin{equation*}
G^k_1:=\{z\in G^k \colon \hn(J_u\cap \Qz)\leq k^{-(n-\frac{1}{2})}\}\,, \qquad G^k_2:= G^k\setminus G^k_1\,.
\end{equation*}
Notice that $G^k_1$ are the good nodes for which the condition on $J_u$ is satisfied for $k^{-\frac{1}{2}}$ in place of $\theta$. 
For each $z \in G^k_1$, we have that \eqref{1005171230} and \eqref{prop3iiCCF16applicata} hold with $k^{-\frac{1}{2}}$ in place of $\theta$, namely
\begin{subequations}\label{eqs:0807171955}
\begin{align}
\mathcal{L}^n(\omega_z)&\leq c \,k^{-(n+\frac{1}{2})}\,,\label{0807171956}\\ 
\int\limits_{\qz}|e(v_z\ast \varphi_k)-e(u)\ast \varphi_k|^p\dx &\leq c \, k^{-\frac{q}{2}} \int\limits_{\Qz}|e(u)|^p\dx\,. \label{0807171957}
\end{align}
\end{subequations}
 
Let us introduce also
\begin{equation}
\begin{split}
\widetilde{G}^k_1&:=\{z\in G^k \colon \ol z \in G^k_1 \text{ for each } \ol z \in (2k^{-1}) \Z^n \text{ with } \| z - \ol z\|_\infty= 2k^{-1}\}\,,\\
\widetilde{G}^k_2&:=\{z\in G^k \colon \text{ there exists } \ol z \in G^k_2 \text{ with } \| z - \ol z\|_\infty= 2k^{-1}\}\,,
\end{split}
\end{equation}
where $\| z - \ol z\|_\infty:=\sup_{1\leq i\leq n} |z_i - \ol z_i| $ is the $L^\infty$ norm of the vector $z - \ol z$. 

Arguing as already done for $\widetilde{\Omega}^k_b$, we get that $G^2_k$ has at most $\hn(J_u)\, k^{n-\frac{1}{2}}$ elements, so 
$
\widetilde{G}^k_2$ has at most $(3^n-1) \hn(J_u)\, k^{n-\frac{1}{2}}$ elements, and  
\begin{equation}\label{0907170918}
\mathcal{L}^n(\widetilde{\Omega}^k_{g,2})\leq C\, k^{-\frac{1}{2}},\qquad \text{ for }\,\widetilde{\Omega}^k_{g,2}:= \bigcup_{z_j \in \widetilde{G}^k_2} \widetilde{Q}_{z_j}\,.
\end{equation}
{\bf The approximating functions.}
Let $G^k=(z_j)_{j\in J}$, so that we order (arbitrarily) the elements of $G^k$, and let us define 
\begin{equation}\label{eq:defappr1}
\widetilde{u}_k:=
\begin{cases}
 u \quad &\text{in }\widetilde{\Omega}\setminus \omega^k\,,\\
 a_{z_j}\quad &\text{in }\omega_{z_j}\setminus \bigcup_{i<j}\omega_{z_i}\,,
\end{cases}
\end{equation}
and
\begin{equation}\label{eq:defapprox}
u_k:= \left(\widetilde{u}_k \ast \varphi_k \right)\chi_{\Omega\setminus \widetilde{\Omega}_b^k}\,.
\end{equation}
These are the approximating functions for the original $u$, for which we are going to prove the properties of the theorem.

By construction, $u_k\in SBV^p(\Omega;\Rn)\cap L^\infty(\Omega; \Rn)$  (notice that $\Omega\setminus \widetilde{\Omega}_b^k + \mathrm{supp} \,\varphi_k \subset \Omega^k_g$, and $u \in L^p(\Omega^k_g \setminus \omega^k)$),
 $J_{u_k}\subset \bigcup_{z\in B^k} \partial \Qz$, which is a finite union of $(n{-}1)$--dimensional closed cubes, and $u_k\in W^{1,\infty}(\Omega\setminus J_{u_k}; \Rn)$.
\newline
\\
{\bf Proof of \eqref{3rough}.}
For any $z\in B^k$ we have that
\begin{equation*}
\hn(\partial \Qz)=C(n)\, k^{-(n-1)}\leq C(n)\,\theta^{-1}\hn(J_u\cap \Qz)\,,
\end{equation*}
so that \eqref{3rough} follows by summing over $z\in B^k$. Notice that we use here the fact that the cubes $\Qz$ are finitely overlapping; this will be done different times also in the following (also for the cubes $\tqz$, $\widetilde{Q}^k_z$).

To ease the reading, in the following we denote $\omega_{z_j}$ by $\omega_j$, and the same for $a_{z_j}$, $v_{z_j}$. We denote also $q^k_{z_j}$ by $q_j$, and the same for $\tilde{q}_{z_j}^k$, $Q_{z_j}^k$, $\widetilde{Q}_{z_j}^k$. Moreover, for any $g\colon \Omega\to \Rn$, $B\subset \Omega$, and $q\in [1, \infty]$ we write $\|g\|_{L^q(B)}$ instead of $\|g\|_{L^q(B;\Rn)}$.
\newline
\\
{\bf Proof of \eqref{1rough}.}
In order to prove \eqref{1rough} let us fix $j\in J$ such that $q_j \subset \Omega \setminus \widetilde{\Omega}_b^k$. By the triangle inequality
\begin{equation}\label{2305172312}
\|u_k-u\|_{L^p(q_j\setminus \omega^k)}\leq \|u_k-a_j\|_{L^p(q_j\setminus \omega^k)}+\|u-a_j\|_{L^p(q_j\setminus \omega^k)}
\end{equation} 
Notice that
\begin{equation*}
u_k-a_j=\varphi_k\ast (\widetilde{u}_k-a_j) \quad\text{in }\Omega\setminus \widetilde{\Omega}_b^k\,,
\end{equation*}
by definition of $\widetilde{u}_k$ and since $\varphi_k\ast a_j=a_j$, being $\varphi$ a radial function.

By \eqref{prop3CCFHolder} we get
\begin{equation}\label{2305172326}
\|u-a_j\|_{L^p(q_j\setminus \omega^k)}\leq \|u-a_j\|_{L^p(\tilde{q}_j\setminus \omega_j)}\leq c\,k^{-1} \Bigg(\int\limits_{Q_j}|e(u)|^p\dx\Bigg)^{1/p}\hspace{-1em}.
\end{equation}
We now estimate the first term on the right hand side of \eqref{2305172312} as follows:  
\begin{equation}\label{2305172355}
\begin{split}
\hspace{-1em}\|u_k-a_j\|_{L^p(q_j\setminus \omega^k)}&\leq \|u_k-a_j\|_{L^p(q_j)} = \|\varphi_k\ast \big((\widetilde{u}_k-a_j)\chi_{\tilde{q}_j}\big)\|_{L^p(q_j)}\\&\leq \|\varphi_k\ast \big((u-a_j)\chi_{\tilde{q}_j\setminus \omega^k}\big)\|_{L^p(q_j)}+ \|\varphi_k\ast \big((\widetilde{u}_k-a_j)\chi_{\tilde{q}_j\cap \omega^k}\big)\|_{L^p(q_j)}\\
&\leq \|u-a_j\|_{L^p(\tilde{q}_j\setminus \omega^k)}+\|\widetilde{u}_k-a_j\|_{L^p(\tilde{q}_j\cap \omega^k)}\,.
\end{split}
\end{equation}
Notice that we have used the fact that $q_j+\mathrm{supp}\,\varphi_k\subset q_j+B(0,k^{-1})\subset \tilde{q}_j$.
The first term on the right hand side of \eqref{2305172355} is estimated by \eqref{2305172326}. As for the second one we have, by definition of $\widetilde{u}_k$, that
\begin{equation}\label{1110171817}
\begin{split}
\int \limits_{\tilde{q}_{j}\cap \omega^k} \hspace{-0.5em} |\widetilde{u}_k-a_{j}|^p\dx & = \sum_{i\neq j} \int \limits_{\tilde{q}_{j}\cap \widetilde{\omega}_i} \hspace{-0.5em} |\widetilde{u}_k-a_{j}|^p\dx
= \sum_{i < j}  \int \limits_{\omega_{j}\cap \, \widetilde{\omega}_i}  \hspace{-0.5em}  |a_{i}-a_{j}|^p\dx\\
& + \sum_{i<j} \int\limits_{\tilde{q}_{j}\cap \widetilde{\omega}_i \sm \omega_j} \hspace{-0.8em} |u-a_j|^p \dx + \sum_{i>j} \int\limits_{\tilde{q}_{j}\cap \widetilde{\omega}_i} \hspace{-0.5em} |u-a_j|^p \dx
\end{split}
\end{equation}
where $\widetilde{\omega}_i:=\omega_{i}\setminus (\bigcup_{h<i} \omega_{h})$. 
Now, the sums above involve at most $3^n-1$ terms  corresponding to the centers $z_i$ with $z_j- z_i= 2k^{-1}\,\alpha_i$ and $\alpha_i\in \{-1,0,1\}^n$, because for any other $z_h$ we have $ \tilde{q}_j \cap \omega_h \subset \tilde{q}_j\cap \tilde{q}_h = \emptyset$.  Let us estimate the terms in the right hand side of  \eqref{1110171817}. 
We have
\begin{equation}\label{0807171820}
\begin{split}
  \int \limits_{\omega_{j}\cap \, \widetilde{\omega}_i}  \hspace{-0.5em}  |a_{i}-a_{j}|^p\dx  \leq  \int \limits_{\tilde{q}_j \cap \widetilde{\omega}_i} \hspace{-0.5em}  |a_{i}-a_{j}|^p\dx & \leq \mathcal{L}^n\left(\omega_i \right) \|a_{i}-a_{j}\|_{L^\infty(Q_{i}\cap Q_{j})}^p\leq C \, \theta \, k^{-p} \hspace{-1em}\int \limits_{ Q_{i}\cup Q_{j}} \hspace{-1em} |e(u)|^p\dx\\ &\leq C \, \theta \, k^{-p} \int \limits_{ \widetilde{Q}_{j}} |e(u)|^p\dx\,,
\end{split}
\end{equation}
by using \eqref{eq:diffaffinifinale} and \eqref{1005171230}.   Moreover, each of the remaining terms in the right hand side of  \eqref{1110171817} is less than
\begin{equation}\label{1109181810}
\int \limits_{\tilde{q}_j\cap \widetilde{\omega}_i \sm \omega_j} \hspace{-0.5em}  |u-a_{j}|^p \dx \leq c \, \mathcal{L}^n(\omega_{i})^{1/n}  k^{-(p-1)} \int \limits_{Q_{j}} |e(u)|^p\dx \leq C \theta^{1/n} k^{-p} \int \limits_{Q_{j}}  |e(u)|^p\dx\,,
\end{equation}
by employing H\"older inequality as in \eqref{prop3CCFHolder}, and \eqref{1005171230}.

Therefore, since the terms in the sums are at most $3^n-1$, we obtain that
\begin{equation}\label{2405170137}
\|\widetilde{u}_k-a_j\|_{L^p(\tilde{q}_j\cap \omega^k)} \leq C\, \theta^{1/ n  p} k^{-1} \|e(u)\|_{L^p(\widetilde{Q}_j)} \,.
\end{equation}
In preparation to the proof of \eqref{2rough} and \eqref{4rough}, we remark that if $z_j \in \widetilde{G}^k_1$ then
\begin{equation}\label{0807171814}
\|\widetilde{u}_k-a_j\|_{L^p(\tilde{q}_j\cap \omega^k)} \leq C\,  k^{-(1+\frac{1}{2 n  p})} \|e(u)\|_{L^p(\widetilde{Q}_j)} \,,
\end{equation}
namely \eqref{2405170137} holds true for $k^{-\frac{1}{2}}$ in place of $\theta$. Indeed, one employs \eqref{0807171956} for every $i$ in \eqref{0807171820}  and \eqref{1109181810}  ($z_i\in G^k_1$ for any $i$ therein, by definition of $\widetilde{G}^k_1$).

Collecting \eqref{2305172312}, \eqref{2305172326}, \eqref{2305172355}, \eqref{2405170137}, and summing on $j$, we deduce
\begin{equation*}
\|u_k-u\|_{L^p(\Omega \setminus E_k)}\leq C\, k^{-1} \|e(u)\|_{L^p(\Omega)}
\end{equation*}
which gives \eqref{1rough} together with \eqref{1805171033}.
\newline
\\
{\bf Proof of \eqref{4rough}.} As above, let us fix $j\in J$ such that $q_j \subset \Omega \setminus \widetilde{\Omega}_b^k$, and let $\psi$ be as in the statement of the theorem. 
Then
\begin{equation}\label{0107171127}
\int \limits_{q_j\cap \omega^k} \hspace{-0.7em} \psi(|u_k-u|)\dx\leq C_\psi  \hspace{-0.7em} \int\limits_{q_j\cap \omega^k}\hspace{-0.7em}  \psi(|u_k-a_j|)\dx+  C_\psi  \hspace{-0.7em} \int\limits_{q_j\cap \omega^k} \hspace{-0.7em}  \psi(|u-a_j|)\dx\,,
\end{equation}
For the first term in the right hand side above we have
\begin{equation}\label{0107171311}
\begin{split}
\int\limits_{q_j\cap \omega^k}\hspace{-0.7em}  \psi(|u_k-a_j|)\dx &\leq  C_\psi \,\mathcal{L}^n(q_j\cap \omega^k)+  C_\psi \hspace{-0.7em}\int\limits_{q_j\cap \omega^k} \hspace{-0.7em} |u_k-a_j|^p\dx \\&\leq  C_\psi  \, \mathcal{L}^n(q_j\cap \omega^k) + C\,  k^{-p} \int\limits_{\widetilde{Q}_j}|e(u)|^p\dx\,,
\end{split}
\end{equation}
by \eqref{2305172326}, \eqref{2305172355}, and \eqref{2405170137} (that control $\int\limits_{q_j} |u_k-a_j|^p\dx$, see the first inequality in \eqref{2305172355}, and then $\int\limits_{q_j\cap \omega^k} \hspace{-0.7em} |u_k-a_j|^p\dx$).
As for the second term in the right hand side of \eqref{0107171127}, it holds that
\begin{equation*}
\int\limits_{q_j\cap \omega^k} \hspace{-0.7em}\psi(|u-a_j|)\dx \leq  C_\psi  \hspace{-0.7em}\int\limits_{q_j\cap \omega^k} \hspace{-0.7em} \psi(|u|)\dx +  C_\psi  \hspace{-0.7em} \int\limits_{q_j\cap \omega^k} \hspace{-0.7em} \psi(|a_j|)\dx\,,
\end{equation*}
and 
\begin{equation*}
\begin{split}
 \int\limits_{q_j\cap \omega^k} \hspace{-0.7em} \psi(|a_j|)\dx &   \leq  C \frac{\mathcal{L}^n(q_j\cap \omega^k)}{\mathcal{L}^n(q_j)} \int\limits_{q_j}\psi(|a_j|)\dx  \leq C \, \theta \int\limits_{q_j}\psi(|a_j|)\dx  \leq C\, C_\psi\, \theta \int\limits_{q_j} \big(\psi(|u|) + \psi(|u-a_j|)\big)\dx \\&\leq C\,C_\psi\, \theta \bigg(\int\limits_{q_j} \psi(|u|) \dx + \hspace{-0.7em}\int\limits_{q_j\cap \omega^k} \hspace{-0.7em} \psi(|u-a_j|)\dx + C_\psi\,\mathcal{L}^n(q_j\setminus \omega^k)+  C_\psi\, k^{-p} \int\limits_{Q_j}|e(u)|^p\dx\bigg)\,,
\end{split}
\end{equation*}
by \eqref{2305172326}.
Being $\theta$ small, by the two previous inequalities 
we get that
\begin{equation}\label{0107171310}
\begin{split}
\int\limits_{q_j\cap \omega^k} \hspace{-0.7em} \psi(|u-a_j|)\dx\leq  C \bigg(\hspace{-0.7em}\int\limits_{q_j\cap \omega^k} \hspace{-0.7em} \psi(|u|)\dx + \theta \int\limits_{q_j} \psi(|u|) \dx  + \theta\,\mathcal{L}^n(q_j\setminus \omega^k)+  k^{-p} \int\limits_{Q_j}|e(u)|^p\dx\bigg)\,,
\end{split}
\end{equation}
where powers of $C_\psi$ have been absorbed in $C$.
We now collect \eqref{0107171127}, \eqref{0107171311}, \eqref{0107171310}, to get
\begin{equation*}
\begin{split}
\int \limits_{q_j\cap \omega^k} \hspace{-0.7em} \psi(|u_k-u|)\dx &\leq C\, \bigg(\mathcal{L}^n(q_j\cap \omega^k) +  k^{-p} \int\limits_{\widetilde{Q}_j}|e(u)|^p\dx +    \hspace{-0.7em}\int\limits_{q_j\cap \omega^k} \hspace{-0.7em} \psi(|u|)\dx + \theta \int\limits_{q_j} \psi(|u|) \dx  \\ &\quad+ \theta\,\mathcal{L}^n(q_j\setminus \omega^k)\bigg)\,.
\end{split}
\end{equation*}
Again, notice that if $z_j \in \widetilde{G}^k_1$, then the inequality above holds for $k^{-\frac{1}{2}}$ in place of $\theta$ (indeed $\frac{\mathcal{L}^n(q_j\cap \omega^k)}{\mathcal{L}^n(q_j)} \leq C\,k^{-\frac{1}{2}}$ in the estimate before \eqref{0107171310}).

Let us sum over $j \in J$, distinguishing the centers in $\widetilde{G}^k_1$ and the remaining ones, that we may assume in  $\widetilde{G}^k_2$, recalling \eqref{1205171038} and the definition of $u_k$ \eqref{eq:defapprox}. We deduce that
\begin{equation*}
\begin{split}
\int \limits_{E_k} \psi(|u_k-u|)\dx &\leq C \bigg( \mathcal{L}^n(E_k) + k^{-p} \int\limits_{\widetilde{\Omega}}|e(u)|^p\dx  +   \int\limits_{E_k} \psi(|u|)\dx +  \theta \hspace{-0.5em}\int \limits_{\widetilde{\Omega}^k_{g,2}} \hspace{-0.5em}\psi(|u|)\dx \\
&\quad  + k^{-\frac{1}{2}} \int\limits_\Omega \psi(|u|)\dx + k^{-\frac{1}{2}}\, \mathcal{L}^n(\Omega)  + \theta \mathcal{L}^n\Big(\bigcup_{z_j \in \widetilde{G}^k_2} q_{z_j}\Big) \bigg)
\,,
\end{split}
\end{equation*}
By \eqref{1805171033}, \eqref{0907170918}, and since $\psi(|u|)\in L^1(\Omega)$, it follows that
\begin{equation}\label{0107171956}
\lim_{k\to \infty} \int \limits_{E_k} \psi(|u_k-u|)\dx =0\,.
\end{equation}
Eventually, by \eqref{1rough} and \eqref{hppsi}
\begin{equation}\label{0107172009}
\lim_{k\to \infty} \int \limits_{\Omega\setminus E_k} \psi(|u_k-u|)\dx =0\,.
\end{equation}
Indeed, $\psi=\psi_1+\psi_2$ for suitable $0\leq \psi_1\leq M_\psi$, and 
$0 < \psi_2(s)\leq M_\psi |s|^p$, with $M_\psi>0$.
Since $\psi_1\,,\psi_2 \geq 0$, \eqref{1rough}, \eqref{hppsi} imply that $v_k^i:=\chi_{\Omega\setminus E_k} \psi_i(|u_k-u|)$ converges to 0 pointwise for $\mathcal{L}^n$-a.e.\ $x\in \Omega$, for $i=1,2$.
Being $\int_{\Omega\setminus E_k} \psi(|u_k-u|)\dx=\int_\Omega v_k^1 \dx + \int_\Omega v_k^2 \dx$, we deduce \eqref{0107172009} since the two integrals go to 0, the first by Dominated Convergence Theorem and the second by \eqref{1rough}.
\newline
\\
{\bf Proof of \eqref{2rough}.}
First we show that, for $v_j$ as in \eqref{prop3iiCCF16applicata}, 
\begin{equation}\label{1005171231}
\int\limits_{\tilde{q}_{j}}|\widetilde{u}_k-v_{j}|^p\dx\leq C \,\theta^{1/n}\, k^{-p} \int\limits_{\widetilde{Q}_{j}}|e(u)|^p\dx\,, \quad\text{ for  } j\in J\,, 
\end{equation}
and 
\begin{equation}\label{0807171945}
\int\limits_{\tilde{q}_{j}}|\widetilde{u}_k-v_{j}|^p\dx\leq C \,k^{-(p+\frac{1}{2n})} \int\limits_{\widetilde{Q}_{j}}|e(u)|^p\dx\,, \quad\text{ for } j\in J \text{ such that }z_j\in \widetilde{G}^k_1\,, 
\end{equation}
Let us first consider a general $j\in J$. Since $\widetilde{u}_k=u=v_j$ in $\tilde{q}_j\setminus \omega^k$ and $v_j=a_j$ in $\omega_j$, it holds that

\begin{equation*}
\begin{split}
\|\widetilde{u}_k-v_{j}\|_{L^p(\tilde{q}_{j})}= \|\widetilde{u}_k-v_{j}\|_{L^p(\tilde{q}_{j}\cap \omega_k)}&\leq \|\widetilde{u}_k-a_{j}\|_{L^p(\tilde{q}_{j}\cap \omega_k)} + \|v_{j}-a_j\|_{L^p(\tilde{q}_{j}\cap \omega_k)}
\\& =\|\widetilde{u}_k-a_{j}\|_{L^p(\tilde{q}_{j}\cap \omega_k)} + \|u-a_j\|_{L^p(\tilde{q}_{j}\cap \omega_k \sm \omega_j)}\,.
\end{split}
\end{equation*}
The first term in the right hand side is estimated by \eqref{2405170137}. As for the other term, we have that
\begin{equation}
\|u-a_j\|_{L^p(\tilde{q}_j\cap \omega^k\setminus \omega_j)} =\sum_{i \neq j} \|u-a_j\|_{L^p(\tilde{q}_j\cap \widetilde{\omega}_i \sm \omega_j) }\,,
\end{equation}
being the $\widetilde{\omega}_i$ pairwise disjoint. Now, the terms in the sum are at most $3^n-1$ (see also before \eqref{1110171817}), each of which bounded by
\begin{equation*}
\|u-a_{j}\|_{L^p(\tilde{q}_j\cap \widetilde{\omega}_i \sm \omega_j)} \leq C\, \theta^{1/ n  p} k^{-1} \|e(u)\|_{L^p(Q_j)}\,,
\end{equation*}
due to \eqref{1109181810}.

Thus \eqref{1005171231} follows.
On the other hand, if $j\in J$ such that $z_j\in \widetilde{G}^k_1$, we deduce \eqref{0807171945} arguing as before, employing  \eqref{0807171814} instead of \eqref{2405170137}, and \eqref{0807171956} instead of \eqref{1005171230} in \eqref{1109181810},  to get 
\begin{equation*}
\|u-a_{j}\|_{L^p(\tilde{q}_j\cap \widetilde{\omega}_i \sm \omega_j)} \leq C\,  k^{-(1+\frac{1}{2np})} \|e(u)\|_{L^p(Q_j)}\,.
\end{equation*}

We now have the inequality 
\begin{equation}\label{1105172002}
(a+b)^p\leq (1+ \varrho)\, a^p +   \frac{ C_p }{\varrho^{p-1}}  b^p
\end{equation}
for  $C_p>0$ depending only on $p$,  any $\varrho \in (0,1)$, and any positive numbers $a$, $b$.  This follows since $(a+b)^p=a^p(1+\frac{b}{a})^p\leq a^p(1+\varrho + c_\varrho (\frac{b}{a})^p)$, for
\[
c_\varrho:=\max_{\alpha > (1+\varrho)^{\frac{1}{p}}-1} \frac{(1+\alpha)^p-(1+\varrho)}{\alpha^p}\,,
\]
and it is not difficult to see that $c_\varrho < \frac{C_p}{\varrho^{p-1}}$, the maximum being attained for $\alpha^*=(1+\varrho)^{\frac{1}{p-1}}-1$.

Fix $j\in J$. By \eqref{1105172002} with $\varrho^{p-1}=\theta^{\frac{1}{2n}}$ we get
\begin{equation}\label{1105172022}
\begin{split}
\int\limits_{q_j} |e(\widetilde{u}_k\ast \varphi_k)|^p\dx \leq \left(1+ \theta^{\frac{1}{2n(p-1)}}\right) \int\limits_{q_j} |e(v_j\ast \varphi_k)|^p \dx  +  C_p \,\theta^{-\frac{1}{2n}} \int\limits_{q_j} |e(\widetilde{u}_k-v_j)\ast \varphi_k|^p\dx \,.
\end{split}
\end{equation}
By \eqref{1005171231} it follows that
\begin{equation}\label{1105172025}
\begin{split}
 C_p \,\theta^{-\frac{1}{2n}} \int\limits_{q_j} |e(\widetilde{u}_k-v_j)\ast \varphi_k|^p\dx  \leq C \theta^{-\frac{1}{2n}} \, k^p \int\limits_{\tilde{q}_j} |\widetilde{u}_k-v_j|^p\dx  \leq C \theta^{\frac{1}{2n}} \int\limits_{\widetilde{Q}_j}|e(u)|^p \dx\,,
\end{split}
\end{equation}
while by \eqref{prop3iiCCF16applicata} and \eqref{1105172002} (for $\varrho^{p-1}=\theta^{\frac{q}{2}}$)
\begin{equation*}
\int\limits_{q_j} |e(v_j\ast \varphi_k)|^p \dx \leq \left(1+ \, \theta^{\frac{q}{2(p-1)}} \right) \int\limits_{q_j} |e(u)\ast \varphi_k|^p\dx+C\,\theta^{\frac{q}{2}} \int\limits_{Q_j}|e(u)|^p\dx\,.
\end{equation*}
Inserting into \eqref{1105172022}, this gives that
\begin{equation}\label{1506171757}
\begin{split}
\int\limits_{q_j} |e(\widetilde{u}_k\ast \varphi_k)|^p\dx &\leq \int\limits_{q_j} |e(u)\ast \varphi_k|^p\dx + C \, \theta^{q'} \Bigg( \int\limits_{q_j} |e(u)\ast \varphi_k|^p\dx + \int\limits_{\widetilde{Q}_j}|e(u)|^p \dx \Bigg) \\ &\leq \int\limits_{q_j} |e(u)\ast \varphi_k|^p\dx + C \, \theta^{q'} \int\limits_{\widetilde{Q}_j}|e(u)|^p \dx \,,
\end{split}
\end{equation}
with $q':=\min\{\frac{q}{2}, \frac{1}{2n},  \frac{q}{2(p-1)}, \frac{1}{2n(p-1)} \}$.
If $j\in J$ is such that $z_j\in \widetilde{G}^k_1$, then \eqref{1506171757} holds true for $k^{-\frac{1}{2}}$ in place of $\theta$, namely
\begin{equation}
\int\limits_{q_j} |e(\widetilde{u}_k\ast \varphi_k)|^p\dx \leq \int\limits_{q_j} |e(u)\ast \varphi_k|^p\dx + C \, k^{-\frac{q'}{2}} \int\limits_{\widetilde{Q}_j}|e(u)|^p \dx \,,
\end{equation}
because we can argue as before, with $\varrho^{p-1}$ equal to $k^{-\frac{1}{4n}}$ and $k^{-\frac{q}{4}}$ in \eqref{1105172002}, and \eqref{0807171945}, \eqref{0807171957} instead of \eqref{1005171231}, \eqref{prop3iiCCF16applicata}, respectively.
Summing for $j \in J$  and recalling the definition of $u_k$ we obtain \eqref{2rough}. Notice that one has to distinguish the contributions for the nodes in $\widetilde{G}^k_1$ and in $\widetilde{G}^k_2$, and to use that
\begin{equation*}
\lim_{k\to \infty}\int \limits_{\widetilde{\Omega}^k_{g,2}} |e(u)|^p\dx =0\,,
\end{equation*}
by \eqref{0907170918} and since $e(u)$ is in $L^p$.
This concludes the proof.
\end{proof}

\section{Proof of the main result}\label{Sec3}

In this section we prove the main approximation result for any $u \in GSBD^p(\Omega)$, 
through more regular functions $u_k$ converging in measure to $u$.
The symmetric difference between the jump sets, $J_{u_k}\triangle J_u$, tends to 0 in $\hn$-measure,
the deformation $e(u)$ is approximated in the strong $L^p$ topology, and there is also convergence for truncation of the traces on $J_u\cup J_{u_k}$ and on the reduced boundary of the domain $\Omega$, which is assumed to be only a set with finite perimeter.

We apply the rough version of the result, that we have shown in Section~\ref{Sec2}, to any (neighbourhood of) set of a suitable partition on $\Omega$, such that the measure of the jump set of $u$ is small in any subset.

A fundamental difference with respect to \cite{Cha04, Iur14, CFI17Density}, that employ also an intermediate rough estimate, is that here we do not use partitions of the unity neither to extend the original function in suitable neighbourhoods of the subsets of the partition nor to glue the approximating functions constructed in any subset. This allows us to avoid any assumption on the integrability of $u$.

\begin{proof}[Proof of Theorem~\ref{teo:main}]
We split the proof into three parts. 
First we approximate in a suitable way $J_u$ (and $\partial \Omega$), in the same spirit of the beginning of the proof of \cite[Theorem~2]{Cha04}, with balls replaced by hypercubes (see also \cite[Lemma~4.2]{CorToa99}). Then we get a finite family of cubes $Q_j$, whose union contains almost all $J_u$, each of which splitted in two parts $Q_j^1$, $Q_j^2$ by the jump set. This gives us a partition of $\Omega$ up to a $\mathcal{L}^n$-negligible set (see \eqref{eqs:2405171202} and \eqref{eq:defb0}).

\medskip

At this stage, the strategy followed in \cite{Cha04} and \cite{Iur14} is to fatten a little bit every set of the covering, defining properly a function in the fattened domain in such a way that the energy does not increase much, and to apply Theorem~\ref{teo:rough} for each subset.
By the way, we have to be very careful both in defining the extension functions and in linking the extended domains. Indeed, for instance we cannot simply glue any approximating functions defined on each enlarged set by a suitable partition of the unity subordinated to the covering, as in \cite[Theorem~2]{Cha04} by the analogous of \cite[Lemma~3.1]{Cha04}. The reason is that, differently from \cite{Cha04}, we do not know a priori the strong convergence in $L^p$ in every subdomain, since now we do not assume $u\in L^p$.
For the same reason, even to extend the function in an enlarged domain, we cannot partition the boundary, make small outer translations and glue by a partition of the unity, as in \cite{Cha04}. 
Consequently, we follow a different argument. First, we use the fact that $\partial Q_j^1 \cap \partial Q_j^2$
is almost flat (this is the intersection of the main part of $J_u$ with $Q_j$), to apply Lemma~\ref{le:Nitsche}, an extension result inspired by Nitsche \cite{Nie81} (see also \cite{FriSol16}), on both sides of any cube. In such a way, we extend the original function in the direction of the outer normal to each side of $J_u$.
Then, we take the function $u$ itself as an extension outside $\partial Q_j$ and apply Theorem~\ref{teo:rough} for each subdomain; the extensions corresponding to $Q_j$ and to the complement of $Q_j$ have the same value on $\partial Q_j$, because they are obtained from $u$ in the same way, in particular by taking convolutions with the same kernel.

\medskip

In the final part the approximating functions on $\Omega$ are introduced, and we verify the approximation properties.
The remarkable fact is that we are allowed to just sum the ``local'' approximating functions, restricted to the original subdomains. Indeed, no additional jump is created on the relative boundaries between any square $Q_j$ and $B_0$, while the relative boundary between $Q_j^1$ and $Q_j^2$ correspond to a jump of the original displacement $u$, so that here we are allowed to still have jump. A minor point is to set the approximating function as 0 in a small neighbourhood of the intersection between $\partial Q_j$ and the small strip that contains the main jump in $Q_j$, in which the function is reflected.
\newline
\\
{\bf Approximation of $J_u$ and $\partial \Omega$.}
Since $J_u$ is $\left(\hn, n{-}1\right)$-rectifiable, there exists a sequence $\Gamma_i$ of $C^1$ 
hypersurfaces  such that $\hn(J_u\setminus \bigcup_{i=1}^{\infty} \Gamma_i)=0$. 
For each $i\geq 1$, let
\begin{equation*}
S_i:=\Bigg\{x\in J_u\cap \Gamma_i\setminus \bigcup_{j<i}S_j\colon \lim_{\varrho \to 0}\frac{\hn(J_u\cap \ol Q(x,\varrho))}{(2 \varrho)^{n-1}}= \lim_{\varrho\to 0}\frac{\hn(J_u\cap \Gamma_i\cap \ol Q(x,\varrho))}{(2 \varrho)^{n-1}}=1\Bigg\}\,,
\end{equation*} 
where $\ol Q(x,\varrho)$ is the closed cube with center $x$, sidelength $2\varrho$, and one face normal to $\nu(x)$, the normal to $\Gamma_i$ at $x$.
Thus $\hn(J_u\setminus \bigcup_{i=1}^\infty S_i)=0$ and for every $x\in S_i$
\begin{equation*}
\lim_{\varrho\to 0^+}\frac{\hn(J_u\cap \ol Q(x,\varrho)\setminus \Gamma_i)}{(2\varrho)^{n-1}}=0\,.
\end{equation*}
Let us fix $\varepsilon>0$. 
Then for every $x\in S_i$ there exists $\ol \varrho(x)$ such that for $0<\varrho<\ol \varrho(x)$
\begin{equation}\label{0106172209}
\begin{split}
 \hn\big((J_u&\triangle \Gamma_{i})\,\cap \, \ol{Q}(x,\varrho)\big)< \varepsilon (2\varrho)^{n-1}< \, \frac{\varepsilon}{1-\varepsilon}  \hn(J_u\cap \ol{Q}(x,\varrho))\,, \\
 \ol{Q}(x,\varrho)\cap \Gamma_{i} &\text{ lies (in the open region) between the hyperplanes }T_x\pm (\varepsilon \varrho)\nu(x) \,,
\end{split}
\end{equation}
where $T_x$ is the hyperplane normal to $\nu(x)$ and passing through $x$, 
\begin{equation}\label{1409181117}
\hn(J_u\cap \partial \ol{Q}(x,\varrho))=0\,,
\end{equation}
and $\Gamma_i$ is a graph with respect to the direction $\nu(x)$ of Lipschitz constant less than $\varepsilon$.

The family $\mathcal{V}:=\{\ol Q(x,\varrho)\colon x \in J_u,\,0<\varrho<\ol\varrho(x)\}$ is a Vitali class of closed sets for $J_u$. Then, by \cite[Theorem~1.10]{Fal85} for $s=n{-}1$, there exists a disjoint sequence $\ol{Q_j}=\ol Q(x_j,\varrho_j)\subset \mathcal{V}$ such that $\hn\big(J_u\setminus \bigcup_{j=1}^{\infty} \ol{Q_j}\big)=0$. In particular, one face of $\ol{Q_j}$ is normal to $\nu_u(x_j)$, the normal to $J_u$ at $x_j$, 
for each $j$ there exists $i_j$ for which $\Gamma_{i_j}$ separates $Q_j$ in exactly two components $Q_j^1$ and $Q_j^2$ (each of the two is an open Lipschitz domain), and, for a suitable $ \overline{\jmath} \in \mathbb{N}$, we have 
 \begin{subequations}\label{eqs:2405171202}
\begin{align} 
&\hn\Big(J_u \setminus \bigcup_{j=1}^{\overline{\jmath}}Q_j\Big)<\varepsilon\,, \label{1305171147}\\
 \hn\big((J_u\triangle \Gamma_{i_j})\,\cap &\, \ol{Q_j}\big)< \varepsilon (2\varrho_j)^{n-1}< \, \frac{\varepsilon}{1-\varepsilon}  \hn(J_u\cap \ol{Q_j})\,, \label{1305171150}\\
Q_j\cap \Gamma_{i_j} \subset R_j:=\Big\{ x_j+ \sum_{i=1}^{n-1} & y_i\,b_{j,i} + y_n\, \nu_u(x_j) \colon  y_i \in (-\varrho_j,  \varrho_j),\,  y_n \in (-\varepsilon \varrho_j, + \varepsilon \varrho_j)\Big\} \,, \label{2105171916} 
 \end{align}
 \end{subequations}
 where 
 \[
 (b_{j,i})_{i=1}^{n-1}\quad\text{ is an orthonormal basis of }\nu_u(x_j)^\perp\,.
 \]
  Moreover, by \eqref{1409181117} we have
  \begin{equation}\label{1409181118}
 \hn(J_u \cap \partial Q_j)=0\,,
\end{equation}   
and we may assume that $\ol{Q_j}\subset \Omega$ for $j=1, \dots, \overline{\jmath}$.

We can argue similarly for $\partial \Omega$ in place of $J_u$, to find a finite set of cubes $(Q_{h,0})_{h=1}^{\ol h}$ of centers $x_{h,0}\in \partial\Omega$ and sidelength $2\varrho_{h,0}$,  with one face normal to $\nu_\Omega(x_{h,0})$ (the outer normal to $\Omega$ at $x_{h,0}$), pairwise disjoint and with empty intersection with any $\ol Q_j$, and $C^1$ hypersurfaces $(\Gamma_{h,0})_{h=1}^{\ol h}$ with $x_{h,0}\in \Gamma_{h,0}$, such that 
\begin{subequations}\label{eqs:0707172255}
\begin{align}
\hn\Big(\partial & \Omega\setminus  \bigcup_{h=1}^{\ol h} Q_{h,0}\Big) <\varepsilon\,,\label{2105171956}\\
 \hn\big((\partial\Omega\triangle \Gamma_{h,0})\cap \ol{Q}_{h,0} \big)&  <\varepsilon (2\varrho_{h,0})^{n-1}  < \frac{\varepsilon}{1-\varepsilon}\hn(\partial\Omega\cap \ol{Q}_{h,0})\,,  \label{2505171242}\\
  Q_{h,0} \cap \Gamma_{h,0}  \subset R_{h,0}:=\Big\{x_{h,0}+\sum_{i=1}^{n-1} y_i\, b_{h,i}^0 & +y_n\, \nu_h^0 \colon y_i\in (-\varrho_{h,0},\varrho_{h,0}),\, y_n \in (-\varepsilon \varrho_{h,0}, +\varepsilon \varrho_{h,0}) \Big\} \,,\label{2505171243}
\end{align}
\end{subequations}
where 
$\nu_{h,0}=-\nu_{\Omega}(x_{h,0})$ is the generalised outer normal to $\Omega$ at $x_{h,0}$ and
\[
(b_{h,i}^0)_{i=1}^{n-1}\quad\text{ is an orthonormal basis of }(\nu_{h,0})^\perp\,.
\]
 We remark that
 we may assume that conditions \eqref{eqs:2405171202} and \eqref{eqs:0707172255} hold also for 
 the enlarged cubes
 \begin{equation}\label{1409181348}
 \widetilde{Q}_j:= \ol Q_j + (-t,t)^n\,,\qquad \widetilde{Q}_{h,0}:= \ol Q_{h,0} + (-t,t)^n\,,
 \end{equation}
 for some $t$ much smaller than $\varepsilon$ and $\min_{j,h}\{\varrho_j, \varrho_h^0\}$ (we will consider below a parameter $k$ chosen such that $k^{-1}$ is much smaller than $t$). 
Let
\begin{equation}\label{eq:defb0}
B_0:= \Omega\setminus \Big(\bigcup_{j=1}^{\overline{\jmath}} \ol{Q_j} \cup \bigcup_{h=1}^{\ol h} \ol Q_{h,0}\Big)\,.
\end{equation}
Notice that, by the assumptions on $\Omega$, we have that the extension of $u$ with the value 0 outside $\Omega$ is $GSBD^p(\widetilde{\Omega})$ for every open set $\widetilde{\Omega}$ in which $\Omega$ is compactly contained. 
An equivalent point of view would be to include
(a part of) $\partial \Omega$ in the jump set of such an extension of $u$. We employ this extension, not relabeled, when we deal with cubes $Q_{h_0}$, that are not contained in $\Omega$. 
\begin{figure}[h]\label{fig}
\vspace{-9em}
\hspace{-2em}
\begin{minipage}[c]{0.49\linewidth}
\includegraphics[width=\linewidth]{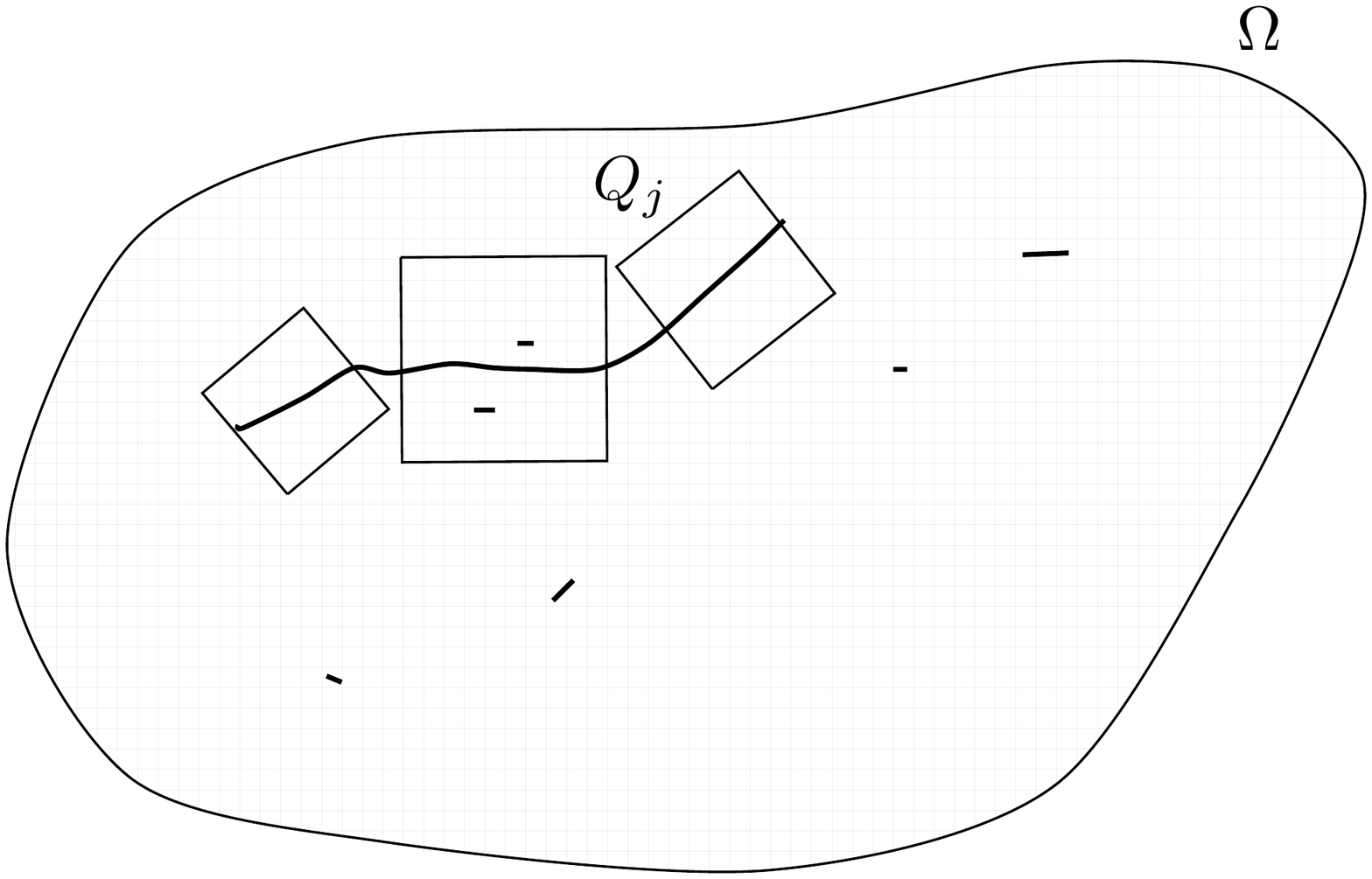}
\end{minipage}
\hfill
\begin{minipage}[c]{0.53\linewidth}
\includegraphics[width=\linewidth]{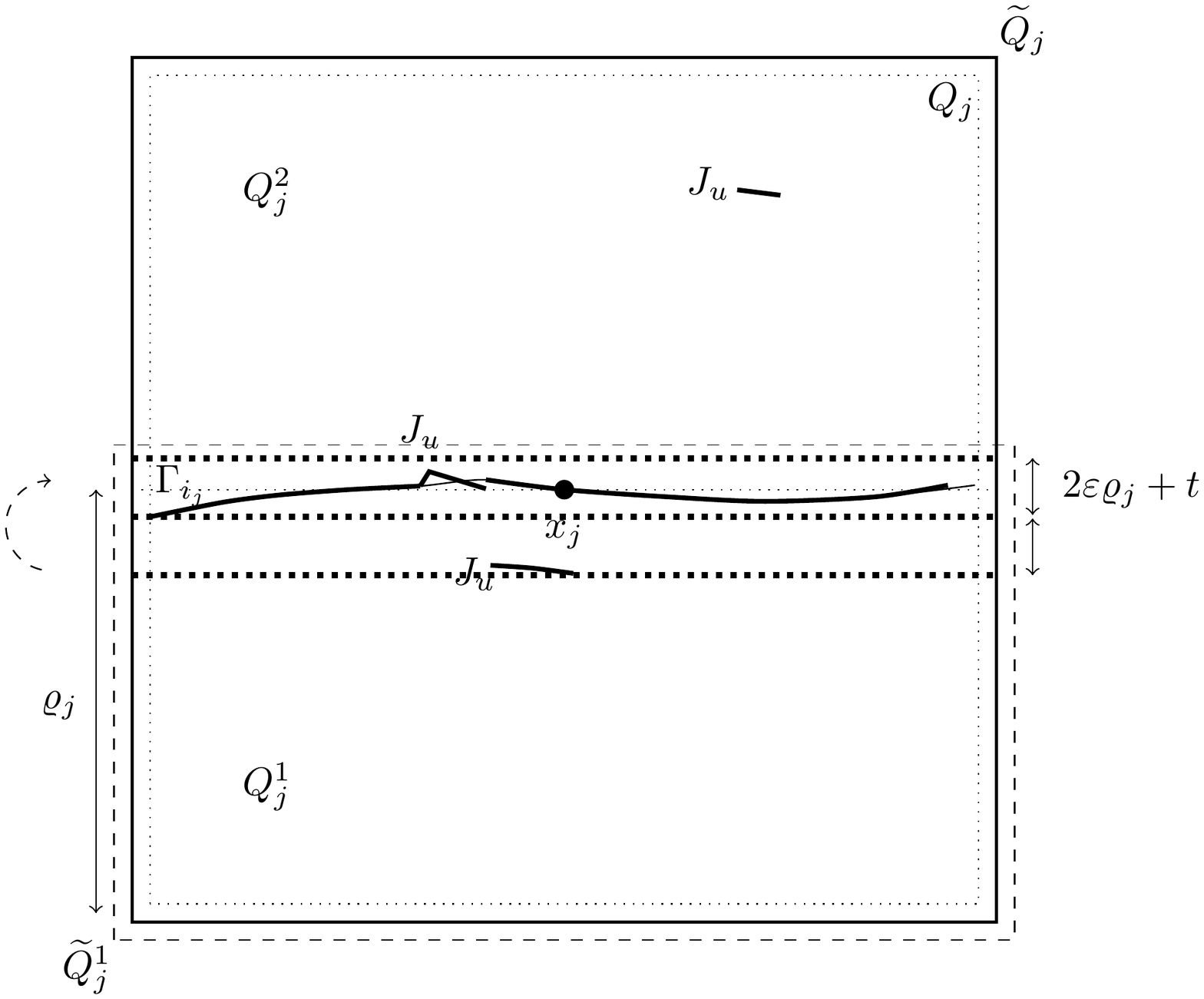}
\end{minipage}
\vspace{-8em}
\caption*{In the first figure, the cubes $Q_j$ covering almost all $J_u$ (see the short curves in $J_u$ in remaining part) and the small cubes $q_z$, employed for the rough approximation (see Theorem~\ref{teo:rough}). In the second one, a zoom on a cube $Q_j$ (dotted), of sidelength $2\varrho_j$, divided by $\Gamma_{i_j}$ (thin, while $J_u$ is thick) in two parts $Q_j^1$, $Q_j^2$. In the neighbourhood $\widetilde{Q}_j$, on the bottom (almost) half $\widetilde{Q}_j^1$ the function $u_j^1$ is obtained by the extension argument \emph{à la} Nitsche in the rectangle $\widehat{R}_j^1=\widetilde{R}_j^1 \cup (\widetilde{R}_j^1)'$ (dashed, in the middle), and by $u$ itself in the remaining part. Notice that $\Gamma_{i_j} \subset (\widetilde{R}_j^1)'$ and that $\Gamma_{i_j}\triangle J_u$ is small.}
\end{figure}
\\
{\bf Definition and properties of the approximating functions in subdomains.}
Let us fix $j\in \{1,\dots, \overline{\jmath}\}$, corresponding to a cube $Q_j$. Consider, for a 
fixed $t>0$ as above, the ``enlarged (almost) half cubes''
\begin{align*}
\widetilde{Q}_j^1:=\Big\{x_j + \sum_{i=1}^{n-1} y_i \, b_{j,i}  + y_n \, \nu_u(x_j) \colon y_i\in (-\varrho_j-t,+\varrho_j+t),\, y_n \in (-\varrho_j-t, \varepsilon \varrho_j+t)\Big\} \,,\\
\widetilde{Q}_j^2:=\Big\{x_j + \sum_{i=1}^{n-1} y_i \, b_{j,i}  + y_n \, \nu_u(x_j) \colon y_i\in (-\varrho_j-t,+\varrho_j+t),\, y_n \in ( -\varepsilon \varrho_j-t, \varrho_j+t)\Big\}\,,
\end{align*}
such that $\widetilde{Q}_j=\widetilde{Q}_j^1+\widetilde{Q}_j^2$, the open rectangles 
\begin{align*}
\widetilde{R}_j^1&:=\Big\{x_j + \sum_{i=1}^{n-1} y_i \, b_{j,i} + y_n \, \nu_u(x_j) \colon y_i\in (-\varrho_j - t,+\varrho_j + t),\, y_n \in (-3\varepsilon \varrho_j-t, -\varepsilon \varrho_j)\Big\}\,,\\
\widetilde{R}_j^2&:=\Big\{x_j + \sum_{i=1}^{n-1} y_i \, b_{j,i} + y_n \, \nu_u(x_j) \colon y_i\in (-\varrho_j - t,+\varrho_j + t),\, y_n \in (\varepsilon \varrho_j, 3\varepsilon \varrho_j+t)\Big\}\,,
\end{align*}
and their reflections with respect to one of their faces
\begin{align*}
(\widetilde{R}_j^1)'&:=\Big\{x_j + \sum_{i=1}^{n-1} y_i \, b_{j,i}  + y_n \, \nu_u(x_j) \colon y_i\in (-\varrho_j - t,+\varrho_j + t),\, y_n \in ( -\varepsilon \varrho_j, \varepsilon \varrho_j+ t)\Big\}\,,\\
(\widetilde{R}_j^2)'&:=\Big\{x_j + \sum_{i=1}^{n-1} y_i \, b_{j,i}  + y_n \, \nu_u(x_j) \colon y_i\in (-\varrho_j - t,+\varrho_j + t),\, y_n \in (-\varepsilon \varrho_j-t,\varepsilon \varrho_j)\Big\}\,.
\end{align*}
(notice that the labels for these sets match with the ones for the cubes $Q_j$).

Let also $\widehat{R}_j^l$ be the union of $\widetilde{R}_j^l$, $(\widetilde{R}_j^l)'$, and their common face, for $l=1,2$.
We have that
\begin{equation}\label{1409181259}
\mathcal{L}^n(\widetilde{R}_j^l)=\mathcal{L}^n\big((\widetilde{R}_j^l)'\big)=(2\varrho_j+2t)^{n{-}1} (2\varepsilon \varrho_j + t)= \varepsilon \, 2^n  \varrho_j^n + O(t)\,,
\end{equation}
where $\lim_{t\to 0} O(t)=0$. Moreover, we define
\begin{equation}\label{1409181028}
\widetilde{R}_j:= \widetilde{R}_j^1 \cup \widetilde{R}_j^2\,. 
\end{equation}

By Lemma~\ref{le:Nitsche}, we may extend the restrictions of $u$ to $\widetilde{R}_j^1$ and $\widetilde{R}_j^2$ by two functions $\widehat{u}_{j}^1\in GSBD^p(\widehat{R}_j^1)$ and $\widehat{u}_{j}^2\in GSBD^p(\widehat{R}_j^2)$ such that for $l=1,2$
\begin{subequations}\label{eqs:2405171953}
\begin{align}
\int\limits_{\widehat{R}_j^l} |e(\widehat{u}_{j}^l)|^p \dx &\leq c \int\limits_{\widetilde{R}_j^l} |e(u)|^p\dx\,,\label{2405172005}\\
\hn(J_{\widehat{u}_{j}^l}\cap \widehat{R}_j^l)&\leq c \, \hn(J_u\cap \widetilde{R}_j^l)\,,\label{2405172007}
\end{align}
where $c>0$ depends only on $n$ and $p$.
Recalling the definition of reflection in Lemma~\ref{le:Nitsche}, it is immediate to see that if $\psi(|u|)\in L^1(\Omega)$, for $\psi$ as in the statement of the theorem, then
\begin{equation}\label{1007170916}
\int \limits_{\widehat{R}_j^l} \psi(|\widehat{u}_{j}^l|) \dx \leq c \int\limits_{\widetilde{R}_j^l} \psi(|u|) \dx\,.
\end{equation}
\end{subequations}

We then introduce, for $l=1,2$, the functions
\begin{equation}\label{1007170937}
u_{j}^l:= u \, \chi_{\widetilde{Q}_j^l \setminus (\widetilde{R}_j^l)'}+ \widehat{u}_j^l \, \chi_{(\widetilde{R}_j^l)'}\,.
\end{equation}

By \eqref{eqs:2405171953}, it holds that
\begin{subequations}\label{eqs:2505171144}
\begin{align}
\mathcal{L}^n(\{u_{j}^l\neq u\} \cap \widetilde{Q}_j^l)& \leq \mathcal{L}^n\big((\widetilde{R}_j^l)'\big) \leq \varepsilon \, 2^n \varrho_j^n + O(t)\,,\label{2505171147}\\
\int\limits_{\widetilde{Q}_j^l } |e(u_{j}^l)|^p\dx &\leq \int\limits_{\widetilde{Q}_j^l \sm \widehat{R}_j^l} |e(u)|^p \dx + c \int\limits_{\widetilde{R}_j^l} |e(u)|^p\dx\,,\label{2505171148}\\
\hn(J_{u_{j}^l} \cap \widetilde{Q}_j^l) &\leq \hn(J_u\cap \widetilde{Q}_j^l \sm \widehat{R}_j^l)+ c \,\hn(J_u\cap \widetilde{R}_j^l) \,.\label{2505171149}
\end{align}
\end{subequations}
We make the same construction starting from the cubes $Q_{h,0}$ in place of $Q_j$, to get $\widetilde{Q}_{h,0}^l$, $\widetilde{R}_{h,0}^l$, $\widetilde{R}_{h,0}$ $(\widetilde{R}_{h,0}^l)'$, $\widehat{R}_{h,0}^l$, $\widehat{u}_{h,0}^l$, $u_{h,0}^l$ in place of $\widetilde{Q}_j^l$, $\widetilde{R}_j^l$, $\widetilde{R}_j$ $(\widetilde{R}_j^l)'$, $\widehat{R}_j^l$, $\widehat{u}_j^l$, $u_j^l$, for $l=1,2$. In this way, \eqref{eqs:2505171144} hold with the corresponding terms for the cubes $Q_{h,0}$. Notice that here we start from the extension of $u$ with value 0 outside $\Omega$: so we have to include in the analogue of \eqref{2505171149} also the contribution due to $\partial \Omega$, and then \eqref{2505171149} reads in this case as
\begin{equation}\label{1209181203}
\hn(J_{u_{h,0}^l} \cap \widetilde{Q}_{h,0}^l) \leq \hn\big((J_u \cup \partial\Omega) \cap \widetilde{Q}_{h,0}^l \sm \widehat{R}_{h,0}^l\big)+ c \,\hn\big((J_u \cup \partial\Omega)\cap \widetilde{R}_{h,0}^l\big) \,.
\end{equation}

As for $B_0$, we set
\[
\widetilde{B}_0:=B_0+B(0,t)\,,
\]
 and consider the function $u \, \chi_{\widetilde{B}_0}$, where $u$ is set equal to 0 outside $\Omega$.

Let us introduce for any $j$ and $h$, and $l=1,2$, suitable rectangles $(Q_j^l)'$ and $(Q_{h,0}^l)'$ such that
\[
Q_j^l \subset \subset (Q_j^l)' \subset \subset \widetilde{Q}_j^l\,,\qquad Q_{h,0}^l \subset \subset (Q_{h,0}^l)' \subset \subset \widetilde{Q}_{h,0}^l\,,
\]
and apply Theorem~\ref{teo:rough}  in correspondence to the functions $u_j^1$, $u_j^2$, $u_{h,0}^1$, $u_{h,0}^2$, $u \chi_{\widetilde{B}_0}$, the sets $(Q_j^1)'$, $(Q_j^2)'$, $(Q_{h,0}^1)'$, $(Q_{h,0}^2)'$, $B_0$ and their neighbourhoods, denoted with a tilde.
More precisely, for $j=1,\dots, \ol\jmath$ and $l=1,2$ we define $(\widetilde{u}_k)_j^l$ as in \eqref{eq:defappr1}, starting from the bad and good cubes, with sidelength of order $k^{-1}$, for the reference sets $(Q_j^l)'$, $\widetilde{Q}_j^l$,
and set
\begin{equation}\label{1209181228}
(u_k)_j^l:= \big((\widetilde{u}_k)_j^l \ast \varphi_k \big) \, \chi_{(Q_j^l)' \sm (\widetilde{Q}_j^l)_b^k}\,,
\end{equation}
where $\varphi_k$ is the mollifier as in Theorem~\ref{teo:rough} and $(\widetilde{Q}_j^l)_b^k$ is the union of bad cubes for $\widetilde{Q}_j^l$ (see the definition of $\widetilde{\Omega}_b^k$ in Theorem~\ref{teo:rough}). Moreover, we have the exceptional sets $(E_k)_j^l$, defined as in \eqref{1209181245}. 
We argue in the very same way for the other sets $(Q_{h,0}^1)'$, $(Q_{h,0}^2)'$, $B_0$, to get the functions $(u_k)_{h,0}^1$, $(u_k)_{h,0}^2$, $(u_k)_0$, and the exceptional sets $(E_k)_{h,0}^1$, $(E_k)_{h,0}^2$, $(E_k)_0$. 

We remark that the construction is the same regardless of the starting subdomain, since the bad and good nodes belong to $(2k^{-1}) \Z^n$, which is fixed for any given $k$.

By Theorem~\ref{teo:rough} we then obtain that
\begin{subequations}\label{eqs:1505171913}
\begin{align}
 \lim_{k\to\infty}\mathcal{L}^n((E_k)_j^l)=\lim_{k\to \infty} & \int \limits_{Q_j^l\setminus (E_k)_j^l } \hspace{-1em}|(u_k)_j^l -u_j^l|^p \dx = 0\,, \label{1roughmain}\\
 \limsup_{k\to \infty}  \int \limits_{Q_j^l} |e((u_k)_j^l)|^p \dx  & \leq  \int \limits_{\widetilde{Q}_j^l} |e(u_j^l)|^p\dx \,, \label{2roughmain}\\
\hn\big(J_{(u_k)_j^l}\cap (Q_j^l)'\big) & \leq C \,\theta^{-1} \hn\big(J_{u_j^l} \cap \widetilde{Q}_j^l)\,,\label{3roughmain}\\
\lim_{k\to \infty} \int \limits_{Q_j^l} \psi(|(u_k)_j^l-u_j^l|)\dx&=  0\,,\quad\text{ if }\psi(|u|)\in L^1(\Omega)\,,\label{4roughmain}
\end{align}
\end{subequations}
and the same for $(u_k)_{h,0}^l$, $(u_k)_0$, and $(E_k)_{h,0}^l$, $(E_k)_0$ in place of $(u_k)_j^l$ and $(E_k)_j^l$.
\newline
\\
{\bf The approximating functions.}
We define

\begin{equation*}\label{eq:defapproxmain}
u_k:= \sum_{j=1}^{\overline{\jmath}} \big((u_k)_j^1\, \chi_{Q_j^1} + (u_k)_j^2\, \chi_{Q_j^2}\big) + \sum_{h=1}^{\ol h} \big((u_k)_{h,0}^1\, \chi_{Q_{h,0}^1 \cap \Omega} + (u_k)_{h,0}^2\, \chi_{Q_{h,0}^2\cap \Omega } \big)  + (u_k)_0 \, \chi_{B_0}\,.
\end{equation*}

We are going to prove the desired approximation properties for the sequence $u_k$. It is immediate that $u_k\in SBV^p(\Omega;\Rn)\cap L^\infty(\Omega; \Rn)$. 
\newline
\\
{\bf Proof of \eqref{1main}, \eqref{2main}, \eqref{3main}.}
In order to describe $J_{u_k}$, notice that for $l=1,2$
\begin{equation}\label{2505171041}
J_{(u_k)_j^l}\cap \partial Q_j \setminus F_j = J_{u_k} \cap \partial Q_j \setminus F_j\,,
\end{equation}
where  (recall \eqref{1409181028})
\begin{equation*}
F_j:= (\partial Q_j \cap   \widetilde{R}_j) + B(0, 10 \sqrt{n} k^{-1})\,.
\end{equation*}
Indeed, notice that the values of $(\widetilde{u}_k)_j^l$, defined as in \eqref{eq:defappr1}, are determined in any cube $\Qz$ (recall \eqref{2105171916}) by the values of $u_j^l$ in $\Qz$. Moreover, $(u_k)_j^l(x)=0$ if $x$ is in the union of the bad cubes $(\widetilde{Q}_j^l)_b^k$, while if $x \notin (\widetilde{Q}_j^l)_b^k$, then $(u_k)_j^l(x)$ depends on $(\widetilde{u}_k)_j^l$ in $B(x, k^{-1})$. The same holds also for $(\widetilde{u}_k)_0$ and $(u_k)_0$.
Now, if $x\in \partial Q_j^l \setminus F_j$, then for any cube $\Qz$ such that $B(x, 2 k^{-1}) \cap \Qz \neq \emptyset$ we have that
$u_j^l=u_0=u$ in $Q_z^k$, and therefore $(u_k)_j^l=(u_k)_0$ in a neighbourhood of $x$, so that $J_{u_k}(x)=0$, unless $x \in J_{(u_k)_j^l} \cap J_{(u_k)_0}$. Thus \eqref{2505171041} is proven. We may reproduce this argument for the cubes $Q_{h,0}$, to get
\begin{equation*}
J_{(u_k)_{h,0}^l} \cap \Omega \cap \partial Q_{h,0} \sm F_{h,0}= J_{u_k} \cap \Omega \cap \partial Q_{h,0} \sm F_{h,0}\,,\qquad F_{h,0}:= (\partial Q_{h,0}  \cap \widetilde{R}_{h,0}) + B(0, 10 \sqrt{n} k^{-1})\,.
\end{equation*}

Since $\hn\big(\partial Q_j \cap  \widetilde{R}_j  \big)= 2^{n+1} \varepsilon \, \varrho_j^{n-1}$ and  $\hn\big(\partial Q_{h,0} \cap  \widetilde{R}_{h,0}  \big)= 2^{n+1} \varepsilon \, \varrho_{h,0}^{n-1}$,  for $k$ large we have that
\begin{equation}\label{2505171111}
\hn\big(\partial Q_j \cap F_j\big) < C\, \varepsilon \, \varrho_j^{n-1}\,,\qquad  \hn\big(\partial Q_{h,0} \cap F_{h,0}\big) < C\, \varepsilon \, \varrho_{h,0}^{n-1} \,.
\end{equation}
By \eqref{1409181118} and \eqref{2505171041}
we deduce that 
\begin{equation*}\label{2505171923}
\begin{split}
J_{u_k}\subset  J_{\mathrm{int}} \cup \bigcup_{j=1}^{\overline{\jmath}} \big( (Q_j\cap \Gamma_{i_j}) \cup  (\partial Q_j \cap F_j)  \big)
\cup \bigcup_{h=1}^{\overline{h}} \big( (Q_{h,0}\cap \Gamma_{{h,0}}) \cup  (\partial Q_{h,0} \cap F_{h,0})  \big) \cap \Omega
\,.
\end{split}
\end{equation*}
where we have set
\begin{equation}\label{1409181437}
J_{\mathrm{int}}:= (J_{(u_k)_0} \cap B_0) \cup \bigcup_{j=1}^{\overline{\jmath}} \big((J_{(u_k)_j^1} \cap (Q_j^1)') \cup (J_{(u_k)_j^2} \cap (Q_j^2)') \big) \cup \bigcup_{h=1}^{\overline{h}} \big( (J_{(u_k)_{h,0}^1}\cap (Q_{h,0}^1)') \cup (J_{(u_k)_{h,0}^2} \cap (Q_{h,0}^2)')) \big) 
\end{equation}
By Theorem~\ref{teo:rough}, the jump of each $(u_k)_j^l$, $(u_k)_{h,0}^l$, $(u_k)_0$ is contained in a finite union of boundaries of cubes (so this holds for $J_{\mathrm{int}}$), and these functions are Lipschitz (with all their derivatives) up to their jump set. Therefore $J_{u_k}$ is closed and included in a finite union of $C^1$ hypersurfaces, and $u_k$ is Lipschitz (with all its derivatives) up to $J_{u_k}$. 

Notice that we may assume that $H:=\bigcup_j (Q_j\cap \Gamma_{i_j}) \subset J_{u_k}$.
Indeed, we can find $a >0$ arbitrarily small such that 
$\hn( H \cap \{x \colon [u_k](x)=a\})=0$ 
(with $[u_k](x)$ the size of the jump of $u_k$),
and then we can add to $u_k$ a perturbation with arbitrarily small $W^{1,\infty}(\Omega\sm H)$ norm, 
which is equal to $a$ on an arbitrarily large subset of $H$.

In particular,
\begin{equation}\label{0206171134}
\begin{split}
J_{u_k}\triangle J_u \subset J_{\mathrm{int}}  \cup (J_u\sm \bigcup_j \ol Q_j)  & \cup \bigcup_{j=1}^{\overline{\jmath}} \Big( \big((J_u \triangle \Gamma_{i_j}) \cap Q_j\big) \cup  (\partial Q_j \cap F_j)  \Big)
\\&
\cup \bigcup_{h=1}^{\overline{h}} \Big( (Q_{h,0}\cap \Gamma_{{h,0}}) \cup  (\partial Q_{h,0} \cap F_{h,0})  \Big) \cap \Omega
\,.
\end{split}
\end{equation}
We introduce the sets $\widetilde{R}$, defined as the union of the sets where $u_j^l$, $u_{h,0}^l$ differs from $u$, and $\widetilde{E}_k$, as the union of the exceptional sets in the rough approximations, starting from $u_j^l$, $u_{h,0}^l$, $u \chi_{\widetilde{B}_0}$. Then
\begin{equation}\label{1409182158}
\begin{split}
 \widetilde{E}_k & := (E_k)_0 \cup\bigcup_{j=1}^{\overline{\jmath}} \big( (E_k)_j^1\cup (E_k)_j^2 \big)  \cup\bigcup_{h=1}^{\overline{h}} \big( (E_k)_{h,0}^1\cup (E_k)_{h,0}^2 \big)  \,,
 \\ 
  \widetilde{R} & := \bigcup_{j=1}^{\ol \jmath} ( \widetilde{R}_j \cap Q_j ) \cup  \bigcup_{h=1}^{\ol h} (\widetilde{R}_{h,0} \cap Q_{h,0} \cap \Omega )  \,.
 \end{split}
\end{equation}
By \eqref{1roughmain} and \eqref{2505171147} (and its analogue for $(\widetilde{R}_{h,0}^l)'$, recall that the $Q_j$ and the $Q_{h,0}$ are pairwise disjoint), we have that
\begin{equation}\label{2505171940}
\lim_{k\to \infty}\mathcal{L}^n(\widetilde{E}_k)=0\,,\qquad \mathcal{L}^n(\widetilde{R})\leq C\,\varepsilon \,.
\end{equation}
We define
\begin{equation*}
E_k:= \widetilde{E}_k \cup \widetilde{R}\,.
\end{equation*}
Then $\lim_{k} \mathcal{L}^n(E_k) < C \, \varepsilon $.
It follows in particular that
\begin{equation}\label{0807172154}
\lim_{\varepsilon\to 0} \int \limits_{\widetilde{R}} |e(u)|^p\dx =0\,.
\end{equation}
Let us now put together \eqref{eqs:2505171144} with \eqref{eqs:1505171913}, corresponding to the cubes $Q_j$,  and their analogues corresponding to the cubes $Q_{h,0}$ and to $B_0$. 
According to the definition of $u_k$, we sum all these estimates. Notice that the $L^p$ norm of the symmetrised gradient of any rough approximation is controlled by that one of the starting function in a neighbourhood of the reference set (to fix the ideas, we have that $\|e((u_k)_j^l)\|_{L^p(Q_j^l)}$ is controlled by $\|e(u_j^l)\|_{L^p(\widetilde{Q}_j^l)}$, and the same for the cubes $Q_{h,0}$).  Then in the estimates we count twice the energy on $\widetilde{H}_t:= \bigcup_j (\widetilde{Q}_j \sm Q_j) \cup \bigcup_h (\widetilde{Q}_{h,0} \sm Q_{h,0})$. Nevertheless, recalling \eqref{1409181348}, we may ensure that
\begin{equation}\label{1409181349}
\lim_{t\to 0}  \int \limits_{\widetilde{H}_t}   |e(u)|^p \dx = 0\,. 
\end{equation}
Notice that we do not treat in this way the jump part, since it is not important to count it twice in the following estimates.

Thus we obtain that (we take into account also \eqref{1209181203}) 
\begin{subequations}
\begin{align}
\lim_{k\to \infty} \int\limits_{\Omega\setminus E_k} &|u_k-u|^p\dx =0\,,\label{2505171819}\\
\limsup_{k\to \infty} \int\limits_\Omega |e(u_k)|^p \dx &\leq \int \limits_\Omega |e(u)|^p \dx+ c \int \limits_{\widetilde{R} \cup \widetilde{H}_t} |e(u)|^p\dx  \,,\label{2505171829}
\end{align}
\begin{equation}\label{2505171924}
\begin{split}
\hn(J_{\mathrm{int}}) \leq  C\, \theta^{-1} \Big[ & \hn(J_u\cap B_0) + \sum_{j=1}^{\overline{\jmath}}\hn\big((J_u \setminus \Gamma_{i_j}) \cap \ol{Q_j}\big)   +\sum_{h=1}^{\overline{h}}\hn\big((\partial \Omega \setminus \Gamma_{h,0} ) \cap Q_{h,0} \big) \Big]\,, 
\end{split}
\end{equation}
\end{subequations}
where we recall the definition of $J_{\mathrm{int}}$ \eqref{1409181437}.
By \eqref{1305171147}, \eqref{1305171150}, \eqref{2105171956}, \eqref{2505171242}, and \eqref{1409181349}  it follows that 
\begin{equation}\label{1409181430}
\hn(J_{\mathrm{int}})< C\, \theta^{-1} \varepsilon\,.
\end{equation}
Above, we have used that \eqref{1305171150} and \eqref{2505171242} imply 
\begin{equation}\label{2505171920}
\sum_{j=1}^{\overline{\jmath}}\hn\big((J_u \triangle \Gamma_{i_j}) \cap \ol{Q}_j\big) <  C \, \varepsilon \, \hn(J_u)\,,\quad \sum_{h=1}^{\overline{h}}\hn\big((\partial\Omega \triangle \Gamma_{h,0}) \cap \ol{Q}_{h,0}\big) <  C \, \varepsilon \, \hn(\partial\Omega)\,,
\end{equation}
since the cubes $Q_j$ and $Q_{h,0}$ are pairwise disjoint.
Therefore, collecting \eqref{0206171134} with \eqref{1305171147}, \eqref{2505171111}, \eqref{1409181430}, and \eqref{2505171920},
we get
\begin{equation}\label{2505172006}
\hn(J_{u_k}\triangle J_u)< C\,\theta^{-1}\, \varepsilon  \,.
\end{equation}
By \eqref{2505171940}, \eqref{0807172154}, \eqref{1409181349}, \eqref{2505171819}, \eqref{2505171829}, \eqref{2505172006}, 
and by the arbitrariness of $\varepsilon$ and $t$, we get \eqref{1main}, \eqref{3main}, and 
\begin{equation}\label{1505172318}
  \limsup_{k\to \infty} \|e(u_k)\|_{L^p(\Omega;\Mnn)}\leq \|e(u)\|_{L^p(\Omega;\Mnn)}\,.
 \end{equation} 
Moreover, \eqref{1main} gives that $u_k\to u$ in measure, and then, by \cite[Remark~2.2]{FriSol16}, there exists a subsequence of $u_k$, not relabelled, and a nonnegative, increasing, concave function $\ol\psi$ such that
\begin{equation*}
\lim_{s\to +\infty} \ol\psi(s) = +\infty
\end{equation*}
 and
\begin{equation*}
\sup_{k \in \mathbb{N}} \int\limits_\Omega \ol\psi(|u_k|) \dx \leq 1\,.
\end{equation*} 
Therefore we can apply the Compactness Theorem for $GSBD^p$ \cite[Theorem~11.3]{DM13}, which implies that, up to a further subsequence,
 \begin{equation*}
 e(u_k)\rightharpoonup e(u) \quad\text{in }L^p(\Omega;\Rn)\,.
 \end{equation*}
  Therefore, by \eqref{1505172318}, the sequence $u_k$ satisfies also \eqref{2main}.
\newline
\\
{\bf Proof of \eqref{4main}.}
Fix $j\in \{1,\dots, \overline{\jmath}\}$ and take 
\begin{equation*}
\int \limits_{\Gamma_{i_j}\cap \ol Q_j}\hspace{-1em}\tau(|\tr ((u_k)_j^1- u)|) \dh\equiv \hspace{-0.5em}\int \limits_{\Gamma_{i_j}\cap \ol Q_j}\hspace{-1em}\tau(|u_k^+- u^+|) \dh \,,
\end{equation*}
where the trace is considered from the interior side of $\Gamma_{i_j}$ with respect to $Q_j^1$, and we assume by convention that this is the ``positive'' side of $\Gamma_{i_j}$.
We define the rectangle
\begin{equation*}
\widehat{Q}_j^1:=\Big\{x_j + \sum_{i=1}^{n-1} y_i \, b_{j,i}  + y_n \, \nu_u(x_j) \colon y_i\in (-(1-\sqrt{\varepsilon})\varrho_j, (1-\sqrt{\varepsilon})\varrho_j),\, y_n \in (-\varrho_j, \varepsilon \varrho_j)\Big\} 
\end{equation*}
and call $\xi_n$ the normal $\nu_u(x_j)$. Then $\Gamma_{i_j}\cap \widehat{Q}_j^1$ is a graph in the direction $\xi_n$ of Lipschitz constant less than $\varepsilon$. By \cite[Lemma~3.1]{Bab15}, there exists a universal constant $\eta_0>0$ (indeed it depends decreasingly on the Lipschitz constant of the graph of $\Gamma_{i_j}\cap \widehat{Q}_j^1$ in the direction $\xi_n$, which is less than $1/2$) such that for any $\xi \in \Sn$ with $|\xi-\xi_n|<\eta_0$, one has that $\Gamma_{i_j}\cap \widehat{Q}_j^1$ is a Lipschitz graph in the direction $\xi$. In particular, let $(\xi_1,\dots,\xi_{n-1},\xi_n)$ be a basis of $\Rn$ with $|\xi_h-\xi_n|<\eta_0$. 
Arguing as in \cite[equations (17)--(19)]{Iur14}, 
\begin{equation*}\label{1506171854}
\int \limits_{\Gamma_{i_j}\cap \widehat{Q}_j^1}\hspace{-1em}\tau(|\tr((u_k)_j^1-u)|)\dh \leq C \sum_{h=1}^n \int \limits_{\Gamma_{i_j}\cap \widehat{Q}_j^1} |\tr(\tau(((u_k)_j^1-u)\cdot \xi_h))| \dh
\end{equation*}
for a universal constant $C>0$. Since $\tau(((u_k)_j^1-u)\cdot \xi_h)\in L^1(Q_j^1)$ and $\mathrm{D}_{\xi_h}\tau(((u_k)_j^1-u)\cdot \xi_h)\in \mathcal{M}_b^+(Q_j^1)$ for any $h$, arguing as in \cite[Theorem~3.2, Steps 1 and 4]{Bab15} we deduce that
\begin{equation*}
\begin{split}
\int \limits_{\Gamma_{i_j}\cap \widehat{Q}_j^1} |\tr(\tau(((u_k)_j^1-u)\cdot \xi_h))| \dh \leq &\frac{C}{\sqrt{\varepsilon}\varrho_j} \|\tau(((u_k)_j^1-u)\cdot \xi_h)\|_{L^1(A^{\xi_h}_\varepsilon)}+ C \int \limits_{A^{\xi_h}_\varepsilon} |e((u_k)_j^1-u)|\dx \\ &+C \, \hn(J_{(u_k)_j^1-u}\cap A^{\xi_h}_\varepsilon) \,,
\end{split}
\end{equation*}
for $C>0$ depending only on $n$, and $A^{\xi_h}_\varepsilon:=\{y-s\xi_h\colon y \in \Gamma_{i_j} \cap \widehat{Q}_j^1, \,0<s<{\sqrt{\varepsilon}\varrho_j}\}\subset Q_j^1$.
Being $\tau$ bounded, by \eqref{2505171147} and \eqref{1roughmain} we get that 
\begin{equation*}
\limsup_{k\to \infty}\frac{C}{\sqrt{\varepsilon}\varrho_j} \|\tau(((u_k)_j^1-u)\cdot \xi_h)\|_{L^1(A^{\xi_h}_\varepsilon)}< C\sqrt{\varepsilon}\varrho_j^{n-1}<C \sqrt{\varepsilon}\hn(J_u\cap \ol{Q_j})\,,
\end{equation*}
where the last inequality follows by \eqref{1305171150}.
By construction of $(u_k)_j^1$ in Theorem~\ref{teo:rough} (in particular by \eqref{1506171757}) we deduce that
\begin{equation}\label{1506171841}
C \int \limits_{A^{\xi_h}_\varepsilon} |e((u_k)_j^1-u)|\dx < C \hspace{-2em}\int \limits_{A^{\xi_h}_\varepsilon+B(0, 8k^{-1})} \hspace{-2em}|e(u)|^p\dx\,,
\end{equation}
by \eqref{3roughmain} that
\begin{equation*}\label{1506171842}
\hn(J_{(u_k)_j^1-u}\cap A^{\xi_h}_\varepsilon) < C\theta^{-1} \hn(J_u\cap Q_j^1) <  C \theta^{-1}\hn\big((J_u\setminus \Gamma_{i_j})\cap Q_j\big)\,,
\end{equation*}
and by definition of $\widehat{Q}_j^1$ that
\begin{equation*}
\hn(\Gamma_{i_j}\cap Q_j\setminus \widehat{Q}_j^1) < C (\sqrt{\varepsilon}\varrho_j)^{n-1} < C \sqrt{\varepsilon}^{n-1}\hn(J_u\cap \ol{Q_j})\,.
\end{equation*}
Collecting the informations above, we get (recall that $\tau$ is bounded) that
\begin{equation*}
\begin{split}
\int \limits_{\Gamma_{i_j}\cap \ol Q_j} \hspace{-1em}\tau(|u_k^+- u^+|) \dh < & C \Big(\sqrt{\varepsilon}\hn(J_u\cap \ol{Q_j})+ \sup_{h}  \hspace{-2.5em}\int \limits_{A^{\xi_h}_\varepsilon+B(0, 8k^{-1})} \hspace{-2.5em}|e(u)|^p\dx \\ &+ \theta^{-1}\hn\big((J_u\setminus \Gamma_{i_j})\cap Q_j\big) \Big)\,.
\end{split}
\end{equation*}
We can now argue similarly in $Q_j^2$ and sum over $j$. Recalling \eqref{2505171920} and \eqref{2505172006}, and since $\tau$ is bounded, it follows that
\begin{equation}
\int \limits_{J_{u_k}\cup J_u} \hspace{-1em}\tau(|u_k^\pm-u^\pm|) \dh  < C\,\theta^{-1}\, \varepsilon + c\, \sqrt{\varepsilon} \, \hn(J_u) + \int \limits_{A_\varepsilon} |e(u)|^p\dx\,,
\end{equation}
where $\mathcal{L}^n(A_\varepsilon)\to 0$ as $\varepsilon\to 0$ (for $k$ much smaller than $\varepsilon$). 
 By the arbitrariness of $\varepsilon$
 \begin{equation}\label{1506172001}
 \lim_{k\to \infty} \int \limits_{J_{u_k}\cup J_u} \hspace{-1em}\tau(|u_k^\pm-u^\pm|) \dh=0\,.
 \end{equation}
As for the trace on $\partial\Omega$, we can argue similarly to \eqref{1506172001}, with $(u_k)_{h,0}^1$ and $(u_k)_{h,0}^2$ in place of $(u_k)_j^1$, $(u_k)_j^2$, and employing \eqref{2105171956}, \eqref{2505171242}, \eqref{1209181203}, to conclude \eqref{4main}.
\newline
\\
{\bf Proof of \eqref{5main}.}
Assume that $\psi(|u|)\in L^1(\Omega)$, for $\psi$ as in the statement of the theorem.
Recalling \eqref{1007170937} and \eqref{1409182158}, by \eqref{2505171147} and \eqref{4roughmain} we have (sum all the contributions)
\begin{equation}\label{1007171004}
\lim_{k\to \infty} \int \limits_{\Omega \setminus \widetilde{R}} \psi(|u_k-u|) \dx = 0\,,
\end{equation} 
 and
\begin{equation*}
\lim_{k\to \infty} \int \limits_{(\widetilde{R}_j^l)'} \psi(|(u_k)_j^l-\widehat{u}_j^l|)\dx = 0\,,
\end{equation*}
for every $j$ and $l=1,2$ (and also for the cubes $Q_{h_0}^l$, replacing formally above the subscript $_j$ with the subscript $_{0,h}$).
By \eqref{1007170916} (and its analogue for the cubes $Q_{h_0}^l$) we get
\begin{equation*}
\limsup_{k\to \infty} \int \limits_{\widetilde{R}}  \psi(|u_k|)\dx \leq c\, \limsup_{k\to\infty} \int \limits_{\widetilde{R}} \psi(|u|)\dx\,.
\end{equation*}
Therefore 
\begin{equation*}
\limsup_{k\to \infty} \int \limits_{\widetilde{R}} \psi(|u_k-u|) \dx\leq  C_\psi  \limsup_{k\to \infty} \int \limits_{\widetilde{R}} \big( \psi(|u_k|)+\psi(|u|)\big) \dx \leq c\, C_\psi \, \limsup_{k\to\infty} \int \limits_{\widetilde{R}} \psi(|u|)\dx\,,
\end{equation*}
which vanishes as $\varepsilon$ tends to 0, due to \eqref{2505171940}, since $\psi(|u|)\in L^1(\Omega)$.
Together with \eqref{1007171004}, this proves \eqref{5main} and completes the proof of the theorem.
\end{proof}

\begin{remark} 
 The point of view we have adopted here may be interpreted also as follows: first extend $u$ to a surface with overlapping (one could also see it as a surface with multiplicity, at most 2), then  use a fixed approximation on this surface and restrict to the zones where there is no overlapping (or the multiplicity is 1). 
The construction in Theorem~\ref{teo:main} may be slightly modified also in the following way: apply Theorem~\ref{teo:rough} to suitable compact subsets of  $\Omega\setminus \Big(\bigcup_{j=1}^{\overline{\jmath}}\ol{R}_j \cup \bigcup_{h=1}^{\ol h} \ol{R}_{h,0}\Big)$,  and reflect the smooth function obtained (so without using Lemma~\ref{le:Nitsche}) on both sides of $R_j$ and $R_{h,0}$ with respect to $J_u$ and $\partial\Omega$,
the further arguments being similar to what done above. Working in a compact subset of  $\Omega\setminus \Big(\bigcup_{j=1}^{\overline{\jmath}}\ol{R}_j \cup \bigcup_{h=1}^{\ol h} \ol{R}_{h,0}\Big)$  should permit to have for free an extension of the original function to a larger domain, without employing partitions of the unity.
Arguing in this way, we expect that one could find alternative proofs to our density result, still without assuming that 
$u$ is $p$-summable, using different approximation techniques, such as the one 
in \cite{CFI17Density}.  
\end{remark}
\section{Approximation of brittle fracture energies}\label{Sec4}
Here we show how the density result of Theorem~\ref{teo:main} may be employed to approximate, in the sense of $\Gamma$-convergence, the Griffith energy for brittle fracture, under no assumption on the integrability of the displacement. This is a novelty in the vectorial case, except for $n=2$, where this convergence (for quadratic bulk energy) may be proven starting from the density result \cite[Theorem~2.5]{FriPWKorn}.
In particular, for phase field approximations \emph{à la} Ambrosio-Tortorelli \cite{AmbTorCPAM, AmbTorUMI}, one needs for a density theorem of the type of Theorem~\ref{teo:main} in the $\Gamma$-$\limsup$ inequality; the $\Gamma$-limit is then determined in the subspace of $GSBD$ in which every displacement is approximated by the density result.

On the other hand, since one is interested in the approximation of minimisers for Griffith energy, it is natural to impose some conditions to prevent that the set of minimisers coincides with the constant displacements. Two important examples are Dirichlet boundary condition and a compliance condition for the displacement with respect to a given datum $g$ on the whole $\Omega$.
We show how to approximate the resulting brittle fracture energy, under some geometric assumptions on the Dirichlet part of the domain in the first case, and for a very large class of compliance functions (possibly such that the displacement is not a priori forced to be even integrable) in the second case. Requiring some integrability on displacement in the density theorem forces to include lower order terms in the energy functional, in order to guarantee \emph{a priori} such integrability. 

 We remark that in \cite{CC18} we have recently shown a general compactness result in $GSBD$, that has been there employed to prove
 existence of minimisers for the Dirichlet minimisation problem in any dimension $n\geq 2$, generalising the $2$-dimensional existence result \cite[Theorem~6.2]{FriSol16}. The compactness result is used also here to show, in Theorem~\ref{thm:compmin}, compactness for minimisers of the phase field approximating functionals: up to a subsequence, these converge pointwise outside an exceptional set where the limit displacement could be infinite.

Let us introduce some notation for this section. Let $p$, $q>1$, $a>0$, $\varepsilon_k$, $\eta_k>0$ with
$
\varepsilon_k\to 0$, $\eta_k\to 0$, $\frac{\eta_k}{\varepsilon_k^{p-1}}\to 0$, for $k\in \N$.
Let $W\colon \R\times \Mnn\to [0,\infty)$ be convex in the second argument and lower semicontinuous, nondecreasing with respect to $s$, with $W(s,0)=0$ and
\begin{equation}\label{2607172253}
s(c_1\,|{\cdot}|^p-c_2)\leq W(s,\cdot) \leq s(c_3\,|{\cdot}|^p+c_4)\quad\text{for every }s\in \R\,
\end{equation}
 for some $c_1, \, c_2,\, c_3, \, c_4>0$, and $d\colon [0,1]\to [0,\infty)$ continuous, decreasing, with $d(1)=0$.
 For every bounded open set $A\subset \R^n$ and measurable functions $u\colon A\to \Rn$ and $v\colon A\to [0,1]$, we define 
 \begin{equation*}
 G^A_k(u,v):=
 \begin{dcases}
\int \limits_A \Big(W(v,e(u))+\frac{d(v)}{\varepsilon_k}+a\,\varepsilon_k^{q-1} |\nabla v|^q \Big)\dx \quad &\text{in }W^{1,p}(A;\Rn)\times V_k^A\,,\\
+\infty &\text{otherwise,}
\end{dcases}
 \end{equation*}
 where
 \begin{equation*}
V_k^A:=\{v\in W^{1,q}(A)\colon \eta_k\leq v\leq 1\}\,.
\end{equation*}
 and the \emph{generalised Griffith energy}
\begin{equation*}
G^A(u,v):=
\begin{dcases}
\int\limits_A W(1,e(u))\dx + \alpha\hn(J_u) \quad &\text{in } GSBD^p(A)\times \{v=1 \,\,\mathcal{L}^n\text{-a.e.\ in }A\}\,,\\
+\infty &\text{otherwise,}
\end{dcases}
\end{equation*} 
with 
\[
\alpha:=2(q')^{\frac{1}{q'}}(aq)^{\frac{1}{q}}\int_0^1d(s)^{\frac{1}{q'}}\,\mathrm{d}s,\qquad \frac{1}{q}+\frac{1}{q'}=1\,.
\]

\begin{theorem}\label{teo:gammaconvG}
Let $A\subset \Rn$ be a bounded set with finite perimeter. Then 
$G_k^A$ $\Gamma$-converge to $G^A$ with respect to the topology of the convergence in measure for $u$ and $v$. 
\end{theorem}

\begin{proof}
Being the convergence in measure metrisable, by \cite[Proposition~8.1]{DMLibro} the $\Gamma$ limit of $G_k^A$ is characterised in terms of convergent sequences.
Let us first prove the $\Gamma$-$\liminf$ inequality, following the lines of the proof of \cite[Theorem~8]{Iur14}. We show that if $(u_k,v_k)$ converge in measure to $(u,v)$ and $F_k^A(u_k,v_k)$ is bounded, then $u\in GSBD^p(A)$, $v=1$ $\mathcal{L}^n$ a.e.\ in $A$, and 
\begin{subequations}\label{eqs:0407170004}
\begin{align}
\int \limits_A W(1,e(u))\dx \leq &\liminf_{k\to \infty} \int \limits_A W(v_k,e(u_k))\dx\,,\label{0307171748}\\
\alpha\hn(J_u)\leq &\liminf_{k\to \infty} \int \limits_A \Big(\frac{d(v_k)}{\varepsilon_k}+a\,\varepsilon_k^{q-1} |\nabla v_k|^q \Big)\dx\,.\label{0307171749}
\end{align}
\end{subequations}
It is immediate that $v_k\to 1$ in $L^1(A)$. To see \eqref{0307171748}, we show that
\begin{equation}\label{0307171836}
(v_k)^{\frac{1}{p}}e(u_k)\rightharpoonup e(u)\quad\text{ in }L^p(A;\Mnn)\,,
\end{equation}
by proving that, for every $\xi \in \Sn$ and $w\in L^p(A)$
\begin{equation}\label{0307171904}
\int \limits_A (e(u)\xi\cdot\xi-w)^p \dx \leq \liminf_{k\to\infty} \int \limits_A ((v_k)^{\frac{1}{p}} e(u_k)\xi\cdot\xi-w)^p \dx\,.
\end{equation}
This gives $(v_k)^{\frac{1}{p}}e(u_k)\xi\cdot\xi\rightharpoonup e(u)\xi\cdot\xi$ in $L^p(A)$ for every $\xi \in \Sn$, and then \eqref{0307171836} by the Polarisation Identity. 
At this stage, \eqref{0307171748} follows by the facts that $v_k\leq 1$, $v_k\to 1$ uniformly up to a set with small measure by Egorov's Theorem, and by the Ioffe-Olech semicontinuity theorem (cf.\ \cite[Theorem~2.3.1]{ButLibro}).

Thus, let us fix $\xi \in \Sn$. For simplicity, we prove \eqref{0307171904} in the case when $w=0$, the general case being obtained by approximating every $w\in L^p(A)$ by piecewise constant functions on a Lipschitz partition of $A$, for which the lower semicontinuity is then immediate.
Notice that it is not restrictive to assume that the $\liminf$ in \eqref{0307171904} is a limit.
Moreover, up to a subsequence, not relabelled, we have that for $\hn$-a.e.\ $y\in \Pi^\xi$
\begin{equation}\label{0307171925}
(\hat{u}_k)^\xi_y \to \hat{u}^\xi_y \quad\text{in measure in }A^\xi_y,\qquad (v_k)^\xi_y\to v^\xi_y\quad \text{in }L^1(A^\xi_y)\,.
\end{equation}
Indeed, a sequence $g_k$ converges to a function $g$ in measure if and only if $\arctan(g_k)$ converges to $\arctan(g)$ in $L^1$. Therefore, by Fubini's Theorem and the fact that $u_k\cdot \xi \to u\cdot \xi$ in measure in $A$, one has
\begin{equation*}
\int\limits_{\Pi^\xi}\Big(\int \limits_{A^\xi_y} |\tau((\hat{u}_k)^\xi_y)-\tau((\hat{u})^\xi_y)|(t)\,\mathrm{d}t\Big)\dh = \int \limits_A |\tau(u_k\cdot \xi)-\tau(u\cdot \xi)|\dx \to 0\,,
\end{equation*}
for $\tau=\arctan$. This gives \eqref{0307171925} for $u_k$, while the convergence for $v_k$ follows easily from Fubini's Theorem for $v_k\cdot \xi$.

It is now standard to see, as in \cite[(65)--(68)]{Iur14}, that $(\hat{u}_k)^\xi_y\in SBV^p(A^\xi_y)$ and
\begin{subequations}
\begin{align}
\int \limits_{A^\xi_y} |\nabla(\hat{u}^\xi_y)|^p\,\mathrm{d}t &\leq \liminf_{k\to \infty} \int \limits_{A^\xi_y} (v_k)^\xi_y \,|\nabla ((\hat{u}_k)^\xi_y)|^p\,\mathrm{d}t\,,\label{0307172120}\\
\alpha \mathcal{H}^0(J_{\hat{u}^\xi_y}) &\leq \liminf_{k\to \infty} \int \limits_{A^\xi_y} \bigg(\frac{d((v_k)^\xi_y)}{\varepsilon_k}+a\varepsilon_k^{q-1}\,|\nabla((v_k)^\xi_y)|^q\bigg)\,\mathrm{d}t\,. \label{0307172121}
\end{align}
\end{subequations}
Moreover, we get $u\in GSBD(A)$ and \eqref{0307171904} for $w=0$ arguing again as in the proof of \cite[Theorem~8]{Iur14}, with the exponents $2$ and $p$ therein for $u$ and $v$ replaced by $p$ and $q$. 
In particular, integrating \eqref{0307172120} over $\Pi^\xi$ gives \eqref{0307171904} for $w=0$, by \eqref{3105171927}. In the same way, one integrates \eqref{0307172121} over $\Pi^\xi$ and applies a localisation argument to deduce \eqref{0307171749}. Notice that the analogous of the Structure Theorem \cite[Theorem~4.5]{AmbCosDM97} holds also for $GSBD$, see for instance \cite[Theorem~3.1]{Fri17ARMA}.
By the discussion at the beginning of the proof, we conclude \eqref{0307171748} and the $\Gamma$-$\liminf$ inequality.

The $\Gamma$-$\limsup$ inequality follows from our density result. Indeed, for every $u\in GSBD(A)$ there exist $u_k\in SBV(A;\Rn)\cap L^\infty(A;\Rn)$ satisfying the approximation properties of Theorem~\ref{teo:main}. In particular, \[G^A(u_k,1)\to G^A(u,1)\,.\]
By a diagonalisation argument, it is then enough to construct a recovery sequence for $(u,1)$, with $u\in SBV(A;\Rn)\cap L^\infty(A;\Rn)$. This is done by the same construction as in \cite{Cha04, Iur14}, that was applied therein to a quadratic bulk energy in $e(u)$ but works also for a bulk energy with $p$-growth (for this case see the proof of the $\Gamma
$-$\limsup$ inequality in \cite{CC19b}).
\end{proof}

Let $A\subset \Rn$ be a bounded set 
and $g\colon A\to \Rn$ be a measurable function such that $\psi(|g|)\in L^1(A)$, 
for $\psi\colon [0,\infty)\to [0,\infty)$ 
increasing, 
continuous, and satisfying \eqref{hppsi} (see Theorem~\ref{teo:rough}).
For every measurable functions $u\colon A\to \Rn$ and $v\colon A\to [0,1]$ we define
\begin{equation*}
 F^A_k(u,v):=G_k^A+\int\limits_A \psi(|u-g|)\dx\,,
 \end{equation*}
and the \emph{generalised Griffith energy with fidelity term}
\begin{equation*}
F^A(u,v):=G^A(u,v)+\int\limits_A \psi(|u-g|)\dx\,,
\end{equation*}
where $F^A(u,v)=+\infty$ if $\psi(|u-g|)$ is not in $L^1(A)$.
Then we have the following convergence.
\begin{theorem}\label{teo:gammaconvF}
Let $A\subset \Rn$ be a bounded set with finite perimeter. Then 
$F_k^A$ $\Gamma$-converge to $F^A$ with respect to the topology of the convergence in measure for $u$ and $v$. 
Moreover, if $F_k^A(u_k,v_k)\leq \inf F_k^A + \gamma_k$ for any $k$, namely if $(u_k,v_k)$ is a $\gamma_k$-minimiser for $F_k^A$, with $\gamma_k \to 0$, then, up to a subsequence, $(u_k,v_k)$ converge in measure to some $(u,1)$, which is a minimiser of $F^A$, and
\begin{equation*}
F_k^A(u_k,v_k)\to F^A(u,1)\,.
\end{equation*}
\end{theorem}

\begin{proof} The $\Gamma$-$\liminf$ inequality follows by Theorem~\ref{teo:gammaconvG} (or by \eqref{eqs:0407170004}, which are the relevant properties here), and by Fatou lemma, that implies 
\[
\int \limits_A \psi(|u-g|)\dx\leq \liminf_{k\to \infty} \int \limits_A \psi(|u_k-g|)\dx\,,
\]
when $u_k\to u$ in measure in $A$.

As for the $\Gamma$-$\limsup$ inequality, let us fix $u\in GSBD^p(A)$ such that $\psi(|u-g|)\in L^1(A)$. Since $\psi(|g|)\in L^1(A)$, then $\psi(|u|)\in L^1(A)$, and by \eqref{eqs:main} there exist $u_k \in SBV(A;\Rn)\cap L^\infty(A;\Rn)$ such that \[F^A(u_k,1)\to F^A(u,1)\,.\]
Notice that we have used also \eqref{5main}, which was not necessary for the case without fidelity term.
The proof now follows as in Theorem~\ref{teo:gammaconvG}.

It lasts to prove the sequential compactness of $\gamma_k$-minimisers for $F_k^A$. This is a consequence of \cite[Corollary~7.20]{DMLibro} and of Proposition~\ref{prop:compattezza} below.
\end{proof}

\begin{remark}\label{rem:exappmin}
If $A$ is a Lipschitz domain then every $F_k^A$ admits a minimiser.
First 
we have that
\begin{equation}\label{2507171721}
\Big(\|\nabla u\|_{L^p(A;\mathbb{M}^{n\times n})}^p + \int_A \big(\psi(|u|)\wedge |u|\big) \dx \Big) + \|v\|_{W^{1,q}(A)} \leq C\,,
\end{equation}
with $C >0$ independent of $u$, $v$ such that $F_k^A(u,v)<M$, for a given $M>0$.
Indeed, \eqref{hppsi} and Korn's Inequality, that holds since $A$ is Lipschitz, imply \eqref{2507171721} with $\nabla (u-a)$ in place of $\nabla u$, for a suitable $a\colon A \to \Rn$ affine, such that $\int_A |u-a|^p \dx \leq \ol C\, F_k^A(u,v)$, for $\ol C>1$ depending on $p$, $A$, and on $c_2$ in \eqref{2607172253}. In view of \eqref{hppsi},
\begin{equation*}
\begin{split}
\int \limits_A \psi(|a|)\dx&\leq C_\psi \int \limits_A \psi(|u|)\dx + C_\psi \int \limits_A \psi(|u-a|)\dx \leq (C_\psi)^2\, F_k^A(u,v) + C_\psi\, \int \limits_A \psi(|g|) \\
&\quad+ (C_\psi)^2 \Big(|A|+ \int \limits_A |u-a|^p\dx\Big) \leq (C_\psi)^2 \ol{C}\, F_k^A(u,v) + \widetilde{c}(|A|,C_\psi,g)\,,
\end{split}
\end{equation*}
and by \cite[Lemma~2.3]{FriSol16} this gives a bound for $\nabla a$ in $A$, so that we conclude \eqref{2507171721}.
Now, the sum $\|\nabla u\|_{L^p(A;\mathbb{M}^{n\times n})} + \int_A \big(\psi(|u|)\wedge |u|\big) \dx$ is a norm on $W^{1,p}(\Omega;\Rn)$ equivalent to the standard norm, as one can verify by using Poincaré-Wirtinger inequality (we use $A$ Lipschitz also here). The existence of minimisers follows now from the Direct Method of Calculus of Variations (recall the properties of $W$ and Ioffe-Olech semicontinuity theorem, see \cite[Theorem~2.3.1]{ButLibro}).
\end{remark}

The compactness of (quasi-)minimisers for $F_k^A$ is obtained arguing similarly to \cite[Theorem~11.1]{DM13} and \cite[Proposition~1]{Iur14}.
A similar result is proven in \cite[Proposition~1]{Iur14}, assuming $\psi(s)=s^2$ and $g\in L^2(A;\Rn)$, and so  a uniform bound for displacements in $L^2(A;\Rn)$. 
\begin{proposition}\label{prop:compattezza}
Let $(u_k,v_k)$ be a sequence such that $F_k^A(u_k,v_k)$ is bounded. Then $v_k \to 1$ in $L^1(A)$ and, up to a subsequence, $u_k$ converge in measure to a suitable $u\in GSBD^p(A)$, with $\psi(|u|)\in L^1(A)$. 
\end{proposition}

\begin{proof}
The first part of the proof is similar to the beginning of \cite[Proposition~1]{Iur14}. 

It is immediate that $v_k\to 1$ in $L^1(A)$.
Let us fix $k\in\N$ and $\xi\in \Sn$. For simplicity of notation, we omit to write the dependence on $k$ and $\xi$ of the objects introduced in the following. We still write $u_k$ and $v_k$ to avoid confusion with the limit functions.
Let
\begin{align*}
\hat{A}_\lambda:=\Big\{y\in \Pi^\xi\colon &\int \limits_{A^\xi_y} \bigg((v_k)^\xi_y \,|\nabla ((\hat{u}_k)^\xi_y)|^p + \frac{d((v_k)^\xi_y)}{\varepsilon_k}+a\varepsilon_k^{q-1}\,|\nabla((v_k)^\xi_y)|^q\bigg)\,\mathrm{d}t\leq \lambda\Big\}\,,\\
& A_\lambda:=\{x \in A\colon \Pi^\xi(x)\in \hat{A}_\lambda\},\qquad B_\lambda:=A\setminus A_\lambda\,,
\end{align*}
where $\Pi^\xi(x)$ is the projection of $x$ on the plane $\Pi^\xi$.
Being $F_k(u_k,v_k)$ bounded, by Fubini's Theorem and Chebychev inequality we have
\begin{equation*}
\mathcal{L}^n(B_\lambda)\leq c  \,\frac{\mathrm{diam}(A)}{\lambda}\,.
\end{equation*}
Let $\tau_\mu(s):= -\mu\vee s \wedge \mu$,
\begin{equation*}
w_\mu^\lambda:=
\begin{cases}
\tau_\mu (u_k \cdot \xi)  &\quad\text{ in }A_\lambda\,,\\
0 &\quad\text{ in }B_\lambda\,,
\end{cases}
\end{equation*}
and let $g\colon [0,\infty)\to [0,\infty)$ be nondecreasing, continuous, subadditive, such that 
\begin{equation*}
g(0)=0\,,\qquad \liminf_{s\to 0^+}\frac{g(s)}{s}>0\,,\qquad g(s)\leq s \,\text{ for }s\in [0,\infty)\,,\qquad\lim_{s\to \infty}\frac{\psi(s)}{g(s)}=+\infty\,.
\end{equation*}
Therefore, following exactly \cite[inequality (11.8)]{DM13} we get that for every $\delta>0$ there exist $\mu_\delta >0$, $\lambda_\delta>0$ such that
\begin{equation*}
\int \limits_A g(|u_k\cdot \xi- w_{\mu_\delta}^{\lambda_\delta}|)\dx < \delta\,,
\end{equation*}
and then
\begin{equation}\label{0507171906}
\int \limits_A g(|\phi(v_k) (u_k\cdot \xi-\, w_\delta)|)\dx < \delta\,,
\end{equation}
for $w_\delta:=w_{\mu_\delta}^{\lambda_\delta}$ and $\phi(s):=\int_0^s d(s)^{1/q'}(t) \,\mathrm{d}t$, since $\phi(v_k)\leq c\, v_k\leq c$. (It is enough to redefine $\delta$ as $\tilde{c}\,\delta$, for a suitable $\tilde{c}$.)
Notice that here we use the fact that $\psi(|u_k|)$ are equibounded in $L^1(A)$, which follows since $F_k(u_k,v_k)$ are equibounded. 

Repeating the same computations done in \cite[Proposition~1]{Iur14} to get (84) therein, we obtain that for every $\delta>0$
\begin{equation}\label{0507171909}
\int \limits_\R |(\phi(v_k) w_\delta)^\xi_y(t+h) - (\phi(v_k) w_\delta)^\xi_y(t)| \,\mathrm{d}t \leq c(\delta)\,h\,. 
\end{equation}
By \eqref{0507171906} and \eqref{0507171909} we are in the hypotheses of \cite[Lemma~10.7]{DM13}, which gives that 
$\phi(v_k)\, u_k$ converge (up to a subsequence, not relabelled) to some $\tilde{u}$ pointwise $\mathcal{L}^n$-a.e.\ in $A$, or also in measure. Since $v_k\to 1$ in $L^1(A)$, we obtain that $u_k$ converge to $u:=\frac{\tilde{u}}{\phi(1)}$ in measure.
As in the proof of Theorem~\ref{teo:gammaconvG} $u\in GSBD^p(A)$, and by Fatou inequality $\psi(|u|)\in L^1(A)$.
\end{proof}

We now consider the Dirichlet problem for the brittle fracture energy.  We give some conditions on the Dirichlet part of the boundary.

Let $\Omega\subset \Rn$ be an open, bounded, 
Lipschitz domain for which \[\dom=\dod\cup \don\cup N\,,\] with $\dod$ and $\don$ relatively open, $\dod \cap \don =\emptyset$, $\mathcal{H}^{n{-}1}(N)=0$, 
$\dod \neq \emptyset$, and $\partial(\dod)=\partial(\don)$. 
Assume that $\dod$ satisfies the following condition: there exist a small $\ol \delta$ and $x_0\in \Rn$ such that for every $\delta \in (0,\ol \delta)$ 
\begin{equation}\label{0807170103}
O_{\delta,x_0}(\dod) \subset \Omega\,,
\end{equation}
where $O_{\delta,x_0}(x):=x_0+(1-\delta)(x-x_0)$.
Let us define, for $u_0\in W^{1,p}(\Rn;\Rn)$, the sets
\begin{equation*}
\begin{split}
W^{1,p}_{u_0}(\Omega;\Rn)&:=\{u\in W^{1,p}(\Omega;\Rn)\colon \mathrm{tr}_\Omega\, u =\mathrm{tr}_\Omega\, u_0 \text{ on }\dod\}\,,\\
V_k^1&:=\{v\in V_k^\Omega\colon \mathrm{tr}_\Omega\, v=1 \text{ on }\dod\}\,.
\end{split}
\end{equation*}

For a given $u_0\in W^{1,p}(\Rn;\Rn)$, the \emph{generalised Griffith energy with Dirichlet boundary condition} $u_0$ is defined for measurable functions $u\colon \Omega\to \Rn$ and $v\colon \Omega\to [0,1]$ by
\begin{equation*}
D(u,v):=G^\Omega(u,v)+\alpha\hn(\dod \cap \{\mathrm{tr}_\Omega \, u \neq \mathrm{tr}_\Omega\, u_0\})\,,
\end{equation*}
and its approximating energies by
\begin{equation*}
D_k(u,v):=
 \begin{dcases}
\int \limits_\Omega \Big(W(v,e(u))+\frac{d(v)}{\varepsilon_k}+a\,\varepsilon_k^{q-1} |\nabla v|^q \Big)\dx \quad &\text{in }W^{1,p}_{u_0}(\Omega;\Rn)\times V_k^1\,,\\
+\infty &\text{otherwise,}
\end{dcases}
\end{equation*}
namely $D_k$ is the sum of $G_k^\Omega$ and the characteristic function of $W^{1,p}_{u_0}(\Omega;\Rn)\times V_k^1$. 

Thanks to Theorem~\ref{teo:main} we can prove the following general $\Gamma$-convergence result.
\begin{theorem}\label{teo:gammaconvD}
Under the assumptions above, $D_k$ $\Gamma$-converge to $D$ with respect to the topology of the convergence in measure for $u$ and $v$.
\end{theorem}

\begin{proof}[Proof of Theorem~\ref{teo:gammaconvD}]
The $\Gamma$-$\liminf$ inequality follows by that one for $G_k^A$. Indeed, let $\widetilde{\Omega}\subset \Rn$ be open such that $\Omega\subset \widetilde{\Omega}$ and $\widetilde{\Omega}\cap \partial \Omega= \dod$, and define for each $u$ and $v$ their extensions 
\begin{equation*}
\tilde{u}:=\begin{cases}
u\quad &\text{in }\Omega\,,\\
u_0 \quad &\text{in }\widetilde{\Omega}\setminus \Omega \,,
\end{cases}
\qquad\qquad
\tilde{v}:=\begin{cases}
v\quad &\text{in }\Omega\,,\\
1\quad &\text{in }\widetilde{\Omega}\setminus \Omega\,.
\end{cases}
\end{equation*}
If $u_k$ converge in measure to some $u \in GSBD^p(\Omega)$, then $\tilde{u}_k$ converge to $\tilde{u} \in GSBD^p(\widetilde{\Omega})$.
Moreover, since $u_0\in W^{1,p}(\Omega;\Rn)$,
\[
D_k(u,v)= G_k^{\widetilde{\Omega}}(\tilde{u}, \tilde{v})-\int \limits_{\widetilde{\Omega}\setminus \Omega} W(1,e(u_0))\dx\,,\qquad D(u,v)= G^{\widetilde{\Omega}}(\tilde{u}, \tilde{v})-\int \limits_{\widetilde{\Omega}\setminus \Omega} W(1,e(u_0))\dx \,.
\]
Therefore Theorem~\ref{teo:gammaconvG} implies the $\Gamma$-$\liminf$ inequality for $D$.

We now prove the $\Gamma$-$\limsup$ inequality. Let us fix $u\in  GSBD^p(\Omega)$.
The goal is to prove that for every small $\eta>0$ (in no context with $\eta_k$) there exists $u^\eta \in SBV^p(\Omega;\Rn) \cap L^\infty(\Omega;\Rn)$ such that $u^\eta=u_0$ in the intersection of $\ol \Omega$ with a $n$-dimensional neighbourhood of $\dod$ and
\begin{equation}\label{0807170115}
D(u^\eta,1) < D(u,1) + \eta\,.
\end{equation}
Indeed, with such $u^\eta$ at hand, one may apply the standard construction for recovery sequences of Ambrosio-Tortorelli type (cf.\ for instance \cite[Theorem~9]{Iur14}), which leaves each approximating function equal to $u_0$ in a neighbourhood of $\dod$ (in the topology of $\ol \Omega$), in particular with the right boundary datum. Then the $\Gamma$-$\limsup$ inequality follows by a diagonal argument.
Thus, let us fix $\eta>0$ and construct $u^\eta$.

Since $\Sigma:=\partial(\dod)=\partial(\don)$ has null $\mathcal{H}^{n{-}1}$ measure, for any $\varepsilon>0$ (in no context with $\varepsilon_k$) there exists a $n$-dimensional neighbourhood $\widetilde{\Sigma}$ of $\Sigma$ with $\mathcal{L}^n(\widetilde{\Sigma})< \varepsilon$ and 
\begin{equation}\label{0807170023}
\hn(\dom \cap \widetilde{\Sigma}) <\varepsilon\,.
\end{equation}
We now argue as done to get \eqref{eqs:2405171202} and \eqref{eqs:0707172255} with the role of $J_u$ and $\partial\Omega$ therein played by $\don \setminus \widetilde{\Sigma}$. For any $\varepsilon$ we obtain a finite set of cubes $(Q_{h, N})_{h=1}^{h^N}$ of centers $x_{h,N}$ and sidelength $\varrho_{h,N}$, whose closures are pairwise disjoint, such that the analogous of \eqref{eqs:0707172255} hold, with the subscripts $_{h,0}$ replaced by $_{h,N}$ and $\ol h$ by $h^N$.
We introduce the rectangles 
\begin{equation*}
R_{h,N}:=\Big\{x_{h,N}+\sum_{i=1}^{n-1} y_i\, b_{h,i}^N+y_n\, \nu_{h,N} \colon y_i\in (-\varrho_{h,N},\varrho_{h,N}),\, y_n \in (-3\varepsilon \varrho_{h,N}-t, -\varepsilon \varrho_{h,N}) \Big\}\,,
\end{equation*}
\begin{equation*}
R'_{h,N}:=\Big\{x_{h,N}+\sum_{i=1}^{n-1} y_i\, b_{h,i}^N+y_n\, \nu_{h,N} \colon y_i\in (-\varrho_{h,N},\varrho_{h,N}),\, y_n \in (-\varepsilon \varrho_{h,N}, \varepsilon \varrho_{h,N}+t) \Big\}\,,
\end{equation*}
and $\widehat{R}_{h,N}:=R_{h,N} \cup R'_{h,N}$, with $t>0$ small, $\nu_{h,N}=-\nu_{\Omega}(x_{h,N})$ the generalised outer normal to $\Omega$ at $x_{h,N}$, and
$(b_{h,i}^N)_{i=1}^{n-1}$ an orthonormal basis of $(\nu_{h,N})^\perp$. Moreover, let $\widehat{u}_{h,N}\in GSBD^p(\widehat{R}_{h,N})$ be the functions provided by Lemma~\ref{le:Nitsche} for which the analogous of \eqref{eqs:2405171953} hold.
Let $\Omega_t:=\Omega+B(0,t)$ and $\widetilde{u}\in GSBD^p(\Omega_t)$ be defined by
\begin{equation*}
\widetilde{u}:=
\begin{cases}
u\quad&\text{in }\Omega\,;\\
\widehat{u}_{h,N} \quad&\text{in } \widehat{R}_{h,N}\,;\\
u_0  \quad&\text{elsewhere in } \Omega_t \,.
\end{cases}
\end{equation*}
We claim that 
\begin{equation}\label{0807170042}
G^{\Omega_t}(\widetilde{u},1)<G(u,1)+\eta\,,
\end{equation}
for $\varepsilon$ and $t$ small enough. Indeed, it is enough to observe that, for $\varepsilon$ and $t$ small enough, 
\begin{equation*}
\int \limits_{\Omega_t\setminus \Omega} |e(u_0)|^p \dx <\eta\,,\qquad \sum_h \int \limits_{\widehat{R}_{h,N}} |e(\widehat{u}_{h,N})|^p\dx \leq C \int \limits_{R_{h,N}} |e(u)|^p\dx <\eta \,,
\end{equation*}
by the absolute continuity of the integral, and
\begin{equation*}
\hn(J_{\widetilde{u}}) < \hn(J_u)  + c\,\hn\Big(J_u\cap \bigcup_{h=1}^N R_{h,N}\Big)  + \hn\Big(\don \setminus \bigcup_h Q_{h,N}\Big) < \hn(J_u) + \eta \,,
\end{equation*}
by  Lemma~\ref{le:Nitsche},  \eqref{0807170023} and the analogous of \eqref{2105171956},  arguing as in Theorem~\ref{teo:main} to get \eqref{3main}. 

Let us consider the functions $\widetilde{u}^\delta:=\widetilde{u}\circ (O_{\delta,x_0})^{-1}+ u_0 - u_0 \circ (O_{\delta,x_0})^{-1}$.
By \eqref{0807170103} and the definition of $\widetilde{u}$, $\widetilde{u}^\delta=u_0$ in a neighbourhood of $\dod$. Moreover, by \eqref{0807170042} and since for $\delta$ small $\int \limits_{\Rn} |e(u_0)-e(u_0 \circ (O_{\delta,x_0})^{-1})|^p \dx < \eta$, we have for $\delta$ small enough that
\begin{equation}\label{0807170114}
D(\widetilde{u}^\delta,1) < D(u,1) + \eta\,.
\end{equation}
We obtain $u^\eta$ by applying the construction of Theorem~\ref{teo:main} starting from a fixed $\widetilde{u}^\delta$ satisfying \eqref{0807170114}: since $u_0$ does not jump, we have that the $k$-th approximating function for $\widetilde{u}^\delta$ is $u_0 \ast \varrho_k$ in a neighbourhood of $\dod$. Then it is enough to correct it by adding $u_0 - u_0 \ast \varrho_k$, which is small in $W^{1,p}$ norm for $k$ large. Therefore, the approximation properties of Theorem~\ref{teo:main} and \eqref{0807170114} give \eqref{0807170115}. This concludes the proof.
\end{proof}

\begin{remark}
The main difficulty without the geometrical assumptions on $\dod$ of Theorem~\ref{teo:gammaconvD} is to correct the boundary datum after the composition with $(O_{\delta,x_0})^{-1}$ or after any convolution. Indeed, there could be some parts of $\dod$ which are brought outside $\Omega$ and replaced by $u$, so that the new trace on $\dod$ may differ too much from the trace of $u_0$ (the trace of $u$ on strips close to $\dod$ is not even in $W^{1-1/p,p}(\dod)$ in general), and there is an analogous problem with the convolution. In subsets of $\dod$ where the traces of $u$ and $u_0$ are different one could bring the jump a little bit inside $\Omega$, arguing as in \cite[Theorem~3.1]{CorToa99} keeping almost the same length, so almost the same energy. But as soon as there are zones where the traces of $u$ and $u_0$ coincide, one may increase very much the energy to fit the boundary condition. We refer to \cite{CarVG18} for a treatment of a Dirichlet boundary condition for $C^1$ domains, see Section~5 therein.
\end{remark}

 We now discuss the compactness for minimisers of the approximating functionals $D_k$. This follows from the compactness result \cite[Theorem~1.1]{CC18}, that we recall below for the reader's convenience. 
\begin{theorem}\label{thm:mainComp}
Let $\phi \colon \R^+ \to \R^+$ be a non-decreasing function with
\begin{equation}\label{eq:equiint}
\lim_{t\to +\infty} \frac{\phi(t)}{t}=+\infty\,,
\end{equation}
and let $(u_h)_h$ be a sequence in $GSBD(\Omega)$ such that
\begin{equation}\label{eq:boundGSBD}
\int\limits_\Omega \phi\big( |e(u_h)| \big) \dx + \hn(J_{u_h}) < M\,,
\end{equation}
for some constant $M$ independent of $h$. 
Then there exists a subsequence, still denoted by $(u_h)_h$, such that 
\begin{equation}\label{eq:defA}
A^\infty=\{x\in \Omega \colon\, |u_h(x)|\rightarrow +\infty\}
\end{equation} 
has finite perimeter, and $u\in GSBD(\Omega)$ with $u=0$ on $A^\infty$ for which
\begin{subequations}\label{1407180906}
\begin{align}
u_h \rightarrow u  \quad\quad\,&\mathcal{L}^n\text{-a.e.\ in } \Omega \sm A^\infty\,,\label{eq:convmisura}\\
e(u_h)  \rightharpoonup e(u) \quad &\text{in } L^1(\Omega\setminus A^\infty; \Mnn)\,,\label{eq:convGradSym}\\
\hn(J_u \cup \partial^* A^\infty )\leq \liminf_{h\to \infty} \,&\hn(J_{u_h})\,.\label{eq:sciSalto}
\end{align}
\end{subequations}
\end{theorem}

Starting from a minimising sequence for the Dirichlet problem, we obtain a pointwise limit that assumes the right Dirichlet boundary datum. By \eqref{eq:sciSalto}, minimisers are obtained by the functions defined as the pointwise limit of minimising sequences outside $A^\infty$, and 0, or any infinitesimal rigid motion, on $A^\infty$. 
Then we prove the following compactness result for minimisers of $D_k$: these converge pointwise to a minimiser of $D$ outside an exceptional set $A^\infty$, and there is convergence of the energy, in both the elastic and the crack part. We refer to the notation introduced above in this section.
\begin{theorem}\label{thm:compmin}
Let $(u_k,v_k)\in W^{1,p}_{u_0}(\Omega;\Rn){\times}V_k^1$ be minimisers of $D_k$
(or ``almost'' minimisers, up to an error $\zeta_k$ with $\zeta_k\to 0$).  Assume that $W$ is differentiable with respect to the first variable.  Then, for a subsequence $(u_h,v_h)$, we have that $v_h$ converges to 1 in $L^1(\Omega)$, the set $A^\infty:=\{ x\in \Omega\colon |u_h(x)| \to +\infty \}$ has finite perimeter, there exists $u\in GSBD^p(\Omega)$ minimiser of $D$ with $u=0$ in $A^\infty$, and $u_h \rightarrow u$ $\mathcal{L}^n$-a.e.\ in $\Omega\sm A^\infty$.
Moreover $\partial^*A^\infty \subset J_u$ and
\begin{subequations}\label{2501181456}
\begin{align}
\int \limits_{\Omega} W(1,e(u))\dx &= \lim_{h\to \infty} \int \limits_{\Omega} W(v_h, e(u_h))\dx\,,\\
\alpha \hn(J_u)&= \lim_{h\to\infty} \int \limits_\Omega \Big(\frac{d(v_h)}{\varepsilon_h}+a\,\varepsilon_h^{q-1} |\nabla v_h|^q \Big)\dx\,. \label{2501182309}
\end{align}
\end{subequations}
\end{theorem}

\begin{proof}
 Since $D_k(u_0,1)= C_0$, where $C_0:=\int_\Omega W(1, e(u_0)) \dx$, 
 we have that 
 $v_k \to 1$ in $L^2(\Omega)$ and that (\textit{cf.}~\cite[Theorem~4]{Cha04})
 \begin{equation*}
 C_0 \geq D_k(u_k,v_k)\geq \int_0^1 \bigg(\int_{\{v_k>s\}}  \partial_s W(s, e(u_k))\dx +\frac{\alpha}{2} \, d(s)^{1/q'}  \hn(\partial^*\{v_k>s\}) \bigg) \mathrm{d}s\,,
\end{equation*}   
employing the coarea formula (recall also $W \geq 0$) and the fact that, by Young's inequality, it holds
\begin{equation*}
\frac{d(v_k)}{\varepsilon_k}+\varepsilon_k^{q-1} |\nabla v_k|^q \geq  \frac{\alpha}{2}  d(v_k)^{1/q'} |\nabla v_k|\,. 
\end{equation*}
By Fatou's lemma, since $W$ is nondecreasing in the first variable,  $\liminf_{k}\hn(\partial^*\{v_k>s\})$ is bounded for $\mathcal{L}^1$-a.e.\ $s\in (0,1)$, so we fix $\ol s \in (0,1)$ satisfying this property and then, up to a subsequence, $\hn(\partial^*\{v_k>\ol s\})\leq C$. By the minimality of $v_k$, recalling that $d$ is decreasing, we deduce also 
\begin{equation}\label{2401181618}
\mathcal{L}^n(\{v_k< \ol s\})\leq \frac{\varepsilon_k}{d(\ol s)}\,C_0\,.
\end{equation}
Therefore the sequence $\tilde{u}_k:=u_k \, \chi_{\{v_k > \ol s\}}$ satisfies the hypotheses of Theorem~\ref{thm:mainComp}, and so there are $A^\infty=\{x\in \Omega\colon\, |\tilde{u}_k(x)|\to \infty\}$, with finite perimeter, and $u\in GSBD^p(\Omega)$ with $u$ any (fixed) infinitesimal rigid motion on $A^\infty$
such that 
$\tilde{u}_k\to u$ $\mathcal{L}^n$-a.e.\ in $\Omega\sm A^\infty$, and $e(\tilde{u}_k)\weak e(u)$ in $L^p(\Omega\sm A^\infty;\Mnn)$. In particular, employing \eqref{2401181618},
 we have that
\begin{equation}\label{2401181630}
\begin{split}
A^\infty=\{x\in \Omega\colon\, |u_k(x)|\to \infty\}\,,
\qquad u_k \to u \quad\text{$\mathcal{L}^n$-a.e.\ in } \Omega\sm A^\infty\,.
\end{split}
\end{equation}
Since now we have determined the pointwise limit of $u_k$, we can follow standard arguments, employing a slicing technique as in Theorem~\ref{teo:gammaconvG} or \cite[Theorem~8]{Iur14} 
(see also \cite[Theorem~1.1]{CC18}),
to obtain that   
\begin{subequations}\label{2510181438}
\begin{align}
\int \limits_{\Omega} W(1, e(u))\dx &\leq \liminf_{k\to \infty} \int \limits_{\Omega} W(v_k, e(u_k)) \dx\,, \label{2501181426}\\
 \alpha \Big(\hn\big(J_u \cap (\Omega\sm A^\infty)\big)+\hn(\partial^*A^\infty) \Big) &\leq \liminf_{k\to\infty} \int \limits_\Omega \Big(\frac{d(v_k)}{\varepsilon_k}+a\,\varepsilon_k^{q-1} |\nabla v_k|^q \Big)\dx\,. \label{2501181431}
\end{align}
\end{subequations}
In particular, observing $J_u \subset \big( J_u \cap (\Omega\sm A^\infty)\big) \cup\partial^*A^\infty$, we have
\begin{equation}\label{2401181636}
D(u,1)\leq \liminf_{k\to \infty} D_k(u_k,v_k)=\liminf_{k\to \infty} \min D_k\,.
\end{equation}
Since $(D_k)_k$ $\Gamma$-converges to $D$ with respect to the topology of the convergence in measure we obtain (\textit{cf.}~\cite[Proposition~7.1]{DMLibro}) that 
\begin{equation*}
\inf_{GSBD^p(\Omega)} D \geq \liminf_{k\to \infty} \min D_k=\liminf_{k\to \infty} D_k(u_k,v_k)\,.
\end{equation*}
Therefore we have that $u$ is a minimiser for $D$ 
and that, up to considering a subsequence $u_h=u_{h_k}$ of $u_k$, 
\begin{equation*}
D(u,1)=\lim_{h\to\infty}D_h(u_h,v_h)\,.
\end{equation*}
In particular the conditions \eqref{2510181438} hold as equalities on $(u_h,v_h)$, so we get that $\partial^*A\subset J_u$ and deduce \eqref{2501181456}.
\end{proof}
\begin{remark}
For any $(u_k,v_k)$ with $D_k(u_k, v_k)\leq M$, we may extract a subsequence $(u_h,v_h)$ converging pointwise to a function $u$, outside an exceptional set $A^\infty$ where $|u_h|$ converge to $+\infty$, such that the conditions \eqref{2510181438} hold. 
\end{remark}

\bigskip
\noindent {\bf Acknowledgements.}
V.\ Crismale has been supported by a public grant as part of the \emph{Investissement d'avenir} project, reference ANR-11-LABX-0056-LMH, LabEx LMH, and is currently funded by the Marie Sk\l odowska-Curie Standard European Fellowship No.\ 793018. The authors wish to thank the anonymous referees for their valuable comments.

\medskip
\noindent {\bf Conflict of interest and ethical statement.} The authors declare that they have no conflict of interest and guarantee the compliance with the Ethics Guidelines of the journal.
\bigskip

\begin{thebibliography}{10}

\bibitem{AmarDeCicco}
{\sc M.~Amar and V.~De~Cicco}, {\em A new approximation result for
  {BV}-functions}, C. R. Math. Acad. Sci. Paris, 340 (2005), pp.~735--738.

\bibitem{Amb90GSBV}
{\sc L.~Ambrosio}, {\em Existence theory for a new class of variational
  problems}, Arch. Rational Mech. Anal., 111 (1990), pp.~291--322.

\bibitem{AmbCosDM97}
{\sc L.~Ambrosio, A.~Coscia, and G.~Dal~Maso}, {\em Fine properties of
  functions with bounded deformation}, Arch. Ration. Mech. Anal., 139 (1997),
  pp.~201--238.

\bibitem{AFP}
{\sc L.~Ambrosio, N.~Fusco, and D.~Pallara}, {\em Functions of bounded
  variation and free discontinuity problems}, Oxford Mathematical Monographs,
  The Clarendon Press, Oxford University Press, New York, 2000.

\bibitem{AmbTorCPAM}
{\sc L.~Ambrosio and V.~M. Tortorelli}, {\em Approximation of functionals
  depending on jumps by elliptic functionals via {$\Gamma$}-convergence}, Comm.
  Pure Appl. Math., 43 (1990), pp.~999--1036.

\bibitem{AmbTorUMI}
\leavevmode\vrule height 2pt depth -1.6pt width 23pt, {\em On the approximation
  of free discontinuity problems}, Boll. Un. Mat. Ital. B (7), 6 (1992),
  pp.~105--123.

\bibitem{Bab15}
{\sc J.-F. Babadjian}, {\em Traces of functions of bounded deformation},
  Indiana Univ. Math. J., 64 (2015), pp.~1271--1290.

\bibitem{BabGia14}
{\sc J.-F. Babadjian and A.~Giacomini}, {\em Existence of strong solutions for
  quasi-static evolution in brittle fracture}, Ann. Sc. Norm. Super. Pisa Cl.
  Sci. (5), 13 (2014), pp.~925--974.

\bibitem{BelCosDM98}
{\sc G.~Bellettini, A.~Coscia, and G.~Dal~Maso}, {\em Compactness and lower
  semicontinuity properties in {${\rm SBD}(\Omega)$}}, Math. Z., 228 (1998),
  pp.~337--351.

\bibitem{Bou07}
{\sc B.~Bourdin}, {\em Numerical implementation of the variational formulation
  for quasi-static brittle fracture}, Interfaces Free Bound., 9 (2007),
  pp.~411--430.

\bibitem{BouFraMar00}
{\sc B.~Bourdin, G.~A. Francfort, and J.-J. Marigo}, {\em Numerical experiments
  in revisited brittle fracture}, J. Mech. Phys. Solids, 48 (2000),
  pp.~797--826.

\bibitem{BraChP96}
{\sc A.~Braides and V.~Chiad\`o~Piat}, {\em Integral representation results for
  functionals defined on {${\rm SBV}(\Omega;{\bf R}^m)$}}, J. Math. Pures Appl.
  (9), 75 (1996), pp.~595--626.

\bibitem{BCG17}
{\sc A.~Braides, S.~Conti, and A.~Garroni}, {\em Density of polyhedral
  partitions}, Calc. Var. Partial Differential Equations, 56 (2017), pp.~Art.
  28, 10.

\bibitem{BurOrtSul13}
{\sc S.~Burke, C.~Ortner, and E.~S\"uli}, {\em An adaptive finite element
  approximation of a generalized {A}mbrosio-{T}ortorelli functional}, Math.
  Models Methods Appl. Sci., 23 (2013), pp.~1663--1697.

\bibitem{ButLibro}
{\sc G.~Buttazzo}, {\em Semicontinuity, relaxation and integral representation
  in the calculus of variations}, vol.~207 of Pitman Research Notes in
  Mathematics Series, Longman Scientific \& Technical, Harlow; copublished in
  the United States with John Wiley \& Sons, Inc., New York, 1989.

\bibitem{CagToa11}
{\sc F.~Cagnetti and R.~Toader}, {\em Quasistatic crack evolution for a
  cohesive zone model with different response to loading and unloading: a
  {Y}oung measures approach}, ESAIM Control Optim. Calc. Var., 17 (2011),
  pp.~1--27.

\bibitem{CarVG18}
{\sc M.~Caroccia and N.~Van~Goethem}, {\em Damage-driven fracture with
  low-order potentials: asymptotic behavior and applications}, 2018, Preprint
  arXiv:1712.08556.

\bibitem{Cha03}
{\sc A.~Chambolle}, {\em A density result in two-dimensional linearized
  elasticity, and applications}, Arch. Ration. Mech. Anal., 167 (2003),
  pp.~211--233.

\bibitem{Cha04}
\leavevmode\vrule height 2pt depth -1.6pt width 23pt, {\em An approximation
  result for special functions with bounded deformation}, J. Math. Pures Appl.
  (9), 83 (2004), pp.~929--954.

\bibitem{Cha05Add}
\leavevmode\vrule height 2pt depth -1.6pt width 23pt, {\em Addendum to: ``{A}n
  approximation result for special functions with bounded deformation'' [{J}.
  {M}ath. {P}ures {A}ppl. (9) {\bf 83} (2004), no. 7, 929--954; mr2074682]}, J.
  Math. Pures Appl. (9), 84 (2005), pp.~137--145.

\bibitem{CCF16}
{\sc A.~Chambolle, S.~Conti, and G.~Francfort}, {\em Korn-{P}oincar\'e
  inequalities for functions with a small jump set}, Indiana Univ. Math. J., 65
  (2016), pp.~1373--1399.

\bibitem{CCF17}
{\sc A.~Chambolle, S.~Conti, and G.~A. Francfort}, {\em Approximation of a
  brittle fracture energy with a constraint of non-interpenetration}, Arch.
  Ration. Mech. Anal., 228 (2018), pp.~867--889.

\bibitem{ChaConIur17}
{\sc A.~Chambolle, S.~Conti, and F.~Iurlano}, {\em Approximation of functions
  with small jump sets and existence of strong minimizers of {G}riffith’s
  energy}, Preprint arXiv:1710.01929 (2017), Accepted for publication on J.\
  Math.\ Pures Appl.

\bibitem{CC19b}
{\sc A.~Chambolle and V.~Crismale}, {\em Phase-field approximation of some
  fracture energies of cohesive type}, In preparation.

\bibitem{CC19}
\leavevmode\vrule height 2pt depth -1.6pt width 23pt, {\em Existence of strong
  solutions to the {Dirichlet} problem for {Griffith} energy}, Preprint arXiv:
  1811.07147 (2018).

\bibitem{CC18}
\leavevmode\vrule height 2pt depth -1.6pt width 23pt, {\em A density result in
  {$GSBD^p$} with applications to the approximation of brittle fracture
  energies}, Preprint arXiv:1802.03302 (2018), to appear on J.\ Eur.\ Math.\
  Soc.\ (JEMS).

\bibitem{ConFocIur15}
{\sc S.~Conti, M.~Focardi, and F.~Iurlano}, {\em Which special functions of
  bounded deformation have bounded variation?}
\newblock Proc. Roy. Soc. Edinburgh Sect. A, (2016).

\bibitem{CFI16ARMA}
\leavevmode\vrule height 2pt depth -1.6pt width 23pt, {\em Integral
  representation for functionals defined on {$SBD^p$} in dimension two}, Arch.
  Ration. Mech. Anal., 223 (2017), pp.~1337--1374.

\bibitem{CFI17DCL}
\leavevmode\vrule height 2pt depth -1.6pt width 23pt, {\em Existence of strong
  minimizers for the {Griffith} static fracture model in dimension two}, In
  press on Ann. Inst. H. Poincaré Anal. Non Linéaire DOI
  10.1016/j.anihpc.2018.06.003.

\bibitem{CFI17Density}
\leavevmode\vrule height 2pt depth -1.6pt width 23pt, {\em Approximation of
  fracture energies with $p$-growth via piecewise affine finite elements}, In
  press on ESAIM Control Optim.\ Calc.\ Var.\, DOI 10.1051/cocv/2018021.

\bibitem{CorToa99}
{\sc G.~Cortesani and R.~Toader}, {\em A density result in {SBV} with respect
  to non-isotropic energies}, Nonlinear Anal., 38 (1999), pp.~585--604.

\bibitem{Cri19}
{\sc V.~Crismale}, {\em On the approximation of {$SBD$} functions and some
  applications}, Preprint arXiv:1806.03076 (2018).

\bibitem{CriLazOrl17}
{\sc V.~Crismale, G.~Lazzaroni, and G.~Orlando}, {\em Cohesive fracture with
  irreversibility: quasistatic evolution for a model subject to fatigue}, Math.
  Models Methods Appl. Sci., 28 (2018), pp.~1371--1412.

\bibitem{DMLibro}
{\sc G.~Dal~Maso}, {\em An introduction to {$\Gamma$}-convergence}, vol.~8 of
  Progress in Nonlinear Differential Equations and their Applications,
  Birkh\"auser Boston, Inc., Boston, MA, 1993.

\bibitem{DM13}
\leavevmode\vrule height 2pt depth -1.6pt width 23pt, {\em Generalised
  functions of bounded deformation}, J. Eur. Math. Soc. (JEMS), 15 (2013),
  pp.~1943--1997.

\bibitem{DMFraToa07}
{\sc G.~Dal~Maso, G.~A. Francfort, and R.~Toader}, {\em Quasistatic crack
  growth in nonlinear elasticity}, Arch. Ration. Mech. Anal., 176 (2005),
  pp.~165--225.

\bibitem{DMLaz10}
{\sc G.~Dal~Maso and G.~Lazzaroni}, {\em Quasistatic crack growth in finite
  elasticity with non-interpenetration}, Ann. Inst. H. Poincar\'e Anal. Non
  Lin\'eaire, 27 (2010), pp.~257--290.

\bibitem{DMToa02}
{\sc G.~Dal~Maso and R.~Toader}, {\em A model for the quasi-static growth of
  brittle fractures: existence and approximation results}, Arch. Ration. Mech.
  Anal., 162 (2002), pp.~101--135.

\bibitem{DMZan07}
{\sc G.~Dal~Maso and C.~Zanini}, {\em Quasi-static crack growth for a cohesive
  zone model with prescribed crack path}, Proc. Roy. Soc. Edinburgh Sect. A,
  137 (2007), pp.~253--279.

\bibitem{DeGioAmb88GBV}
{\sc E.~De~Giorgi and L.~Ambrosio}, {\em New functionals in the calculus of
  variations}, Atti Accad. Naz. Lincei Rend. Cl. Sci. Fis. Mat. Natur. (8), 82
  (1988), pp.~199--210 (1989).

\bibitem{DeGCarLea}
{\sc E.~De~Giorgi, M.~Carriero, and A.~Leaci}, {\em Existence theorem for a
  minimum problem with free discontinuity set}, Arch. Rational Mech. Anal., 108
  (1989), pp.~195--218.

\bibitem{DPFusPra17}
{\sc G.~de~Philippis, N.~Fusco, and A.~Pratelli}, {\em On the approximation of
  {SBV} functions}, Atti Accad. Naz. Lincei Rend. Lincei Mat. Appl., 28 (2017),
  pp.~369--413.

\bibitem{DibSer97}
{\sc F.~Dibos and E.~S\'er\'e}, {\em An approximation result for the minimizers
  of the {M}umford-{S}hah functional}, Boll. Un. Mat. Ital. A (7), 11 (1997),
  pp.~149--162.

\bibitem{Fal85}
{\sc K.~J. Falconer}, {\em The geometry of fractal sets}, vol.~85 of Cambridge
  Tracts in Mathematics, Cambridge University Press, Cambridge, 1986.

\bibitem{Fed}
{\sc H.~Federer}, {\em Geometric measure theory}, Die Grundlehren der
  mathematischen Wissenschaften, Band 153, Springer-Verlag New York Inc., New
  York, 1969.

\bibitem{FocIur14}
{\sc M.~Focardi and F.~Iurlano}, {\em Asymptotic analysis of
  {A}mbrosio-{T}ortorelli energies in linearized elasticity}, SIAM J. Math.
  Anal., 46 (2014), pp.~2936--2955.

\bibitem{FraLar03}
{\sc G.~A. Francfort and C.~J. Larsen}, {\em Existence and convergence for
  quasi-static evolution in brittle fracture}, Comm. Pure Appl. Math., 56
  (2003), pp.~1465--1500.

\bibitem{FraMar98}
{\sc G.~A. Francfort and J.-J. Marigo}, {\em Revisiting brittle fracture as an
  energy minimization problem}, J. Mech. Phys. Solids, 46 (1998),
  pp.~1319--1342.

\bibitem{Fri17ARMA}
{\sc M.~Friedrich}, {\em A derivation of linearized {G}riffith energies from
  nonlinear models}, Arch. Ration. Mech. Anal., 225 (2017), pp.~425--467.

\bibitem{FriPWKorn}
\leavevmode\vrule height 2pt depth -1.6pt width 23pt, {\em A {P}iecewise {K}orn
  {I}nequality in {SBD} and {A}pplications to {E}mbedding and {D}ensity
  {R}esults}, SIAM J. Math. Anal., 50 (2018), pp.~3842--3918.

\bibitem{Fri19}
\leavevmode\vrule height 2pt depth -1.6pt width 23pt, {\em A compactness result
  in {$GSBV^p$} and applications to {$\Gamma$-convergence} for free
  discontinuity problems}, Preprint arXiv:1807.03647 (2018).

\bibitem{FriSol16}
{\sc M.~Friedrich and F.~Solombrino}, {\em Quasistatic crack growth in
  2d-linearized elasticity}, Ann. Inst. H. Poincaré Anal. Non Linéaire, 35
  (2018), pp.~27--64.

\bibitem{Griffith}
{\sc A.~A. Griffith}, {\em The phenomena of rupture and flow in solids},
  Philos.\ Trans.\ Roy.\ Soc.\ London Ser.\ A, 221 (1920), pp.~163--198.

\bibitem{Hut}
{\sc J.~W. Hutchinson}, {\em A course on nonlinear fracture mechanics},
  Department of Solid Mechanics, Techn. University of Denmark, 1989.

\bibitem{Iur14}
{\sc F.~Iurlano}, {\em A density result for {GSBD} and its application to the
  approximation of brittle fracture energies}, Calc. Var. Partial Differential
  Equations, 51 (2014), pp.~315--342.

\bibitem{KohnTemam}
{\sc R.~Kohn and R.~Temam}, {\em Dual spaces of stresses and strains, with
  applications to {H}encky plasticity}, Appl. Math. Optim., 10 (1983),
  pp.~1--35.

\bibitem{Laz11}
{\sc G.~Lazzaroni}, {\em Quasistatic crack growth in finite elasticity with
  {L}ipschitz data}, Ann. Mat. Pura Appl. (4), 190 (2011), pp.~165--194.

\bibitem{MumSha}
{\sc D.~Mumford and J.~Shah}, {\em Boundary detection by minimizing
  functionals}.
\newblock Proc. IEEE Conf. on Computer Vision and Pattern Recognition, San
  Francisco, 1985.

\bibitem{Neg03}
{\sc M.~Negri}, {\em A finite element approximation of the {G}riffith's model
  in fracture mechanics}, Numer. Math., 95 (2003), pp.~653--687.

\bibitem{Neg06}
\leavevmode\vrule height 2pt depth -1.6pt width 23pt, {\em A non-local
  approximation of free discontinuity problems in {$SBV$} and {$SBD$}}, Calc.
  Var. Partial Differential Equations, 25 (2006), pp.~33--62.

\bibitem{Nie81}
{\sc J.~A. Nitsche}, {\em On {K}orn's second inequality}, RAIRO Anal. Num\'er.,
  15 (1981), pp.~237--248.

\bibitem{Suquet1981}
{\sc P.-M. Suquet}, {\em Sur les \'equations de la plasticit\'e: existence et
  r\'egularit\'e des solutions}, J. M\'ecanique, 20 (1981), pp.~3--39.

\bibitem{Tem}
{\sc R.~Temam}, {\em Mathematical problems in plasticity}, Gauthier-Villars,
  Paris, 1985. Translation of Problèmes mathématiques en plasticité.
  Gauthier-Villars, Paris, 1983.

\bibitem{TemStr}
{\sc R.~Temam and G.~Strang}, {\em {Duality and relaxation in the variational
  problem of plasticity}}, J. Mécanique, 19 (1980), pp.~493--527.

\bibitem{White}
{\sc B.~White}, {\em The deformation theorem for flat chains}, Acta Math., 183
  (1999), pp.~255--271.

\end{thebibliography}

\end{document}